\documentclass[10pt,leqno,a4paper]{article}

\usepackage{amsmath} 
\usepackage{amsthm}
\usepackage{amssymb}

\usepackage{yhmath} 

\usepackage{color}

\usepackage{hyperref} 
\hypersetup{
    colorlinks=true,       
    linkcolor=blue,          
    citecolor=green,        
    filecolor=magenta,      
    urlcolor=cyan           
} 

\usepackage{tikz}
\usepackage{tikz-cd}
\usetikzlibrary{matrix,arrows,decorations.pathmorphing}

\usepackage[english]{babel}
\usepackage[totalheight=22 true cm, totalwidth=14 true cm]{geometry}

\usepackage{setspace}
\usepackage{mathrsfs}

\newtheorem{thm}{Theorem}[subsection]
\newtheorem{cor}[thm]{Corollary}
\newtheorem{lemma}[thm]{Lemma}

\newtheorem{prop}[thm]{Proposition}

\newtheorem{conv}[thm]{Convention}

\newtheorem{defn}[thm]{Definition}

\theoremstyle{remark}

\theoremstyle{definition}
\newtheorem{parag}[thm]{}

\newtheorem{rmk}[thm]{Remark}

\newtheorem{exa}[thm]{Example}
\newtheorem{exas}[thm]{Examples}
\newtheorem{notation}[thm]{Notation}

\numberwithin{equation}{thm}
\def\beq{\begin{equation}}
\def\eeq{\end{equation}}
\def\beqa{\begin{equation*}}
\def\eeqa{\end{equation*}}
\def\ben{\begin{enumerate}}
\def\een{\end{enumerate}}
\def\besp{\begin{split}}
\def\eesp{\end{split}}

\def\crash#1{}
\def\N{{\mathbb N}}

\def\Z{{\mathbb Z}}
\def\0{{\mathbb O}}

\def\Q{{\mathbb Q}}

\def\D{{\mathbb D}}

\def\F{{\mathbb F}}

\def\1{{\mathbf 1}}

\def\l{\left}
\def\r{\right}
\def\[[{\l[\l[}
\def\]]{\r]\r]}

\def\discr{{\rm discr}}

\def\cl{{\rm cl}}

\def\cf{\emph{cf.}\,}
\def\ie{\emph{i.e.}\,}

\def\rhs{\emph{r.h.s.}\,}
\def\lc{\emph{loc.cit.\,}}

\def\cA{{\mathcal A}}
\def\cB{{\mathcal B}}
\def\cC{{\mathcal C}}

\def\cE{{\mathcal E}}
\def\cF{{\mathcal F}}

\def\cI{{\mathcal I}}
\def\cM{{\mathcal M}}

\def\cO{{\mathcal O}}

\def\cL{{\mathcal L}}

\def\cP{{\mathcal P}}
\def\cR{{\mathcal R}}
\def\cS{{\mathcal S}}
\def\cT{{\mathcal T}}

\def\g"{``}
\def\g'{`}

\def\sU{{\mathscr U}}

\def\wtilde{\widetilde}
\def\what{\widehat}
\def\nwhat{\wideparen} 

\def\veps{\varepsilon}

\def\dis{{\rm discr}}

\def\naive{{\rm naive}}

\def\ab{{\rm ab}}

\def\Spa{{\rm Spa\,}}

\def\Spf{{\rm Spf\,}}

\def\id{{\rm id}}
\def\Ker{{\rm Ker}}
\def\Coker{{\rm Coker}}
\def\Im{{\rm Im}}
\def\Coim{{\rm Coim}}

\def\ol{\overline}
\def\iso{\xrightarrow{\ \sim\ }}
\def\map#1{\xrightarrow{\ #1\ }}

\def\sep{{\rm sep}}

\def\limit{{\rm lim}}
\def\colimit{{\rm colim}}

\def\PROD{ \mathop{{\prod}}\limits}
\def\SUM{ \mathop{{\bigoplus}}\limits}

\def\PRODsq{ \mathop{{\prod}^{\square}}\limits}
\def\PRODsqu{ \mathop{{\prod}^{\square,{\rm u}}}\limits}
\def\PRODcan{ \mathop{{\prod}^{\can}}\limits}

\def\PRODsqcont{ \mathop{{\prod}^{\square,{\rm c}}}\limits}

\def\Ab{{{\cA}b}}

\def\Mod{{{\cM}od}}

\def\Rings{{\cR ings}}

\def\QGR[#1]{{{\cO_{#1}}\hbox{-}{\bf Mod}_{\rm qcoh}}}
\def\QGRcan[#1]{{{\cO_{#1}}\hbox{-}{\bf Mod}^\can_{\rm qcoh}}}
\def\QGRborn[#1]{{{\cO_{#1}}\hbox{-}{\bf Mod}^\born_{\rm qcoh}}}

\def\Omod[#1]{{{\cO_{#1}}\hbox{-}{\bf Mod}}}

\def\LM{{\cL\cM}}
\def\LMc{{\cL\cM}^{\rm c}}

\def\LMu{{\cL\cM}^{\rm u}}

\def\Lin{{\cL}in}

\def\Ban{{{\cB}an}}

\def\Ab{{\cA b}}
\def\TAb{{\cT\!\!\cA b}}
\def\SAb{{\cS\!\!\cA b}}
\def\CAb{{\cC\!\cA b}}
\def\TAbo{{\cT\!\cA b}^{\omega}}
\def\SAbo{{\cS\!\cA b}^{\omega}}
\def\CAbo{{\cC\!\cA b}^{\omega}}
\def\cRc{{\cR^{\rm c}}}
\def\cRu{{\cR^{\rm u}}}

\def\cRfop{{\cR^{\rm u, fop}}}
\def\cRclop{{\cR^{\rm clop}}}

\def\cRuop{{\cR^{\rm u,op}}}
\def\cRfop{{\cR^{\rm u, fop}}}
\def\cRuclop{{\cR^{\rm u,clop}}}
\def\cRouclop{{\cR^{\omega, \rm u,clop}}}

\def\cRou{{\cR^{\omega, \rm u}}}

\def\cRoufop{{\cR^{\omega, \rm u,fop}}}
\def\cSR{{\cS\cR}}
\def\cSRc{{\cS\cR^{\rm c}}}
\def\cSRu{{\cS\cR^{\rm u}}}

\def\cSRfop{{\cS\cR^{\rm u, fop}}}

\def\cSRuop{{\cS\cR^{\rm u,op}}}
\def\cCR{{\cC\cR}}

\def\cCRc{\cC\cR^{\rm c}}
\def\cCRu{{\cC\cR^{\rm u}}}

\def\cCRuclop{{\cC\cR^{\rm u,clop}}}

\def\cCRfop{{\cC\cR^{\rm u, fop}}}

\def\cCRuop{{\cC\cR^{\rm u,op}}}
\def\cCRou{{\cC\cR^{\omega, \rm u}}}
\def\cCRouclop{{\cC\cR^{\omega, \rm u,clop}}}
 
\def\cCRoufop{{\cC\cR^{\omega, \rm u,fop}}}
\def\cMc{\cM^{\rm c}}
\def\cMs{\cM^{\rm s}}

\def\cMu{\cM^{\rm u}}
\def\cMclop{\cM^{\rm clop}}
\def\cMop{\cM^{\rm op}}

\def\cMcclop{\cM^{\rm c,clop}}

\def\cMcop{\cM^{\rm c,op}}

\def\cSM{{\cS\cM}}
\def\cSMc{\cS\cM^{\rm c}}

\def\cSMu{\cS\cM^{\rm u}}

\def\cSMcclop{\cS\cM^{\rm c,clop}}

\def\cSMcop{\cS\cM^{\rm c,op}}

\def\cCM{\cC\cM}
\def\cCMo{\cC\cM^{\omega}}

\def\cCMc{\cC\cM^{\rm c}}

\def\cCMu{\cC\cM^{\rm u}}

\def\cCMcclop{\cC\cM^{\rm c,clop}}

\def\cLM{\cL\cM}
\def\cLMc{\cL\cM^{\rm c}}

\def\cLMu{\cL\cM^{\rm u}}
\def\cLMou{\cL\cM^{\omega, \un}}

\def\cLMclop{\cL\cM^{\rm clop}}
\def\cLMcclop{\cL\cM^{\rm c,clop}}
\def\cLMuclop{\cL\cM^{\rm u,clop}}
\def\cLMpscan{\cL\cM^{\rm pscan}}

\def\cLMubarrell{\cL\cM^{\rm u, barrell}}

\def\cCLMoubarrell{\cC\cL\cM^{\omega,\rm u, barrell}}
\def\cLMop{\cL\cM^{\rm op}}
\def\cLMcop{\cL\cM^{\rm c,op}}
\def\cLMuop{\cL\cM^{\rm u,op}}

\def\cSLM{\cS\cL\cM}
\def\cSLM{{\cS\cL\cM}}
\def\cSLMc{\cS\cL\cM^{\rm c}}
\def\cSLMu{\cS\cL\cM^{\rm u}}

\def\cCLM{\cC\cL\cM}
\def\cCLMc{\cC\cL\cM^{\rm c}}
\def\cCLMoc{\cC\cL\cM^{\omega,{\rm c}}}

\def\cCLMu{\cC\cL\cM^{\rm u}}

\def\cCLMcan{\cL\cM^{\rm can}}

\def\cCLMnaive{\cC\cL\cM^{\rm naive}}
\def\cCLMpscan{\cC\cL\cM^{\rm pscan}}

\def\cCLMuclop{\cC\cL\cM^{\rm u,clop}}

\def\cCLMuop{\cC\cL\cM^{\rm u,op}}

\def\cCLMoc{\cC\cL\cM^{\omega,\rm c}}
\def\cCLMou{\cC\cL\cM^{\omega,\rm u}}

\def\cCLMouclop{\cC\cL\cM^{\omega,\rm u, clop}}
\def\cCLMouop{\cC\cL\cM^{\omega,\rm u, op}}

\def\se{{\rm s}}

\def\un{{\rm u}}
\def\co{{\rm c}}
\def\ba{{\rm b}}

\def\born{{\rm born}}

\def\for{{\rm for}}

\def\can{{\rm can}}
\def\pscan{{\rm pscan}}
\def\barrell{{\rm barrell}}
\def\clop{{\rm clop}}

\def\top{{\rm top}}

\def\op{{\rm op}}

\def\Bil{{\rm Bil}}
\def\Hom{{\rm Hom}}

\def\wt{\what{\otimes}}

\author{Francesco Baldassarri 
\thanks{Universit\`{a} di Padova,
Dipartimento di Matematica, Via Trieste, 63, 35121 Padova, Italy.}
 }
 
 \title{Closed exact categories of modules \\ over generalized adic rings.\\ Part 1: The bounded case.}

\begin{document} 
\date{\today}

\maketitle 

\begin{abstract}  
We develop general foundations of topological algebra over a linearly topologized ring $k$ in a format applicable
 to  both formal schemes and analytic adic spaces. We are especially interested in determining 
 exact closed tensor categories of 
complete linearly topologized $k$-modules, 
with enough projectives or injectives. 
For $k$ a widely generalized adic ring, we 
 describe here a few examples of such categories consisting of bounded modules. The application to the construction of  stacks of quasi-coherent modules  over  formal schemes \cite[Chap. 15]{GR} will be given elsewhere. 
 \end{abstract}
\tableofcontents
\bigskip

\setcounter{section}{-1}
\begin{section}{Introduction}  
\par We develop  foundations  of the theory of commutative  rings and modules equipped with a $\Z$-linear topology. 
We  mainly have in mind  the case of a (separated and) complete 
linearly topologized base
ring $k$ (\ie for which a basis of open neighborhoods of $0$ consists of ideals) and 
 $k$-linear topologies on $k$-modules  (\ie topologies for which a basis of open neighborhoods of $0$ consists of $k$-submodules).  \par \medskip
We avoid Noetherian assumptions on rings and modules but describe a weak form  of finiteness named \emph{clop}  (from ``closure-open''). It is a slight generalization of the notion of \emph{c-adic} of \cite[Defn. 8.3.8 (iii)]{GR}: a linearly topologized ring $k$  is clop iff for $I,J$ open ideals of $k$, the closure of the product ideal $IJ$ is open in $k$.   
More stringent conditions also appear in order to obtain more familiar results.  For example, we say that 
$k$ is \emph{op} (from ``open'') if the product of open ideals is open and  \emph{fop} (from ``finite open'') if, moreover, $k$ admits a basis of finitely generated open ideals. 
Finally, as in  \cite[Chap. 15]{GR},  $k$ is \emph{$\omega$-admissible} if it admits a countable basis of open ideals.  
\par \medskip
Our framework  encompasses both
\begin{itemize}
\item the formal  setting of \cite[{\bf 0}.7]{EGA}, where all topological $k$-modules are bounded, in the sense that their topology is coarser than the one induced by  the \emph{topological} ring $k$ which we call \emph{the naive $k$-canonical topology} (see subsection~\ref{naivesec} below and
\cite[15.1.2]{GR}),   
\item  the (non-archimedean) analytic  setting, where  locally convex topological vector spaces over a (complete) non-archimedean field $K$ \cite{schneider}  are viewed as, typically unbounded, topological modules over the ring of integers $k=K^\circ$ of $K$.  So,    the description of unbounded   
$k$-linearly topologized modules generalizes functional analysis over $K$.  
\end{itemize}  
The distinction of bounded versus possibly unbounded $k$-linearly topologized $k$-modules $M$, generally discussed  in sections \ref{topstruct} and \ref{topmod},  is crucial all over this paper and its follow-up \cite{unbounded}. It corresponds to the fact that the map ``multiplication by scalars''   
$$k \times M \longrightarrow M \;\;,\;\; (\lambda, m) \longmapsto \lambda m
$$
is  required to be  uniformly continuous for the former but  just continuous for the latter. Correspondingly, our $k$-linearly topologized modules are called 
\emph{uniform} in the former case and \emph{continuous}  in the latter. The category $\cLMc_k$ of continuous 
$k$-linearly topologized modules contains the category $\cLMu_k$ of  uniform ones as a full subcategory, but colimits are very different for the two.  
\par \medskip  
The original overall motivation of this study was the search of reasonable categories of quasi-coherent sheaves on 
formal schemes  and  non-archimedean analytic spaces.  The natural expectation, motivated by Gabber's rigid-analytic counterexample reported in \cite[\S 2.1]{conrad} and its  analog on formal schemes,  was that such sought-for sheaves could not possibly be just sheaves   of algebraic structures as \cite[{\bf 0}.5.1]{EGA} seems to suggest, but should  rather carry a topological structure. 
The fundamental steps of this approach via topological algebra were carefully established by Gabber and Ramero \cite{GR} and we naturally build on their work.   
\par \medskip \emph{Aside from some generalities treated in sections 
 \ref{topstruct} and \ref{topmod}, we defer the study of relevant categories of unbounded modules  to the second part \cite{unbounded} of this work. In this paper, we concentrate   on subcategories of 
$\cLMu_k$.} 
\par \medskip 
Since abelian categories in the realm of linearly topologized modules are scarce, while
derived categories exist for very general exact categories in the sense of Quillen,  
the main ingredient in our plan was to prove that certain natural  categories of  linearly topologized modules 
are  exact. 
It turns out that the more special and simpler class of quasi-abelian categories and their derived categories, thoroughly studied by  F. Prosmans \cite{Pros} and J.-P. Schneiders \cite{schneiders}, includes many interesting categories of $k$-linearly topologized modules. For example, when $k$ has a countable basis of open ideals, the category 
of $\omega$-admissible $k$-modules in the sense of \cite[Rmk. 15.1.26 $(ii)$]{GR} and continuous $k$-linear homomorphisms, used by Gabber and Ramero and here denoted by $\cCLMou_k$, is quasi-abelian.
It also has enough injectives 
 (in the sense of \cite[Defn. 1.3.18]{schneiders}), although, as usual, these are very inexplicit. 
 It is easy to check that $\cCLMou_k$ is a symmetric monoidal category\footnote{or a \emph{tensor category}. We follow the notation of \cite{SLN900} and \cite{Mil}.} with unit object $k$ for the complete tensor product of \cite[{\bf 0}.7.7.1]{EGA} here denoted $\wt^\un_k$.  But, in  section~\ref{inthoms} we show that the adjunction formula needed to make $\cCLMou_k$ a \emph{closed} symmetric monoidal category only holds for \emph{pseudocanonical} modules, 
namely $M \in \cCLMu_k$ such that $\{\ol{IM}\}_I$, for $I$ an open ideal of $k$,  is a basis of open submodules of $M$. 
For a pseudocanonical object of $\cCLMou_k$, but not in general, ``pro-flatness $\Rightarrow$ topological flatness'' in the sense that if $N \in \cCLMou_k$  is pseudocanonical and, for any open ideal $I$ of $k$, $N/\ol{IN}$ is a flat (discrete) $k/I$-module, then the functor $M \longmapsto M \wt^\un_k N$ is exact
(see \cite[Lemma 15.1.27]{GR}). \par \smallskip 
We introduce  in section~\ref{canmodules} the main character of our play, namely the full subcategory $\cCLMcan_k$  of  $\cCLMou_k$
whose objects are quotients of  small direct sums of copies of $k$ in $\cCLMu_k$. We call them \emph{canonical} $k$-modules. They can also be characterized as the $M \in \cCLMu_k$ which are \emph{maximal} in the sense 
 that any bijective morphism $N \to M$ in $\cCLMu_k$ is an isomorphism.  
A motivating example arises when $k = K^\circ$ from balls in   $K$-Banach spaces endowed with the subspace topology and continuous $k$-linear morphisms among them. When $k$ is discrete $\cCLMcan_k = \Mod_k$, the abelian category of all (small) $k$-modules.  
We prove that the category $\cCLMcan_k$   is a  bicomplete quasi-abelian category with the  projective generator  $k$ \footnote{unfortunately, $k$ is not compact}.  Projectives of $\cCLMcan_k$ are precisely direct summands of  small direct sums of copies of $k$ in $\cCLMu_k$; they are automatically topologically flat.  
 Now,  $\cCLMcan_k$ is a   symmetric monoidal sub-category of $\cCLMu_k$. 
We prove in section~\ref{inthoms} that $\cCLMcan_k$ admits an internal \emph{Hom}, denoted $\cL^\can_k(M,N)$, for $M,N \in \cCLMcan_k$ whose underlying $k$-module is $\Hom_{\cLMu_k}(M,N)$; then $(\cCLMcan_k,\wt^\un_k,\cL^\can_k, k)$ is a closed symmetric monoidal category.
\par \medskip  
When $k = K^\circ$, for $K$ as before,   the category $\cLMc_k$  contains the category of locally convex $K$-vector spaces \cite{schneider} so that classical  nonarchimedean functional analysis is comprised in our setting. 
More generally, when $k=R_0$ is the ring of definition of an analytic Huber ring $R$ \cite[Defn. 1.1.2]{K}, \cite{huber0}, \cite{huber1}, 
the  category $\cLMc_k$ contains interesting categories of topological $R$-modules. We will dedicate to them the follow-up of this paper \cite{unbounded}  where we will explain how to generalize classical results on locally convex and especially on bornological quasi-complete spaces to this type of relative situation. 
  Examples of unbounded $k$-modules arise when $k= R_0$ is  a ring of definition of a Tate ring $R$~:  
one may then regard (unbounded) $R_0$-linearly topologized $R$-modules as modules of global sections of 
 sheaves of locally convex vector spaces on the adic space $\Spa (R,R^\circ)$ or over the formal scheme $\Spf R_0$.    
The study of tensor product  and internal \emph{Hom} initiated here in sections \ref{tensors} and  \ref{inthoms}  will be resumed  in \cite{unbounded} where the two functors  will appear in different forms  for the various categories of unbounded modules. 
 \par \smallskip
 \subsection{Acknowledgements}  
Most of the research described in this paper was conducted together with Maurizio Cailotto, who, confronted with a never-ending task, eventually declined to be a co-author. I am indebted to him for his patient listening and his well-founded criticism. \par
It is a pleasure to acknowledge the strong influence of J.-P. Schneiders through his paper \cite{schneiders} on our present work. The present paper  owes a lot to the extremely careful and useful book project \cite{GR} of O. Gabber and L. Ramero, in particular to its Chapters 8 and 15.
Similarly useful was the book of P. Schneider \cite{schneider}, whose statements and proofs  often easily translate from nonarchimedean/analytic  into formal/relative situations.  
The papers  of F. Bambozzi, O. Ben Bassat, K. Kremnitzer \cite{BBK}, \cite{BBB}, \cite{BBBK2} were also helpful in the understanding of the properties of the categories of topological modules we discuss. 
\par 
All of the previously mentioned mathematicians were also kind enough to answer  questions I kept asking them~: I am indebted  to them for their help and  their friendship.
\par
I am also  grateful 
to Tomoyuki Abe, Fumiharu Kato and Nobuo Tsuzuki  for the opportunity of lecturing on this topic in Japan in April 2019. \end{section}  
%
\begin{section}{Exact categories}  
 We recall in this section some  basic definitions and results on exact and quasi-abelian categories.  The purpose of this section is the understanding of condition \cite[(1.3.0)]{GL} and of its dual for general exact categories. 
 \begin{subsection}{Basic definitions  
 }
 \begin{parag} \label{exactpair} 
 An \emph{exact category}  is a pair $(\cC,\cE)$ consisting of an additive category $\cC$ and of a family $\cE$ of $0$-sequences in $\cC$ of the form 
 \beq \label{exseq0}
 A \map{u} B \map{v} C
 \;,
 \eeq 
 called \emph{short exact sequences} satisfying a number of requirements listed in 
 \cite[\S 2]{exact} (and equivalent to the original requirements of Quillen reported in \cite[\S 1.0]{GL}). In particular it is required that $(u,v)$ is a \emph{kernel-cokernel pair} in the sense that $u$  is  a kernel  of $v$ and $v$ is a cokernel of $u$.   A morphism $v$ (resp. $u$) of $\cC$ appearing in a short exact sequence \eqref{exseq0} is called a \emph{strict epimorphism}  (resp. a \emph{strict monomorphism}); we  often shorten these names into ``strict epi'' and ``strict mono''.  In general a morphism $f$ of $\cC$ is \emph{strict} if $f$ factors as $f = u \circ v$, with $v$ a strict epi and $u$ a strict mono. 
 \begin{rmk} \label{strictcond}
As recalled in  \cite[Rmk. 8.5]{exact}, for any exact category $\cC$ any strict morphism 
$f:X \longrightarrow Y$ of $\cC$ admits a kernel, a cokernel, an image and a coimage; moreover, the canonical morphism $\tilde{f}: \Coim (f) \to \Im (f)$ is an isomorphism. 
 \end{rmk}
 \begin{defn} 
 Given  exact categories $(\cC,\cE)$ and $(\cC',\cE')$ an additive functor $F: \cC \longrightarrow \cC'$ is \emph{exact} if it transforms any short exact sequence in $\cE$ into a short exact sequence in $\cE'$. 
 \end{defn}
  \end{parag}
 
  \begin{parag} \label{exacat} For  a maximal choice of the family $\cE$ of sequences \eqref{exseq0} in an additive category $\cC$, the procedure of checking whether $(\cC,\cE)$ is an exact category 
may be  simplified.   
\begin{prop} \label{exactcheck}   
Let $\cC$ be an additive category and let $\cE$  be  the family of all sequences \eqref{exseq0} in $\cC$ where $(u,v)$ is a kernel-cokernel pair. Then $\cE$ is an  exact structure on $\cC$ iff 
\ben
 \item  any kernel is a  strict mono (\ie a kernel coincides with its image);
 \item any cokernel is a  strict epi (\ie a cokernel coincides with its coimage);
\item the pull-back of a strict epi by any morphism in $\cC$ exists and is still a strict epi;
  \item the push-out of a strict mono by any morphism in $\cC$  exists and is still a strict mono.
 \een
 \end{prop}
 \begin{proof}  It suffices to show that axioms E1 and E1$^\op$ of  \cite[\S 2]{exact}. 
 This follows from  \cite[Prop. 2.11]{LH} (a result the authors attribute to W. Rump \cite{rump}) and its dual. 
\end{proof}
\begin{defn} \label{specex} An exact category of the type described in Proposition~\ref{exactcheck}, that is in which all kernel-cokernel pairs are exact,  will be called \emph{special}.
\end{defn}
\begin{rmk} \label{strictcondconv} If $\cC$ is a special exact category, 
the converse of Remark~\ref{strictcond} also holds. Namely, if a particular morphism $f:X \longrightarrow Y$ of $\cC$ admits a kernel, a cokernel, an image and a coimage and $\tilde{f} : \Coim (f) \to \Im(f)$ is an isomorphism, then $f$ is strict in the sense of 
subsection~\ref{exactpair}. In fact, let $\pi_f: X \to \Coim (f)$ and $\iota_f: \Im(f) \to Y$ be the canonical morphisms. We obtain a  factorization of $f$ as the composite 
$$X \map{\pi_f} \Coim (f) \map{\tilde{f}} \Im(f) \map{\iota_f}Y$$ 
so  that by axioms E1 and E1$^\op$ of  \cite[\S 2]{exact}, $f$ is a strict morphism. 
\end{rmk}
\begin{defn}  \label{quasiabdef} A \emph{quasi-abelian}  category  is an additive category in which 
\ben
\item any morphism has  kernel  and cokernel;
\item  kernels (resp. cokernels) are stable under push-out (resp. under pull-back) along arbitrary morphisms.
\een
\end{defn} 
It follows immediately from the characterization of  Proposition~\ref{exactcheck} that
 \begin{prop} \label{quasiabchar} Let $\cC$ be an additive category with kernels and cokernels and let $\cE$  be the class of all  kernel-cokernel pairs of $\cC$. Then,  $\cC$ is a quasi-abelian category iff $(\cC,\cE)$ is an exact category (necessarily special).
\end{prop} 
\begin{rmk} \label{bim=iso} 
\hfill
\ben
\item In a quasi-abelian category $\cC$ we will use the notion of exactness of Proposition~\ref{quasiabchar}. In particular, a morphism in $\cC$ is a strict epi (resp. mono) iff it is a cokernel (resp. a kernel).
\item It follows from Remarks~\ref{strictcond} and \ref{strictcondconv}  that in a quasi-abelian category a morphism $f$ is strict iff the canonical morphism $\tilde{f} : \Coim\,f \longrightarrow \Im\,f$ is an isomorphism.
\item  We recall that, in any category, a \emph{bimorphism} is a morphism which is both a monomorphism and an epimorphism. In a quasi-abelian category $\cC$ a strict bimorphism  is an isomorphism.  In fact, let $f:E \longrightarrow F$ be a strict morphism in $\cC$. By \cite[Rmk. 1.1.2 (b)]{schneiders} if $f$ is a monomorphism (resp. an epimorphism) then $f$ coincides with the morphism $\Im(f) = \Ker  \, (\Coker \,f) \longrightarrow F$ (resp. $E \longrightarrow \Coim \,f = \Coker \, (\Ker \, f)$) so that $E = \Im(f)$ (resp. $F = \Coim\, f$). If therefore $f$ is a strict bimorphism, we get that $f:E \longrightarrow F$ identifies with the  canonical morphism $\tilde{f}: \Coim\, f \longrightarrow \Im\, f$, so that it is an isomorphism.
\item  In a quasi-abelian category for a strict morphism $A \map{f} B$ with image $\Im\,f \map{\chi} B$ (resp. with coimage $A \map{\varphi} \Coim\,f$), we have $\Coker\,f = \Coker\,\chi$ (resp. $\Ker\,f = \Ker\,\varphi$). 
\item Let $\cC$ be an additive category   with cokernels and images.  Then, any morphism of  $\cC$ which is a cokernel coincides with its coimage. 
In fact, let $M \map{f} N$  be a cokernel in $\cC$ and let
$H \map{h} M$ be the kernel of $f$. We need to prove that $f = \Coker \, h$. Suppose $f$ is the cokernel of $P \map{g} M$ and let $M \map{u} Q$ be such that $u \circ h = 0$. Since $f \circ g = 0$ there exists $P \map{\ell} H$ such that $g = h \circ \ell$. Then $u \circ g = 0$ and therefore there exists $N \map{v} Q$ such that $v \circ f = u$. 
\item
Dually,  let $\cC$ be  an additive category with kernels and coimages. Then  any kernel $M \map{f} N$ in $\cC$ coincides with its image.
\een
\end{rmk} 
\begin{prop} \label{laumonax} \hfill \ben
\item The exact categories considered by Laumon in \cite[(1.3.0)]{GL} coincide with the special exact categories  with kernels and coimages. 
\item Dually, the exact categories mentioned   in \cite[Rmk. 1.3.0.1 (iii)]{GL} coincide with the special exact categories  with cokernels and images. 
\een
\end{prop}
\begin{proof} Let us prove part $\mathit 1$. Let $(\cC,\cE)$ be an exact category satisfying condition (1.3.0) of \cite{GL} and let $A \map{u} B \map{v} C$ be any kernel-cokernel pair in $\cC$. Obviously,  (1.3.0) implies that $A \map{u} B \map{v} C$ is exact, so that $(\cC,\cE)$ is a special exact category. \par
Conversely, let $(\cC,\cE)$ be a special exact category and assume $\cC$ has kernels and coimages. We check condition (1.3.0) of \cite{GL}. Let $u:E \to F$   be any morphism of $\cC$ and consider the sequence 
\beq \label{exseq}
\Ker(u) \map{i} E \map{j} \Coim (u)\;.
\eeq 
By 6 of Remark~\ref{bim=iso} above $\Ker(u) = \Ker (j)$, so that \eqref{exseq} is a kernel-cokernel pair. Hence,  it belongs to $\cE$ and condition (1.3.0) of \cite{GL} is verified. 
\par Part $\mathit 2$ holds by duality. 
\end{proof}
\begin{rmk} \label{laumonax2} 
The categories described in the second comma of Proposition~\ref{laumonax} will be most useful to us. 
\end{rmk}
We recall the general
\begin{defn} \label{leftrightex}    Let $F: \cE \longrightarrow \cF$ be a functor of finitely complete (resp. cocomplete) categories.   Then 
$F$ is  \emph{strongly left exact} (resp. \emph{strongly right exact}) if it preserves finite limits (resp. colimits).
\end{defn}
\end{parag}
  \end{subsection}
  \begin{subsection}{Projective, injective,  flat objects} 
  \begin{defn} \label{projdef} Let $(\cC,\cE)$ be an exact category and let $(\cC^\circ, \cE^\circ)$ be the opposite exact category.  
An object $P$ (resp. $I$) of $\cC$ is \emph{projective} (resp. \emph{injective}) if the functor 
$X \mapsto \Hom_\cC(P,X)$ (resp. $X \mapsto \Hom_\cC(X,I)$), $\cC \to \Ab$ (resp. ${\cC}^\circ \to \Ab$) 
transforms strict epimorphisms (resp. strict monomorphisms) into surjections.  
\end{defn} 
\begin{defn} \label{enough} We say that the exact  category $(\cC,\cE)$ has \emph{enough projectives} (resp. 
\emph{enough injectives}) if, for any object $C$ of $\cC$, there exists a strict epimorphism $P \to C$ (resp. a strict monomorphism $C \to I$) with $P$ projective (resp. with $I$ injective). 
\end{defn}
\begin{defn}\label{topflat} An object $M$ of an exact tensor category 
$(\cC,\cE, \otimes,U)$, with unit $U$, is said to be  \emph{$\otimes$-flat} 
if the functor $X \longmapsto X \otimes M$ is exact.  
\end{defn}
\end{subsection}
\end{section}

\begin{section}{Topological groups, rings, and modules} \label{topstruct}

\begin{subsection}{Topological groups} \label{topgroup0}
\par 
The discussion of this subsection appears with more detail in \cite[\S 8.2]{GR},  especially as Proposition~{8.2.13} of \lc. 

\begin{notation} 
Let $X$ be a topological space and $Y \subset X$ be any subset. 
The closure of $Y$ in $X$ is denoted by $\ol{Y}$ if no confusion can possibly arise. 
Otherwise, we denote it by $\cl_X(Y)$. 
\end{notation}
We will only deal with topological abelian groups of  non-archimedean type as in the following definition.
\begin{defn}\label{topgroup}
A  \emph{(non-archimedean) topological abelian group} is an abelian group $(G,+)$ 
equipped with a topology such that, if  $\cP(G)$ denotes the family of open subgroups of $G$, then  
 for any $g \in G$   the family 
 $$g + \cP(G) := \{g+H\}_{H \in \cP(G)}$$  
  is a fundamental system of  neighborhoods of $g$. 
\end{defn}

A topological abelian group $G = (G,\cP(G))$  is separated if and only if 
$\bigcap\limits_{H\in\cP(G)}H= \{0\}$. 
For any subset $S \subset G$ and any subgroup $H \leq G$, we set $S+H = \bigcup_{s \in S} s +H$. If $H$ is an open subgroup of $G$,
$S+H$ is then both open and closed in $G$.
\begin{rmk} \label{topgroup1} \hfill 
\ben
\item
In any topological group $G$, any open subgroup $T$ is closed. 
Moreover, if $G$ is a topological abelian group, the closure $\skew3\ol{S}$ of any subset $S \subset G$ in $G$ is 
$$ 
\skew3  \ol{S}=\bigcap_{H \in \cP(G)} (S+H) \; . 
$$ 
In particular, the closure of a subgroup $K$ of $G$ is the intersection of all open subgroups 
of $G$ which contain $K$, and it is therefore a closed subgroup of $G$. 
\item
Let $f:G\to H$ be a morphism of topological abelian groups. Then $f$  is an open 
map of topological spaces if and only if 
for any open subgroup $P$ of $G$, the image $f(P)$  is an open subgroup of $H$. 
\een
\end{rmk}
On any topological abelian group $G = (G,\cP(G))$ 
there is  a canonical uniform structure $\sU_G$ with  basis of entourages the family of 
$$
U_P := \{(x,y) \in G \times G \,|\, x-y \in P\} \;, 
$$ 
for $P \in \cP(G)$. The difference  map $G \times G \longrightarrow G$ 
is uniformly continuous for the product uniformity on $G \times G$.
\begin{defn}\label{compl-topgroup}
The group $G$ is said to be \emph{complete} if  its canonical uniform structure is separated and complete.
\end{defn} 

We recall  from Bourbaki's \emph{Topologie G\'en\'erale} \cite[III, \S3, N.5, Cor. 2 to Prop. 10]{topgen}:
\begin{lemma} \label{BourbComplLemma}
Let $(G,\cP_1)$ and $(G,\cP_2)$ be two structures of separated  topological abelian group 
on the same abelian group $G$ such that the identity map of $G$ induces a continuous map  
$(G,\cP_1) \to (G,\cP_2)$. 
Assume there is a basis of neighborhoods of $0$ in $(G,\cP_1)$   which are complete for 
the uniform structure induced  on them by $\cP_2$. 
Then $(G,\cP_1)$ is complete.  
\end{lemma}
\begin{parag} \label{incl-for}
Let us denote by $\Ab$ the category of abelian groups, 
by $\TAb$ the category of non-archimedean topological abelian groups with continuous maps of groups, 
and by $\SAb$ (resp.~$\CAb$) the full subcategory of separated 
(resp.~complete, which implies separated)  topological abelian groups. 
\begin{prop}
The additive categories categories $\TAb$, $\SAb$, $\CAb$ are bicomplete. 
\end{prop}
\begin{proof}
Omitted. 
\end{proof}
We denote by 
$\TAbo$, $\SAbo$, $\CAbo$ the full additive subcategories of $\TAb$, $\SAb$, $\CAb$, respectively, of objects 
having a countable basis of neighborhoods of $0$. 
We have canonical inclusion and forgetful functors 
\beq \label{iclff}
\CAb \longrightarrow \SAb \longrightarrow \TAb \longrightarrow \Ab 
\eeq 
and similarly for the $\omega$-decorated versions. 
\end{parag}
\begin{rmk} \label{dicsr-adj}
The forgetful functor  $(-)^\for: \TAb \to \Ab$ admits a left adjoint $\Ab\to\TAb$ 
sending an abelian group $G$ to the topological group $G^\discr$ which is 
$G$ itself endowed with the discrete topology 
$$
\Hom_{\TAb}(G^\discr,H) = \Hom_{\Ab}(G,H^\for) \;\;,\;\; \forall \, G \in \Ab \;\mbox{and}\; H \in \TAb \;.
$$ 
Since $G^\discr$ is separated and complete, the functor $(-)^\discr$ is left adjoint  
to the forgetful functors $\SAb \to \Ab$ and $\CAb \to \Ab$, as well. As a consequence, 
the forgetful functors commute with  projective limits. 

\end{rmk}

\begin{rmk} \label{compl-adj} \hfill
\ben
\item
The inclusion functor of $\CAb$ in $\TAb$ admits as left adjoint $\TAb\to\CAb$
the usual (separated) completion 
$$
\what{G} = \limit_{H \in \cP(G)} G/H  
$$
where the limit is calculated group-theoretically 
and the topology on $\what{G}$ is the weak topology of the projections 
$\what{\pi}_H : \what{G} \to G/H$, where $G/H$ is discrete. 
A fundamental system $\cP(\what{G})$ of open subgroups of $\what{G}$ is then given by 
the subgroups $\Ker(\what{\pi}_H)$ for $H \in \cP(G)$.  
\item
The canonical universal map
\beq \label{caninj} 
i=i_G: G \longrightarrow \what{G}
\eeq 
has dense image and it is injective (resp.\ bijective) 
if and only if $G$ is separated (resp.\ complete). 
Recall that, for any \emph{open} subgroup $H$,  
$\Ker(\what{\pi}_H)$ coincides with the closure of the image of $H$ in $\what{G}$, 
and can be identified with the separated completion $\what{H}$ of $H$ 
(where $H$ is endowed with the topology induced by $G$) \cite[II, \S3, N. 9, Cor. 1]{topgen}. 
\item
For any open subgroup $H$ of $G$, the canonical map \eqref{caninj} 
induces a canonical 
isomorphism (of discrete groups) $G/H\to\what{G}/\what{H}$ 
(whose inverse is induced by $\what{\pi}_H$). 
\een
\end{rmk}

\begin{rmk} \label{sub-compl-adj}
Let $G$ be a topological abelian group and let $K$ be \emph{any} subgroup of $G$. 
Let us consider now the quotient group $G/K$. 
It is a topological abelian group with basis $\cP(G/K)$ of open subgroups 
given by $(K{+}P)/ K$ 
with $P$ varying in $\cP(G)$. 
It is separated if and only if $K$ is a closed subgroup of $G$. 
The canonical projection $G\to G/K$ is a continuous, surjective and open map. 
The separated completion of $G/K$ is therefore computed by 
$$ 
\what{G/K} = \limit_{P \in \cP(G)} G/(K{+}P)  
$$ 
and the kernel of $\what{G} \to \what{G/K}$ is then the closure of $K$ in $\what{G}$, which may be identified with 
$\what{K}$ by the discussion above.
 
From the commutative diagram of canonical morphisms 
$$ 
\begin{tikzcd}[column sep=3em, row sep=3em] 
G \arrow{r}{i_G} \arrow{d}{}
& \what{G}  \arrow{d}{} \arrow{rd}{}
\\ 
G/K \arrow{r}[']{}  \arrow[bend right]{rr}{i_{G/K}}
& \what{G}/\what{K} \arrow[dashrightarrow]{r}{}
& \what{G/K} 
\end{tikzcd}
$$ 
where the dashed morphism is injective, 
we deduce that the canonical morphism $G/K\to \what{G}/\what{K}$ 
(which is injective if and only if $K$ is a closed subgroup)
induces an isomorphism  
$$\Big( \what{G} / \what{K} \Big)\what{\phantom{)}} \iso \what{G/K}
\;.
$$ 
\end{rmk} 

\begin{rmk} \label{ML-compl-adj} In general the maps 
$$  \what{G} / \what{K} \map{} \what{G/K} \;\;\mbox{and}\;\; \what{G}  \map{} \what{G/K} \;,
$$ 
which have the same  set-theoretic image, are not surjective. In other words, the quotient group $\what{G} / \what{K}$ is not in general complete in its quotient topology.  
In fact, we have an exact sequence of projective systems of abelian groups with surjective transition maps
\beq \label{ML1} \begin{split}
0 \map{} \{(K + P)/P &\cong K/ (K \cap P)\}_{P \in \cP(G)} \map{} \{G/P \}_{P \in \cP(G)}\\ & \map{} \{G/(K + P) \}_{P \in \cP(G)} \map{}0 \;.
\end{split} \eeq
If we apply $\limit_{P \in \cP(G)}$, we obtain the exact sequence 
\beq \label{ML2}
0 \map{}  \what{K} = \ol{K}  \map{} \what{G} \map{} \what{G/K} \map{} \limit^1\{(K + P)/P \}_{P \in \cP(G)} \map{} \dots 
\eeq
of abelian groups. For the vanishing of $\limit^1$ it is however necessary, in general, \cite[Prop. 13.2.2]{EGA3} that those projective systems be essentially countable. 
\end{rmk}
\begin{cor} \label{quotclop2}
Let $G$ be a complete topological abelian group which admits a countable fundamental system of open subgroups.  
\hfill \ben \item
Let $K$ be a closed subgroup of $G$.
The canonical morphism $i_{G/K}:G/K\to \what{G/K}$ is an isomorphism in $\TAb$, 
that is the quotient $G/K$ is a complete abelian topological group. 
\item A morphism $f:G \to H$ in the category $\CAb$ is a cokernel   
iff it is an open surjection.
\een
\end{cor}
\begin{proof}  The first part of the statement is proven in Remark~\ref{ML-compl-adj}.
For the second part, assume $f$ is a cokernel in $\CAb$ and let $K$ be the kernel of $f$. 
Then, by 5 of Remark~\ref{bim=iso},  $H = \Coim(f) = \what{G/K}$.
But the previous point tells us that $G/K \iso \what{G/K}$, so that  $G/K \iso H$ in $\TAb$ and $f$ is an open surjection. Conversely, if $f:G \to H$ is an open surjection 
then $G/K \iso H$ and this implies that $G/K \iso \what{G/K}$ so that  $f$ is a cokernel in $\CAb$. 
\end{proof}
\begin{rmk} \label{sep-adj}
The inclusion functor of $\SAb$ in $\TAb$ admits as left adjoint $\TAb\to\SAb$
the usual separation functor
$$
{G}^\sep =G/\ol{\{0_G\}}\cong\Coim(G\to\what{G}) \;.
$$ 
The quotient is calculated group-theoretically 
and the topology of ${G}^\sep$ is the quotient topology. 
The canonical universal map
\beq \label{cansurj} 
p=p_G: G \longrightarrow {G}^\sep
\eeq 
is continuous, surjective, open (it is a cokernel in $\TAb$)
and it is bijective if and only if $G$ is separated.
\endgraf 
The canonical morphism $i: G \to \what{G}$ 
factorizes through an injective map ${G}^\sep\to\what{G}$ 
(but, unless $G$ is complete, this map is not a kernel of   $\SAb$, 
since it has dense image). 
\end{rmk}

\begin{rmk} \label{sub-sep-adj}
Let $G$ be a topological abelian group, $K$ a subgroup of $G$ endowed with the induced topology, 
$j:K\to G$ the inclusion. 
Then $j^\sep:K^\sep\to G^\sep$ is injective and identifies the image of $K^\sep$ 
with the subgroup $K/(K\cap\ol{\{0_G\}})\cong(K+\ol{\{0_G\}})/\ol{\{0_G\}}$. 
\endgraf 
For quotients, we have a commutative diagram 
of canonical morphisms 
$$ 
\begin{tikzcd}[column sep=3em, row sep=3em] 
G \arrow{r}{p_G} \arrow{d}{}
& G^\sep  \arrow{d}{} \arrow{rd}{}
\\ 
G/K \arrow{r}[']{}  \arrow[bend right]{rr}{p_{G/K}}
& G^\sep/K^\sep \arrow[dashrightarrow]{r}{}
& (G/K)^\sep
\end{tikzcd}
$$ 
where the dashed morphism is surjective. 
The canonical morphism $G/K\to G^\sep/K^\sep = G/(K+\ol{\{0_G\}})$ 
induces an isomorphism 
$$(G^\sep/K^\sep)^\sep = G^\sep/\ol{K^\sep} \iso (G/K)^\sep = G/\ol{K}
\;.
$$ 
\endgraf 
If $K$ is closed in $G$, then $K^\sep$ is closed in $G^\sep$ 
and $G/K$ is already separated,   
so that the canonical morphisms $G/K\to G^\sep/K^\sep\to (G/K)^\sep$ are isomorphisms. 
\end{rmk}

\begin{rmk} \label{co-limits-adj}
By the existence of left adjoints, 
the formation of projective limits 
in the categories of \ref{incl-for}
commutes with the inclusions and forgetful functors \eqref{iclff}. 
In particular, the projective limits in $\CAb$, $\SAb$, $\TAb$ 
are computed as the projective limits of the underlying groups, 
endowed with the weak topology of the projection maps. 
\endgraf 
By contrast, the computation of inductive limits does not commute 
with the inclusion functors in \eqref{iclff}. 
More precisely, the inductive limit of an inductive system  in $\TAb$ 
is the inductive limit of the inductive system of underlying abelian groups 
in the category $\Ab$ endowed with the strong abelian group topology 
of the inclusion maps, while the inductive limit of inductive systems 
in $\SAb$ (resp.~$\CAb$) is the separation (resp.~the separated completion) 
of the inductive limit in $\TAb$. 
\end{rmk}
\begin{rmk} \label{baire}  Let $G$ be a complete topological abelian group which admits a countable fundamental system of open subgroups.
It is a particular case of the Birkhoff-Kakutani Theorem  that the topology of  $G$  is induced by a translation invariant metric. By Baire's theorem \cite[Chap. IX \S 5.3 Thm.~1]{topgen2} it follows that $G$ has no countable covering $G = \bigcup_{n \in \N} A_n$ by closed subsets with empty interior. 
\end{rmk}
It is possible to prove directly that
\begin{prop}\label{abgps=mod} 
The categories $\TAb$ and $\SAb$ are quasi-abelian and bicomplete, while  
the category $\CAbo$ is quasi-abelian and has enough injectives. 
\end{prop}
We omit the proof for now because
\begin{itemize}
\item the category  $\TAb$  coincides with the category $\cLMu_\Z$ of Theorem~\ref{quasiabelian} below,
\item the category $\SAb$  coincides with the category $\cSLMu_\Z$ of Theorem~\ref{quasiabelian-sep} below,
\item the category $\CAbo$  coincides with the category $\cCLMou_\Z$ of Theorem~\ref{quasiabelian-compl} below. 
\end{itemize}
So, Proposition~\ref{abgps=mod} will appear as a special case of the above mentioned theorems. 
Using the notation of quasi-abelian categories,  Corollary~\ref{quotclop2} reads
\begin{cor} \label{quotclop2bis}  A morphism $f:G \to H$ in the category $\CAbo$ is a cokernel  iff it is an open surjection.
\end{cor} 
\end{subsection} 

\begin{subsection}{Topological rings and modules} \label{topringmod}
\begin{notation}\label{prod-sep-def} 
Let $X,Y,Z$ be topological (resp.~uniform) spaces, 
and let $f:X\times Y\to Z$ be a function. 
We say that $f$ is continuous  
(resp.~uniformly continuous) 
if it is so for the product topology (resp.~product uniformity) of $X\times Y$. 
We say that $f$ is separately continuous in the first (resp.~the second) variable 
if for any $y\in Y$ the function $f(-,y):X\to Z$ 
(resp.~for any $x\in X$ the function $f(x,-):Y\to Z$) is continuous. 
We simply say that $f$ is separately continuous if it is so in both variables. 
\end{notation}

\begin{notation} \label{ring-mod-notat} \hfill
\ben \item  
A ``ring'' will always  be assumed to be commutative with 1; for a ring $R$, an ``$R$-algebra'' $A$ will be assumed to be associative and unital, but not necessarily commutative.
We denote by $\Ab$ (resp. $\Rings$) the category of abelian groups (resp. of rings). For $R$ a ring, we denote by $\Mod_R$ the category of $R$-modules.  
\item All topological rings  $R$ appearing in this paper will be (non-archimedean) topological abelian groups for the operation $+$ as in Definition~\ref{topgroup}. 
The product map 
$$ 
\mu_R:R\times R\to R 
$$ 
will be denoted by $\mu_R(a,b)=ab$ and \emph{will be assumed to be separately continuous in the two variables.}  
For $A,B\subseteq R$ we will write $AB$ for 
the additive subgroup generated by $\mu_R(A\times B)$. \emph{We will denote by $\cR$ the category of such topological rings}, where morphisms are continuous ring homomorphisms, and by  $R \longmapsto R^\ab$  the forgetful functor to $\TAb$. 
\item
A topological ring $R$ is said to be  \emph{linearly topologized}  if
it admits a fundamental system of neighborhoods of $0$ formed by ideals.  
\item For any topological ring $R$ in $\cR$, a  (non-archimedean) topological $R$-module  $M$ will be meant to be an abelian topological group $M$ as in Definition~\ref{topgroup} endowed with a structure of $R$-module. 
The map multiplication by scalars  
$$ 
\mu_M:R\times M\to M 
$$ 
will be denoted by $\mu_M(a,m)=am$  and, for any fixed $a \in R$,  \emph{will be assumed to be  continuous in the  variable $m \in M$}.  
For any $A\subseteq R$ and $P\subseteq M$ we will write $AP$ for 
the additive subgroup generated by $\mu_M(A\times P)$. \emph{We will denote by $\cM_R$ the category of such topological $R$-modules}, where morphisms are continuous $R$-linear homomorphisms, and $M \longmapsto M^\ab$  the forgetful functor to $\TAb$. 
\item
Let $R$ be a linearly topologized ring and let $M$ be 
 a topological $R$-module.  Then,  $M$ is said to be \emph{linearly topologized}  (or \emph{$R$-linearly topologized}) if its open $R$-submodules form a fundamental system of neighborhoods of $0$. We will denote by $\cLM_R$ the full subcategory of $\cM_R$ consisting of linearly topologized $R$-modules. For $M$ in $\cLM_R$ we denote by $\cP_R(M) \subset  \cP(M^\ab)$  the family of open $R$-submodules of $M$.  We usually shorten $\cP_R(R)$ into $\cP(R)$; similarly, if $R$ is understood and $M \in \cLM_R$, we  may shorten $\cP_R(M)$ into $\cP(M)$ unless this creates confusion.  We sometimes, more seriously, abuse the language also in case $R \in \cR$, $M \in \cM_R$, when the topologies are \emph{not} specified to be $R$-linear: in that case, $\cP(R)$ and $\cP(M)$ can only (and will) stand for $\cP(R^\ab)$ and $\cP(M^\ab)$, respectively. 
\een
\end{notation} 
\begin{defn}\label{bound-sponge-lintop} Let $R$ (resp. $M$) be an object of $\cR$ (resp. of $\cM_R$).
\hfill \ben \item
A subset $B$ of $M$ is \emph{bounded} if for any $P \in \cP(M^\ab)$ there exists $I_P\in\cP(R^\ab)$ 
such that $I_PB\subseteq P$.  We denote by $\cB(M)$ (resp. $\cB^c(M)$) the family of bounded (resp. and closed) subsets of $M$.
\item  
A subset $P$ of $M$ is an  $R$-\emph{sponge}  or simply a \emph{sponge} in $M$
if for any $m\in M$ there exists 
$I_m \in \cP(R)$ such that $I_m m\subseteq P$ 
(that is, $P$ ``absorbs'' any element of $M$). 
\een
\end{defn} 
\begin{rmk}\label{bounded-closure} Let $R$ (resp. $M$) be an object of $\cR$ (resp. of $\cM_R$), and let 
$B \subseteq M$ be a bounded subset.  Then the closure $\ol{B}$ of $B$ in $M$ is bounded. 
In fact, for any $a \in R$, by continuity of the map $M \to M$, $x \mapsto ax$, 
we have $a \ol{B} \subseteq \ol{aB}$. Therefore, for any
open subgroups $I$ of $R$ and $P$ of $M$ such that $I B \subseteq P$, we deduce that 
$I \ol{B} \subseteq \ol{I B} \subseteq P$. \par 
In general, the additive subgroup of $M$ generated by a bounded subset $B$  is bounded. 
If $R$ is linearly topologized   then the $R$-submodule $RB$ of $M$ generated by $B$ is also bounded. In fact, let $U$ be an open subgroup of $M$ and $I$ an open ideal of $R$ such that $IB \subset U$. Then 
$I (RB) = IB$  and is therefore contained in $U$. 
\end{rmk} 
\begin{defn} \label{bddlindef} 
Let $R$ be a linearly topologized ring and let $M$ be 
 a topological $R$-module. Then we denote by $\cB_R(M)$ (resp. $\cB^c_R(M)$) the family of bounded (resp. and closed) $R$-submodules of $M$. 
\end{defn}
\begin{prop}\label{ring-prod-prop}
Let $R$ be an object of $\cR$.  Then:  
\ben 
\item  T.F.A.E. 
\ben
\item
$\mu_R$ is continuous; 
\item $\mu_R$ is continuous at $(0,0)$; 
\item for any $I\in\cP(R)$ there exists $J_{I}\in\cP(R)$ 
such that $J_I^2\subseteq I$.
\een
\item T.F.A.E. \ben
\item $\mu_R$ is uniformly continuous;
\item  $R$ is linearly topologized;
\item $R$ is bounded.
\een
If $R$ is complete, the previous conditions are also equivalent to \hfill \ben  \setcounter{enumii}{3}
\item   the ring $R$  is the limit of a cofiltered projective system of discrete rings and surjective morphisms,  equipped 
with the weak topology of the canonical projections.
\een
\een
\end{prop}
\begin{proof} Omitted.
\end{proof} 
\begin{prop}\label{scalar-prod-prop}
Let $R$ (resp. $M$) be an object of $\cR$ (resp. of $\cM_R$). Then: 
\par \smallskip
$(1)$ T.F.A.E. \hfill \ben 
\item [(a)]
the map $\mu_M$ is separately continuous in its two variables;
\item [(b)]
the map $\mu_M$ is separately continuous in its first variable; 
\item [(c)] for any $x\in M$ and for any $U\in\cP(M)$, there exists $I_{x,U}\in\cP(R)$ such that $I_{x,U}x\subseteq U$;
\item [(d)] for any $x\in M$, $\{x\}$ is a bounded subset of $M$;
\item [(e)] any $U \in \cP(M)$ is a sponge.  
\een
\par
$(2)$ 
the map $\mu_M$ is continuous if and only if 
it is separately continuous in the two variables (that is, the equivalent conditions  $(1)$ hold)
and it is continuous at $(0,0)$ (that is, for any $U\in\cP(M)$ there exist $I_U\in\cP(R)$ 
and $V_U\in\cP(M)$ such that $I_UV_U\subseteq U$); 
\par \smallskip
$(3)$  Assume $R$ is a linearly topologized ring.  
The map $\mu_M$ is uniformly continuous 
if and only if for any $U\in\cP(M)$ there exist $I_U \in \cP_R(R)$ 
and $V_U\in\cP(M)$ such that $I_UM \subseteq U$ and $RV_U\subseteq U$, 
that is, if and only if 
$M$ is  $R$-linearly topologized and bounded. 
\par \smallskip
$(4)$ Assume $R$ is a linearly topologized ring and $M$ is a  linearly topologized $R$-module. Then
$\mu_M$ is continuous at $(0,0)$. In particular, $\mu_M$ is continuous if and only if it is separately continuous in its first variable. (This is not always the case, see Example~\ref{counterbox} below). 
\par \smallskip
$(5)$ Let $R$ (resp. $M$) be a linearly topologized ring (resp. $R$-module). Then, T.F.A.E.~: \hfill \ben
\item the map  $\mu_M$ 
is uniformly continuous;
\item  $M$ is bounded.
\een 
\end{prop}
\begin{proof} Omitted. 
\end{proof}
\begin{defn}  \label{topnaivedef} Let $R$ be a linearly topologized ring and let $N$ be any $R$-module. 
The \emph{naive canonical topology} on  $N$ is the $R$-linear topology 
with a basis of open $R$-submodules consisting of $\{I N\}_I$, 
for $I$ running over the set of open ideals of $R$. 
Endowed with this topology, $N$ is denoted by $N^\naive$.  
\end{defn} 
\begin{prop} \label{unif-can} Let $R$ (resp. $M$) be a linearly topologized ring  (resp.  $R$-module). 
 Then the following properties are equivalent: 
 \endgraf 
 $(1)$ $M$ is uniform;
  \endgraf 
 $(2)$
 $M$ in bounded;
 \endgraf 
 $(3)$  the topology of $M$ is weaker than its naive canonical topology. 
 \end{prop}
\begin{proof}  The equivalence of $(1)$ and $(2)$ has been observed before  (see $(3)$ of Proposition~\ref{scalar-prod-prop}).
 \par Assume now $M$ is uniform. 
  Then, for any $U \in \cP_R(M)$ there is a $V \in \cP_R(M)$ and an $I \in \cP(R)$ such that 
  for any $a \in R$ and $m \in M$,  
  $$
  (a+I)(m + V) \subset am + U \;.
  $$
  But this implies that $IM \subset U$, so that the topology of $M$ is weaker than the naive canonical one. 
 The converse $(3) \Rightarrow (1)$  is clear. 
\end{proof} 
 \begin{exa}\label{locconvcont} As an important case in which Proposition~\ref{scalar-prod-prop} applies, we cite locally convex $K$-vector spaces in the sense of \cite{schneider}. There, $K$ is a complete non-trivially valued non-archimedean field with ring of integers   $k = K^\circ$. A locally convex $K$-vector space $V$ is in particular a topological $k$-module endowed with a $k$-linear topology, since a basis of open neighborhoods of 0 consists of $k$-modules called ``lattices'' \cite[Chap. 1, \S2]{schneider}. It follows from (1) of Proposition~\ref{scalar-prod-prop} that   $\mu_V$ is separately continuous because, by definition,  lattices are   $k$-sponges. It then follows from (4) of \lc that $\mu_V$ is in fact continuous. This is Lemma 4.1 of \cite{schneider}. 
\end{exa}
\begin{defn} \label{defclopmod} Let $R$ (resp. $M$) be an object of $\cR$ (resp. of $\cM_R$).
\hfill \ben
\item We say that $R$ is   $\op$ (resp. $\clop$) if, for any $I,J\in\cP(R)$, 
the additive subgroup $IJ$ (resp. its closure $\ol{IJ}$) is open in $R$. 
\item  
We say that  $M$  is   
$\op$ (resp. $\clop$)
 if for any $U\in\cP(M)$ and any $I\in\cP(R)$ 
the subgroup $IU$ (resp.~$\ol{IU}$) is open in $M$. 
\een
\end{defn}
\begin{rmk}  \label{defclopmod2}
An object $R$ of $\cR$ is $\op$ (resp. $\clop$) if and only if for any $I\in\cP(R)$ 
we have that ${I^2}\in\cP(R)$ (resp.~$\ol{I^2}\in\cP(R)$). 
\end{rmk} 
We slightly generalize in the following the definitions of Bourbaki \cite[III,\S 6, N. 5,6]{topgen}. 

\begin{defn}\label{top-ring-cat}  We define the following full subcategories of $\cR$~: \ben 
\item  $\cSR$ (resp.~$\cCR$) is the full subcategory of $\cR$ consisting of  
separated (resp.~separated and complete) topological rings. 
\item 
An object $R$ of $\cR$  is \emph{continuous} (resp.  \emph{uniform})
if the map $\mu_R$ is (resp. uniformly) continuous.
We let  $\cRc$ (resp. $\cRu$)  be the full subcategory of $\cR$ 
whose objects are continuous (resp. uniform, that is, equivalently, linearly topologized).
\item We let $\cRuop$ (resp. $\cRuclop$)  be the full subcategories of $\cRu$ 
whose objects are  $\op$ (resp. 
$\clop$). 
\item We let $\cRou$ be the full subcategory of $\cRu$ of topological rings with a countable basis of open ideals.
\item We combine the previous notation as in the following examples~:
\beq \label{shorthand} \begin{aligned} \cSRc = \cSR \cap \cRc \;\;,\;\;
\cCRc &= \cCR \cap \cRc \;\;,\;\;\cCRuclop = \cCR \cap \cRuclop \;\;,\;\; \\ \cCRouclop &= \cCRuclop \cap \cRou \;\;,\;\;   \mbox{and so on \dots} 
\end{aligned}
\eeq 
\een
\end{defn} 
\begin{defn} \label{top-mod-cat1} Let $R$ be an object of $\cR$ and let $M$ be an object  of $\cM_R$. 
\hfill\ben\item  
$M$ is said to be \emph{separately continuous}  (resp.  \emph{continuous}, 
resp. \emph{uniform}) if the map $\mu_M$ is separately continuous (resp. continuous, resp. uniformly continuous) in the two variables.   
\item $M$ is said to be \emph{separated} (resp. \emph{complete}) if the underlying topological abelian group $M^\ab$ of $M$ is separated (resp. separated and complete). We denote by $\cSM_R$ (resp.~$\cCM_R$)  the full subcategory of $\cM_R$ whose objects are  
 separated (resp.~complete) topological $R$-modules.  
\een
\end{defn}
\begin{defn}\label{top-mod-cat} 
Let $R$ be an object of $\cR$.  For $\ast = \se, \co, \un, \op, \clop$,  we define the following full subcategories of $\cM_R$~: \ben
\item let 
$\cM^\ast_R$ be the full subcategory of $\cM_R$ 
whose objects are  separately continuous, continuous, uniform, $\op$, 
$\clop$, respectively;  
\item   
let
$$
\cSM^\ast_R = \cSM_R \cap \cM^\ast_R  \;\;,\;\; \cCM^\ast_R = \cCM_R \cap \cM^\ast_R \;.
$$ 
\een
Assume $R$ is linearly topologized (that is, $R$ is in $\cRu$) and  
 recall  the full subcategory $\cLM_R$  of $\cM_R$ of 
$R$-linearly topologized objects. 
We set 
 $$\cLMc_R=\cLM_R\cap\cMc_R \;\;,\;\; \cSLMc_R=\cSLM_R\cap\cMc_R
 \;\;,\;\; \cCLMc_R=\cCLM_R\cap\cMc_R \;.
  $$
By $(3)$ of Proposition~\ref{scalar-prod-prop} for 
$R$ linearly topologized
an object $M$ of $\cM_R$ is uniform if and only if it is  bounded and $R$-linearly topologized, so we have
 $$\cLMu_R= \cMu_R \;\;,\;\; \cSLMu_R= \cSMu_R
 \;\;,\;\; \cCLMu_R= \cCMu_R \;.
  $$ 
\end{defn}
\begin{rmk} For $R$ linearly topologized  the category $\cLM_R \cap \cMs_R$ (resp. $\cLM_R \cap \cMu_R$) coincides with 
$\cLMc_R$ (resp. $\cLMu_R$).  
\end{rmk}
\begin{rmk} \label{BourbCont} Topological rings in the sense of Bourbaki \lc are here called  ``continuous''. 
Similarly, a topological $R$-module in the sense of Bourbaki \lc is only defined when $R$ is continuous, 
and is here called a \emph{continuous} $R$-module. A linearly topologized ring 
is here called  ``uniform''.
If $R$ is uniform,   a \emph{uniform} $R$-module is then 
the same thing as a continuous and bounded $R$-module whose topology is $R$-linear, 
\ie is defined by a fundamental system of open $R$-submodules of $M$. 
\end{rmk} 
\begin{prop} \label{heredit} Let $R \in \cR$ (resp. $\cRc$, resp. $\cRu$) and let $S$ be a subring of $R$ equipped with the subspace topology. 
Let $M \in \cM_R$ and  $N$ be an $S$-submodule of $M$, equipped with the subspace topology. Then 
\ben 
\item 
$S \in \cR$  (resp. $\cRc$, resp. $\cRu$) ;  
\item $N \in \cM_S$;
\item if  $M \in \cMs_R$, then   $N \in \cMs_S$;
\item if  $R \in \cRc$ (resp. if $R \in \cRu$) and $M$ is a continuous (resp. uniform) $R$-module then it is also such as an $S$-module. 
\een
\end{prop}
\begin{rmk} \label{clopopen} If $R \in \cRu$ and $M \in \cM_R$ is  op (resp. clop), then 
any open $R$-submodule of $M$ is op (resp. clop).
\end{rmk}

\begin{notation} \label{doubleindex} As we set out to do  for categories of topological rings (see  Definition~\ref{top-mod-cat}), we will also 
decorate categories of topological  modules with multiple superscripts to recall the properties of their objects.   
 If $R$ is in $\cRuop$ (resp. in $\cRuclop$), we also set 
$\cLMop_R=\cLM_R\cap\cMop_R$, $\cLMuop_R=\cLMu_R\cap\cMop_R$, $\cLMcop_R=\cLMc_R\cap\cMop_R$ (resp. $\cLMclop_R = \cLM_R\cap\cMclop_R$, 
$\cLMuclop_R = \cLMu_R\cap\cMclop_R$, $\cLMcclop_R=\cLMc_R\cap\cMclop_R$),
and similarly in  the separated or separated and complete  case. 
By $\cRfop$, $\cSRfop$, $\cCRfop$ we mean the full subcategories of $\cRuop$, $\cSRuop$, $\cCRuop$, 
respectively, of topological rings such that any open ideal contains an open finitely generated ideal.
\endgraf 
For any of the previous categories of topological rings or modules, we will use the superscript $\omega$ to indicate  that the objects of 
that category admit a countable basis of open neighborhoods of $0$. 
\end{notation}
\begin{rmk} \label{baire1}
According to Remark~\ref{baire}, for any $R \in \cCRu$, any object of $\cCMo_R$
 is a Baire space. 
\end{rmk}
\begin{rmk}\label{uclop}  For $k$ in $\cRuop$ (resp. $\cRuclop$)
an object  $M$ of $\cLMu_k$ is in $\cLMuop_k$ (resp. $\cLMuclop_k$)
if and only if $M$ 
 admits a basis of open $k$-submodules  of the form $IM$ (resp.  $\ol{IM}$) for $I \in \cP(k)$.
\end{rmk}
\begin{exa} \label{discrete-ring} 
Assume the topology of the ring $k$ is discrete.
Then $\cM_k = \cMc_k$ and $\cLM_k = \cLMc_k = \cLMu_k$. 
Of course, the topology of an object in 
any of the previous categories is not necessarily the discrete one.  
However, any $\clop$ separated topological $k$-module is discrete.
\end{exa} 

\begin{rmk} \label{discrete-mod} 
For an object $k$ of $\cRu$, 
let $M$ be an object of $\cLM_k$ whose topology is discrete. 
Then,  
$M$ is an object of $\cLMc_k$ if and only if 
$$ 
 M = \bigcup_{I \in \cP_k(k)} M_{[I]} 
$$
where  
$$ 
M_{[I]} := \{m \in M \,|\, am = 0\,,\,\forall a \in I\,\}\;.
$$ 
In fact the previous condition is equivalent to the continuity of the map $\mu_M$ in the first variable, 
\ie of any map $\mu_M(-,m) : k \to M$, for $m \in M$.  
Notice that $M_{[I]}$, but not $M$ in general,  is an object of $\cLMu_k$. 
In particular, we obtain:
\par
\emph{An object of $\cLMc_k$ which carries the  discrete topology  is uniform if and only if it is an object of  $\cLM_{k/I}$, for some open ideal $I$ of $k$.}
\end{rmk}
\end{subsection}
\begin{subsection}{Boundedness}  \label{boundedness}
\begin{prop} \label{boundedness1}
 Let  $\{M_\alpha\}_{\alpha \in A}$ be a projective system in  $\cM_k$. 
A subset $B$ of the limit  $M = \limit_\alpha M_\alpha$ is bounded 
 in the sense of Definition~\ref{bound-sponge-lintop} if and only if the  projection of $B$ 
in $M_\alpha$ is bounded for any $\alpha \in A$.  
\end{prop}
\begin{proof} In fact, the topology of $M$ is the weak topology of the family of canonical 
projections $\pi_\alpha :M \to M_\alpha$. 
Let $B \subset M$. If $\pi_\alpha (B)$ is bounded in $M_\alpha$  for any $\alpha$, and if 
$P =\bigcap_{i=1}^r \pi_{\alpha_i}^{-1}(P_i) \in \cP(M)$, where $P_i \in \cP(M_{\alpha_i})$ for $i=1,\dots,r$, there exists $J \in \cP(k)$ such that 
 $J \pi_{\alpha_i} (B) \subset P_i$,  for $i=1,\dots,r$. So, $JB \subset \pi_{\alpha_i}^{-1}(P_i)$, for $i=1,\dots,r$, hence $JB \subset P$.  Therefore, $B$ is bounded in $M$. 
 \end{proof}
 \begin{rmk}  \label{boundedclosedrmk1}  The category $\cMs_k$ (resp. $\cLMc_k$) is the full subcategory of 
$\cM_k$ (resp. $\cLM_k$)  consisting of the objects $M$ such that $\cB(M)$ is  a (set-theoretic) covering of $M$ 
(\cf   Proposition~\ref{scalar-prod-prop}).
\end{rmk}
\begin{rmk} \label{bounded-trivial} 
Let $M$ be an object of $\cLM_k$.
If $M$ carries  the discrete topology,   a $k$-submodule $B$ of $M$ is bounded if and only if 
there exists $I\in\cP(k)$ such that $IB=(0)$, 
that is, if and only if $B$ is a $k/I$-module for some $I\in\cP(k)$. 
Equivalently, a $k$-submodule $B$ of $M$ is bounded if and only if 
$B \subset M_{[I]}$ for some $I \in \cP(k)$ (see Remark~\ref{discrete-mod} for  notation). 
So, when $M$ is discrete,  $\{ M_{[I]}\}_{I \in \cP(k)}$ is a filter basis for $\cB(M)$. 
\endgraf 
In the general case, let $M$ be an object of $\cLM_k$.
A $k$-submodule $B$ of $M$ is bounded if and only if 
for every $P\in\cP_k(M)$ the image of $B$ in $M/P$ is bounded, 
that is for every $P\in\cP_k(M)$ the image of $B$ in $M/P$ is a $k/I$-module 
for some $I\in\cP(k)$. 
\end{rmk}  
 \end{subsection}  
 \begin{subsection}{Separation and separated completion}
\begin{prop}  \label{stab-sep-ring}
The construction of Remark~\ref{sep-adj} induces  a functor, called \emph{separation}, 
$$
\cRc \longrightarrow \cSRc \quad,\quad R \longmapsto  R^\sep
$$ 
left adjoint 
to the canonical inclusion of $\cSRc$ in $\cRc$. 
Separation transforms uniform (resp. continuous $\op$, resp. continuous $\clop$) topological rings into 
uniform (resp. continuous $\op$, resp. continuous $\clop$) separated topological rings.
Separation functors 
are left adjoint to the natural inverse inclusions of categories. 
\end{prop}  
\begin{proof}   
For any $R$ in $\cRc$, $\ol{\{0_R\}}$ is a closed ideal of $R$ \cite[III, \S 6, N. 4, Prop. 5]{topgen}. So
$$ 
R^\sep := R/\ol{\{0_R\}}
$$
is a ring equipped with a separated topology. 
It is clear that the product map of $R^\sep$ is  continuous. 
The remaining assertions follow.
\end{proof}

\begin{prop} \label{stab-sep-mod}
Let $R \in \cRc$. The canonical inclusion of categories $\cSMc_{R^\sep} \subset \cMc_R$ admits a left adjoint
$$
\cMc_R \longrightarrow \cSMc_{R^\sep}\;\;,\;\; M \longmapsto M^\sep = M/\ol{\{0_M\}} 
$$
called \emph{separation}.
Separation induces functors 
$$
\cMu_R \longrightarrow \cSMu_{R^\sep}\;\;,\;\;\cMcop_R \longrightarrow \cSMcop_{R^\sep}\;\;,\;\;\cMcclop_R \longrightarrow \cSMcclop_{R^\sep}\;\;,\;\;
$$
and, for $R \in \cRu$,
$$
\cLMc_R \longrightarrow \cSLMc_{R^\sep}\;\;,\;\;\cLMu_R \longrightarrow \cSLMu_{R^\sep}\;\;,
$$ 
all left adjoints to  the natural inverse  inclusions of categories. 
\end{prop} 
\begin{proof}  
For continuous modules see the proof of 
\cite[III, \S 6, N. 5, Thm. 1]{topgen}. The remaining assertions are easy.
\end{proof}
\begin{rmk} 
\label{extsep} For $R \in \cRc$ the adjoint functors $\cSMc_{R^\sep} \hookrightarrow \cMc_R$ and 
$\cMc_R \longrightarrow \cSMc_{R^\sep}\;$, $M \longmapsto M^\sep\;$, establish an equivalence of categories between 
$\cSMc_{R^\sep}$ and $\cSMc_R$. 
Because of this, when dealing with objects of $\cSMc_R$  with  $R \in \cRc$, we may as well assume that 
$R$ is separated. 
\end{rmk} 
\begin{prop} \label{stab-compl-ring}
Separated completion gives a functor 
$$\cRc \longrightarrow \cCRc\;\;,\;\; R \longmapsto \what{R}  
$$ 
 left adjoint 
to the natural full  inclusion of $\cCRc$ in $\cRc$. 
Separated completion transforms uniform (resp. continuous $\clop$) topological rings into 
uniform (resp. continuous $\clop$) complete topological rings.
Separated completion functors  
are left adjoint to the natural inverse inclusions of categories.
\end{prop} 
\begin{proof} See  \cite[III, \S 6, N. 6]{topgen} for the existence and the adjointness 
property of separated completion $\cRc \longrightarrow \cCRc$. \par
Suppose now $R$ is in $\cRu$ (which is equivalent to $R$ being linearly topologized). 
For any ideal $I$ of $R$, the separated completion $\what{I}$ identifies 
with the closure of the image of $I$  in $\what{R}$ and this is an ideal of 
$\what{R}$. 
A basis for the topology of $\what{R}$ is given by the family of $\what{I}$, for $I \in \cP(R)$. 
If $R$ is continuous and clop, and $I,J \in \cP(R)$, 
then $\cl_{\what{R}} (IJ) = \cl_{\what{R}} (\cl_R(IJ))$ is open in $\what{R}$. 
On the other hand, 
$$
\cl_{\what{R}} (IJ) \subset \cl_{\what{R}}(\cl_{\what{R}}(I) \cl_{\what{R}}(J)) = \cl_{\what{R}}(\what{I} \what{J})\;,
$$
so the latter is open in $\what{R}$, as well. Therefore, $\what{R}$ is $\clop$.  
\end{proof} 
\begin{prop} \label{stab-compl-mod} Let $R \in \cRc$. The canonical inclusion of categories $\cCMc_{\what{R}} \subset \cMc_R$ admits a left adjoint
$$
\cMc_R \longrightarrow \cCMc_{\what{R}}\;\;,\;\; M \longmapsto \what{M} = \what{M^\sep} 
$$
called (\emph{separated}) \emph{completion}.
Separated completion induces functors 
$$
\cMu_R \longrightarrow \cCMu_{\what{R}}\;\;,\;\;\cMcclop_R \longrightarrow \cCMcclop_{\what{R}}\;\;,\;\;
$$
and, for $R \in \cRu$, 
$$
\cLMc_R \longrightarrow \cCLMc_{\what{R}}\;\;,\;\;\cLMu_R \longrightarrow \cCLMu_{\what{R}}\;\;,
$$ 
all left adjoints to  the natural inverse  inclusions of categories. 
\end{prop}
\begin{proof}  Omitted.
\end{proof} 
\begin{rmk} 
\label{extcompl} For $R \in \cRc$ the adjoint functors $\cCMc_{\what{R}} \hookrightarrow \cMc_R$ and  
$\cMc_R \longrightarrow \cCMc_{\what{R}}\;$, $M \longmapsto \what{M}\;$,  establish an equivalence of categories 
between 
$\cCMc_{\what{R}}$ and $\cCMc_R$.  
Because of this, when dealing with objects of $\cCMc_R$  with  $R \in \cRc$, we may as well assume that 
$R$ is (separated and) complete. 
\end{rmk}
\begin{rmk} \label{clop-rmk} \hfill \ben
\item For a continuous and $\clop$  topological ring $R$ (resp.~and a continuous and $\clop$ topological $R$-module $M$), 
$\what{R}$  (resp. and $\what{M}$) is (resp. are) continuous and $\clop$. But notice that, for a continuous and $\op$  topological ring $R$ (resp.~and a continuous and $\op$ topological $R$-module $M$), 
$\what{R}$  (resp. and $\what{M}$) are $\clop$ but not necessarily $\op$.
This is the main reason for introducing the $\clop$ condition. 
\item Notice however that, if the uniform and $\op$ topological ring $R$ admits a countable basis 
of open ideals, then, for any open finitely generated ideal $J$ of $R$, the ideal 
$\what{J} = \cl_{\what{R}}(J\what{R})$ of $\what{R}$  equals $J \what{R}$ (see \cite[Rk. 8.3.3 (iv)]{GR} or Lemma~10.96.3 of  \cite[Tag 05GG]{stacks}). 
In particular, if  
$R$ admits a countable basis 
of open finitely generated  ideals, \ie if $R$ is an object of $\cRoufop$, $\what{R}$
is an object of $\cCRoufop$.   
\een
\end{rmk} 
\begin{exas} \label{clop-exa}  
An example of an object of $\cRu$ not in $\cRclop$ is the following.
Let $R=\Z_p[\veps]$ with $\veps \neq 0$ but $\veps^2 = 0$, with 
the linear topology determined by the fundamental system of open ideals  $\{p^n \veps R =p^n \veps \Z_p\}_{n\geq 0}$. 
Then  
$\Z_p$ is closed 
and the topology induced by $R$ on it is the discrete.
Here $(\veps R)^2 = (0)$ is closed
but  not open, because it does not contain any  ideal in the previous fundamental system. Notice that $R$ is complete.
\endgraf 
An object of $\cRouclop$ which is not $\op$ 
is obtained as follows: 
let  $R=\Z[X_i\,;\, i\in\N]$ endowed with the topology generated 
by the fundamental system of open ideals $I_{j}=(X_i: i\geq j)$, for $j\in\N$. 
The product $I_{j_1}I_{j_2}$ is not open 
(because it does not contain any basic open ideal).
Its closure $\ol{I_{j_1}I_{j_2}}$ is the intersection of  the open ideals which 
contain $I_{j_1}I_{j_2}$,   that is $I_{\max(j_1,j_2)}$,  so that it is open. \ 
Taking the completion we have an example of an object of $\cCRouclop$
which is not $\op$. Another example of a ring $R \in \cCRouclop$ which is not op will be given in part 3 of Remark~\ref{naive-pbl}.
\end{exas}

\begin{notation} \label{forget0} \hfill \ben
\item
We denote by $k^\for$ the ring underlying the topological ring $k$, 
and by $M^\for$ the $k^\for$-module underlying a topological module $M$. 
To avoid excessively burdening the notation however, 
the category $\Mod_{k^\for}$ will be simply denoted by $\Mod_k$. 
\item
Similarly, we generally 
write $\Hom_k$ for $\Hom_{k^\for}$, $\Bil_k$ (standing for ``$k$-bilinear'') 
for $\Bil_{k^\for}$, and shorten $M^\for \otimes_{k^\for} N^\for$ into $M \otimes_k N$  
(topological tensor products will have a different notation, anyhow.) 
\item
Terms like ``surjective'', ``injective'', ``bijective'' (only rarely qualified by ``set-theore\-tically'') 
applied to a morphism $f$ in $\cM_k$ refer to set-theoretic  properties of the morphism $f^\for$. 
\item For any topological ring $R$, the category $\cM_R$ admits a canonical 
faithful  functor $X \mapsto X^\top$  to the category of topological spaces. Then, a morphism $f:X \to Y$ of $\cM_R$ is closed, resp. open, resp. dominant, if so is the continuous map $f^\top: X^\top \to Y^\top$. Similarly, the term ``topological embedding'' or just ``embedding'' refers to a morphism  $i: Y \to X$ in  $\cM_R$ such that $i^\top$ is
the inclusion of a subspace $Y^\top$ of $X^\top$.  In particular  we have the notion of ``closed'' (resp. ``open'') (topological) embedding in $\cM_R$. 
\par 
Notice that the kernel  of a morphism $f:X \to Y$ in $\cM_R$   is an embedding, but it is not necessarily closed  unless $Y$ is separated.
Similarly, the cokernel of $f:X \to Y$ in  $\cM_R$ or in $\cLMc_R$  is a quotient map, but if $f$ is a morphism of  $\cCLMc_R$, then 
its cokernel in  $\cCLMc_R$ is not always surjective.  
 We will later deal with full additive subcategories $\cC$ of $\cCLMc_R$ such that the kernel of  a morphism $f:X \to Y$ of $\cC$, taken in $\cC$,  is not necessarily an embedding. See Remark~\ref{limcan} for the  example of $\cCLMcan_k \subset \cCLMu_k$.
 \een
\end{notation} 

\end{subsection} 

\begin{subsection}{Uniform and clop rings}  
\begin{defn} \label{clop-adic}
A morphism  $\phi:R \to S$ of $\cRu$  is said to be \emph{$\op$-adic} (resp. 
\emph{$\clop$-adic}) if  
for any $I \in\cP_R(R)$ one has $\phi(I)S \in \cP_S(S)$ (resp. $\ol{\phi(I)S} \in \cP_S(S)$). 
\end{defn}
Obviously, a composition of $\op$-adic (resp. $\clop$-adic) morphisms is $\op$-adic (resp. $\clop$-adic).
\begin{exa} \label{exa-clop}   For any linearly topologized ring $R$,  
the canonical map $R\to\what{R}$ is a $\clop$-adic morphism.  The same map is $\op$-adic when 
$R$ is an object of $\cRoufop$ \cite[Prop. 8.3.3 (iv)]{GR} or Lemma~10.96.3 of \cite[Tag 05GG]{stacks}. 
\end{exa}

\begin{rmk} \label{op-canonical} Let $\phi:R\to S$ be any $\op$-adic (resp.  $\clop$-adic) morphism in $\cRu$. Then~: \hfill \ben  
\item A basis of open ideals of 
  $S$ consists  of  ideals of the form 
$\phi(I)S$ (resp. $\ol{\phi(I)S}$) for $I\in\cP(R)$. 
In fact the latter ideals are open by the $\op$-adic (resp.  $\clop$-adic) property of $\phi$ 
and, on the other hand, any open ideal $J\in\cP(S)$ 
contains an ideal of this form with $I=\phi^{-1}(J)$. Therefore our notion of $\clop$-adic (resp. 
$\op$-adic) morphism of rings coincides with the notion of $c$-adic (resp. adic) morphism appearing in \cite[Defn. 8.3.23]{GR}.
\item 
Suppose moreover that $R$ is an object of $\cRuop$ (resp. of $\cRuclop$). 
Then  $S$ is one, as well. 
In fact, if $J,J'\in\cP(S)$, then $J\supseteq\phi(I)$ and $J'\supseteq\phi(I')$ 
for some $I,I'\in\cP(R)$, so $JJ'\supseteq \phi(I)\phi(I')=\phi(II')$ 
(resp. $\ol{JJ'}\supseteq \ol{\phi(I)\phi(I')}\supseteq\phi(\ol{II'})$)
and since $II'\in\cP(R)$ (resp. $\ol{II'}\in\cP(R)$) 
we conclude that $JJ'$ (resp. $\ol{JJ'}$) is an open ideal of $S$. 
\een
\end{rmk}
\begin{rmk}\label{indlim-clop} 
Colimits in $\cRu$ 
of   inductive systems  in $\cRuop$ (resp. $\cRuclop$) are in $\cRuop$ (resp.~$\cRuclop$). Similarly for 
colimits in $\cCRu$ 
of   inductive systems  in  $\cCRuclop$.
As an example, let us prove that $\cRu$-colimits of inductive systems 
of $\clop$ linearly topologized rings 
are  $\clop$. 
Let    $\{R_\alpha\}_\alpha$ be an inductive system in $\cRuclop$, 
and let $R$ be its colimit in $\cRu$: 
an ideal $I$ of $R$ is open if $I_\alpha$ (= inverse image of $I$ by $i_\alpha:R_\alpha \to R$) 
is open in $R_\alpha$ for any $\alpha$. 
Given two open ideals $I,J$ of $R$, we have that $i_\alpha^{-1}(IJ) \supseteq I_\alpha J_\alpha$, 
for any $\alpha$,  and
therefore $i_\alpha^{-1}(\ol{IJ}) \supseteq   \ol{I_\alpha J_\alpha}$ 
for any $\alpha$; we deduce that $\ol{IJ}$ is open in $R$. The case of $\cCRuclop$ follows from Example~\ref{exa-clop} and (1)  of 
Remark~\ref{op-canonical}.
\endgraf 
A minimal variation of the previous argument shows that the  category $\cRuop$  with $\op$-adic morphisms and the category $\cRuclop$ (resp. $\cCRuclop$) with   $\clop$-adic morphisms admit  all colimits. 
\end{rmk}

\begin{rmk} \label{prolim-clop}
Conditions op and clop are not stable under   limits in   $\cCRu$. 
In fact, any object of $\cCRu$ is a projective limit of discrete quotients (which are in $\cCRuclop$), while an example of  a ring $R$ in $\cCRu$  not $\clop$ was given in Examples~\ref{clop-exa}. But the inclusion of $\cCRuclop$ in $\cCRu$  admits a right adjoint $(-)^\clop$. Namely, for $R \in \cCRu$, $R^\clop$ is set-theoretically the same ring $R^\for$ equipped with the topology defined by the system 
$$\bigcup_{n\in \N} \{\ol{J_1\cdots J_n}\, | \, J_i \in \cP(R)\,,\, \forall \, i =1,\dots,n\,\} $$
of open ideals. This is complete by Lemma~\ref{BourbComplLemma} and is clearly clop since, for any $$I_1,\dots,I_m,J_1,\dots,J_n \in \cP(R)\;,$$
we have
$$
\ol{\ol{J_1\cdots J_n}\;\cdot\;\ol{I_1\cdots I_m}} \supset \ol{J_1\cdots  J_nI_1\cdots I_m} \;.
$$
\par 
So,  the category $\cCRuclop$   admits  all  limits calculated by application of the functor $(-)^\clop$ to the same limits calculated in $\cCRu$. 
\end{rmk}    
Any object $A$ of $\cCRu$ is the  limit of a  cofiltered projective system $(A_\lambda)_{\lambda \in \Lambda}$ of discrete rings and surjections 
$\pi_{\mu,\lambda} : A_\lambda \to A_\mu$, for any $\lambda \geq \mu$. In the following discussion we  fix such an $A$,
\beq \label{prodiscrring}
A = \limit_{\lambda \in \Lambda} A_\lambda 
\eeq
and let $\pi_\lambda : A \longrightarrow A_\lambda$ denote the projection; let $I_\lambda \in \cP(A)$ be the kernel of $\pi_\lambda$. 
To any ideal $J$ of $A$, we associate a projective sub-system $\pi_\bullet(J) := (J_\lambda)_{\lambda \in \Lambda}$ of 
$(A_\lambda)_{\lambda \in \Lambda}$, where, for any $\lambda$, $J_\lambda = \pi_\lambda(J) = (J +I_\lambda)/I_\lambda$ is an ideal of $A_\lambda$. For any $\lambda \geq \mu$ in $\Lambda$ we have 
$\pi_{\mu, \lambda}(J_\lambda) = J_\mu$. 
Clearly, $J$ is open iff there exists an index $\lambda_0 \in \Lambda$ such that $J \supset I_{\lambda_0}$, \ie $J = \pi_{\lambda_0}^{-1}(J_{\lambda_0})$, 
or, equivalently, $J_\lambda = \pi_{\lambda_0,\lambda}^{-1}(J_{\lambda_0})$, for any 
$\lambda \geq \lambda_0$; $J$ is closed iff $J = \bigcap_{\lambda \in \Lambda} \pi_{\lambda}^{-1}(J_{\lambda})$. Similarly for a subring $B \subset A$.  
\par Any $M \in \cCLMu_A$ is the projective limit of a  cofiltered projective system 
$(M_\sigma)_{\sigma \in \Sigma}$ of discrete uniform $A$-modules with surjective transition maps. As observed in 
Remark~\ref{discrete-mod}, for any $\sigma \in \Sigma$, there is $\lambda(\sigma) \in \Lambda$  such that $I_{\lambda(\sigma)} M_\sigma = (0)$. We may then replace both filtered posets $\Lambda$ and $\Sigma$ by the filtered poset
$$\Gamma := \{(\lambda,\sigma) \in \Lambda \times \Sigma \; : \; \lambda \geq \lambda(\sigma) \} 
$$
and set, for $\gamma = (\lambda,\sigma) \in \Gamma$,
$$
A_\gamma := A_\lambda \;\;,\;\; M_\gamma = M_\sigma  
$$
so that $M_\gamma$ is a discrete $A_\gamma$-module.  
So, we may assume that, for $A \in \cCRu$ and  $M \in \cCLMu_A$ there is a filtered set $\Lambda$ such that $A$ is expressed as in \eqref{prodiscrring} and 
\beq  \label{prodiscrmod}
M = \limit_{\lambda \in \Lambda} M_\lambda
\eeq 
  is the projective limit of a  cofiltered projective system 
 $(M_\lambda)_{\lambda \in \Lambda}$  of discrete  modules over the 
 projective system of discrete rings $(A_\lambda)_{\lambda \in \Lambda}$, where the   transition maps 
 $A_\lambda \to A_\mu$ and $M_\lambda \to M_\mu$ are all surjective. 
Let $\pi_\lambda : M \longrightarrow M_\lambda$ denote the projection. To any $A$-submodule $N$ of $M$, we associate a projective sub-system $\pi_\bullet(N) := (N_\lambda)_{\lambda \in \Lambda}$ of 
$(M_\lambda)_{\lambda \in \Lambda} =: \pi_\bullet(M)$, where, for any $\lambda$, $N_\lambda = \pi_\lambda(N)$ is a $A_\lambda$-submodule of $M_\lambda$. Again, $N$ is open iff there exists an index $\mu \in \Lambda$ such that 
$N = \pi_{\mu}^{-1}(N_{\mu})$, or, equivalently, $N_\lambda = \pi_{\mu,\lambda}^{-1}(N_\mu)$, for any $\lambda \geq \mu$,  while $N$ is closed iff $N = \bigcap_{\lambda \in \Lambda} \pi_{\lambda}^{-1}(N_{\lambda})$.  We summarize the situation in the following
\begin{lemma}  \label{clopchar0} Let $A \in \cCRu$ (resp. $M \in \cCLMu_A$) and let  $J$ be an ideal of $A$.\hfill \ben 
\item $A$ (resp. $M$) can 
be expressed as in \eqref{prodiscrring} (resp. \eqref{prodiscrmod}), with the same filtered set $\Lambda$.  
\item 
The closure $\ol{J}$ of $J$ in $A$ is $\bigcap_{\lambda} \pi_\lambda^{-1}(J_\lambda) = \limit_{\lambda \in \Lambda} J_\lambda$, and $\ol{J}$ is  open in $A$  if and only if there exists $\mu \in \Lambda$ such that  
$J_\lambda = \pi_{\mu,\lambda}^{-1} J_\mu$ 
(or, equivalently,  if and only if
$
A_\lambda/ J_\lambda \iso A_\mu/J_\mu 
$, via the map induced by $\pi_{\mu,\lambda}$)  
for any $\lambda \geq \mu $ in $\Lambda$. 
\item  
Let $N$ be an $A$-submodule of $M$. The closure $\ol{N}$ of $N$ in $M$ is $\bigcap_{\lambda} \pi_\lambda^{-1}(N_\lambda) =\limit_{\lambda \in \Lambda} N_\lambda$, and  $\ol{N}$ is an open subobject of $M$ in $\cCLMu_A$
if and only if   there exists $\mu \in \Lambda$ such that  $N_\lambda = \pi_{\mu,\lambda}^{-1}N_\mu$ 
(or, equivalently,  if and only if 
$
M_\lambda/ N_\lambda \iso M_\mu/N_\mu 
$, via the map induced by $\pi_{\mu,\lambda}$)
for any $\lambda \geq \mu$ in $\Lambda$. 
\een
\end{lemma}
\begin{proof} The first part has been proven before. We just prove the second part, since the third is proven similarly. For any $\lambda$, $\pi_\lambda^{-1}(J_\lambda)$
is an open ideal containing $J$. Conversely, any open ideal containing $J$ is a finite intersection of ideals of that  form. Since, for $\lambda_1 \leq \lambda_2$,  $\pi_{\lambda_2}^{-1}(J_{\lambda_2}) \subset \pi_{\lambda_1}^{-1}(J_{\lambda_1})$, any open ideal 
containing $J$ is of the form  $\pi_\lambda^{-1}(J_\lambda)$. 
We conclude that 
$$\ol{J} = \bigcap_{\lambda} \pi_\lambda^{-1}(J_\lambda)$$
and that $\ol{J}$ is open iff there exists $\mu$ such that 
$\pi_\lambda^{-1}(J_\lambda) = \pi_\mu^{-1}(J_\mu)$ for any $\lambda \geq \mu$. This is equivalent to the conditions in the statement. 
\end{proof}
 \begin{prop} \label{clopchar} Notation as before. 
Then  $A \in \cCRuclop$  iff for any $\mu \in \Lambda$ and any pair of ideals $J_\mu, H_\mu$ of $A_\mu$
\beq \label{clopchar2101}
\pi_{\lambda_1,\lambda_2}^{-1} (\pi_{\mu,\lambda_1}^{-1}(J_\mu) \pi_{\mu,\lambda_1}^{-1}(H_\mu)) 
= \pi_{\mu,\lambda_2}^{-1}(J_\mu) \pi_{\mu,\lambda_2}^{-1}(H_\mu)
\eeq
for $\lambda_2 \geq  \lambda_1  >> \mu$,
or, equivalently, iff the projective system
\beq \label{clopchar2102}
A_{\lambda_2}/\pi_{\mu,\lambda_2}^{-1}(J_\mu) \pi_{\mu,\lambda_2}^{-1}(H_\mu)  \map{\pi_{\lambda_1,\lambda_2}} A_{\lambda_1}/\pi_{\mu,\lambda_1}^{-1}(J_\mu) \pi_{\mu,\lambda_1}^{-1}(H_\mu)  \;,
\eeq 
for $\lambda_2 \geq \lambda_1 (\geq \mu)$, is eventually constant. 
\end{prop} 
\begin{proof} 
We apply the lemma to $J = \pi_\mu^{-1}(J_\mu)$ and $H = \pi_\mu^{-1}(H_\mu)$. Then, for any $\lambda \geq \mu$, 
$$J_\lambda = \pi_{\mu,\lambda}^{-1}(J_\mu) \;\;,\;\; H_\lambda = \pi_{\mu,\lambda}^{-1}(H_\mu) 
$$
and $\pi_\lambda (JH) = J_\lambda H_\lambda$. So, $\ol{JH}$ is open iff there exists $\lambda_1 \geq \mu$ such that 
$$\pi_{\lambda_1,\lambda_2}^{-1} (\pi_{\mu,\lambda_1}^{-1}(J_\mu) \pi_{\mu,\lambda_1}^{-1}(H_\mu))  = \pi_{\lambda_1,\lambda_2}^{-1} (J_{\lambda_1}  H_{\lambda_1} ) = J_{\lambda_2}  H_{\lambda_2} = 
 \pi_{\mu,\lambda_2}^{-1}(J_\mu) \pi_{\mu,\lambda_2}^{-1}(H_\mu) \;,$$
 $\forall \; \lambda_2 \geq \lambda_1$, where the central equality follows from the lemma. 
\end{proof}
\begin{defn}  \label{pseudocan} Let $k \in \cRu$. An object $M$ of $\cLMu_k$ is  \emph{pseudocanonical} if, for any $I \in \cP(k)$, $\ol{IM}$ is open in $M$.  Equivalently, $M$ is pseudocanonical iff the family $\{\ol{IM}\}_{I \in \cP(k)}$ is a basis of open $k$-submodules of $M$. We denote by  $\cLMpscan_k$ the full subcategory of $\cLMu_k$ consisting of pseudocanonical $k$-modules. 
\end{defn}
\begin{prop}  \label{clopmod} Notation as in 
Lemma~\ref{clopchar0}.  
Then \hfill
\ben
\item
$M  \in \cCLMpscan_A$  if and only if for any $\mu \in \Lambda$ and any ideal $J_\mu$ of $A_\mu$  
\beq \label{clopchar210}
M_{\lambda_2}/\pi_{\mu,\lambda_2}^{-1}(J_\mu)M_{\lambda_2} \iso M_{\lambda_1}/\pi_{\mu,\lambda_1}^{-1}(J_\mu)
M_{\lambda_1}  \;.
\eeq
for $\lambda_2 \geq \lambda_1 >> \mu$. 
\item
$M \in \cCLMuclop_A$ if and only if for any $\mu \in \Lambda$,  any ideal $J_\mu$ of $A_\mu$ and any $A_\mu$-submodule $N_\mu$ of $M_\mu$ 
\beq \label{clopchar21}
M_{\lambda_2}/\pi_{\mu,\lambda_2}^{-1}(J_\mu)\pi_{\mu,\lambda_2}^{-1}(N_\mu) \iso M_{\lambda_1}/\pi_{\mu,\lambda_1}^{-1}(J_\mu)
\pi_{\mu,\lambda_1}^{-1}(N_\mu)  \;.
\eeq
for $\lambda_2 \geq \lambda_1 >> \mu$. 
\een
\end{prop}
\begin{proof} 
Similar to the one of Proposition~\ref{clopchar}.
\end{proof} 
\end{subsection} 
\end{section}

\begin{section}{Limits of topological modules} \label{topmod} 
In this  section, $k$ is any object of $\cRu$. According to remarks \eqref{extsep} and \eqref{extcompl}, whenever a statement involves separated (resp. 
  complete) $k$-modules, $k$ may (and often will) be understood to be in $\cSRu$ (resp. $\cCRu$). 
\begin{subsection}{Limits and left-adjoints of inclusions} \label{limadj}
We have a commutative  diagram of categories and  inclusions of full subcategories 
\beq \label{diagcat}
\begin{tikzcd}[column sep=1.5em, row sep=0.5em] 
\cLMu_k  \arrow{r}{}  
&\cLMc_k \arrow{r}{}  
&\cLM_k   
\\ 
{}&{}&{} 
\\
\cSLMu_k \arrow{r}{} \arrow{uu}{} 
&  \cSLMc_k  \arrow{r}{} \arrow{uu}{} 
&\cSLM_k \arrow{uu}{}  
\\ 
{}&{}&{} 
\\
\cCLMu_k \arrow{r}{} \arrow{uu}{} 
&  \cCLMc_k  \arrow{r}{} \arrow{uu}{} 
&\cCLM_k \arrow{uu}{}  \;.
\end{tikzcd}
\eeq
An easy variant of  Propositions~\ref{stab-sep-mod} and \ref{stab-compl-mod} shows that all
 vertical inclusions in diagram \eqref{diagcat} have left adjoints, namely the separation and separated completion functors. 
A formal (partial) consequence is indicated in the next proposition, where we prefer to explicitly describe limits.
\begin{prop} \label{limproj} 
The categories  $\cLMu_k$, $\cLMc_{k}$, 
$\cLM_{k}$ admit   limits, compatible 
with the  forgetful functors to $\Mod_k$. 
The subcategories  
$\cSLMu_k$, $\cSLMc_{k}$, 
$\cSLM_{k}$
and $\cCLMu_k$, $\cCLMc_{k}$, 
$\cCLM_{k}$ are stable 
by  limits. 
\end{prop}
\begin{proof}  
Let  $(M_\alpha)_{\alpha \in A}$ be a projective system 
in $\cLM_k$ indexed by the 
preordered set $A$. Its   limit  in $\cLM_k$ is simply 
the   limit 
$M = \limit_{\alpha \in A} M_\alpha^\for$ in $\Mod_k$, 
equipped with the weak topology  of the canonical projections 
$\pi_\alpha : M \to M_\alpha$. For any $a \in k$ the map $\mu_M(a,-):M \to M$ is continuous since the composition $\pi_\alpha \circ \mu_M(a,-) = \mu_{M_\alpha} (a,\pi_\alpha(-)):M \to M_\alpha$ 
is continuous for any $\alpha$. 
It is then clear that $M$ is indeed the   limit of 
$(M_\alpha)_{\alpha \in A}$ in $\cLM_k$. 
If the projective system lies in $\cLMc_k$, we have to prove that 
the scalar multiplication $k\times M\to M$ is continuous for the product topology. 
This follows from $(4)$ of Proposition~\ref{scalar-prod-prop}: since $M$ is linearly topologized, 
it suffices to show that for any $x = (x_\alpha)_\alpha \in M$ the map $\mu_M(-,x):R \to M$ is continuous. 
The latter fact holds because the composition with  the projection $\pi_\alpha \circ \mu_M(-,x) = \mu_{M_\alpha} (-, x_\alpha)$ 
is continuous for any $\alpha$. 
If the projective system $(M_\alpha)_{\alpha \in A}$ lies in $\cLMu_k$, then the scalar multiplication of $M$ is uniformly continuous.   In fact, by $(5)$ of Proposition~\ref{scalar-prod-prop} it suffices to show that 
$M$ is bounded. This in turn follows from Proposition~\ref{boundedness1} since  $M_\alpha$ is bounded for any $\alpha$.
\endgraf 
If the projective system lies in $\cSLM_k$ (resp.~$\cCLM_k$), 
to show that $M$ is an object of $\cSLM_k$ (resp.~$\cCLM_k$) it suffices to prove that 
$M$ is separated (resp. and  complete). This is clear (resp. is proven in  \cite[II.5, Cor. to Prop. 10]{topgen}).
\endgraf 
\end{proof} 
\begin{parag} \label{functors}  
 
We now show that the horizontal arrows in  diagram \eqref{diagcat} also admit left adjoints. 
\end{parag} 
\begin{prop} \label{leftadjunif} 
For any object $M$ of $\cLM_k$, we define 
$$
\cP^\un(M) := \{ P + IM \,|\, I \in \cP(k),\ P\in\cP_k(M)\,\} 
$$
(notice that any $P + IM$ is an open submodule of $M$), 
and set $M^\un$ to be the 
$k$-linearly topologized $k$-module $(M^\for,\cP^\un(M))$. \hfill
\ben
\item
The $k$-linear topology having $\cP^\un(M)$ as a basis of open submodules 
is the maximal topology on $M$ weaker than the given one making $M$ a uniform module.  
In particular,  $M^\un$ is an object of $\cLMu_k$ and 
the canonical bijective morphism $M\to M^\un$ 
in $\cLM_k$ is an isomorphism if and only if $M$ is an object of $\cLMu_k$. 
For any $I \in \cP(k)$, the closure of $IM$ in $M$ coincides with the closure of $IM$ in $M^\un$.
\item
The correspondence $M \mapsto M^\un$ extends to a functor named \emph{uniformization} 
$\cLM_k \to \cLMu_k$ which is left adjoint to the inclusion $\iota_\un:\cLMu_k \to \cLM_k$.  
Namely, for any  $M$ in $\cLM_k$ and $N$ in $\cLMu_k$, there are canonical 
bifunctorial  identifications
\beq \label{adjointu}
  \Hom_{\cLMu_k}(M^\un,N) = \Hom_{\cLM_k}(M,\iota_\un (N)) \; .
\eeq 
\item  
The uniformization functor induces a functor
$\cSLM_k \to \cSLMu_k$, 
 $$M \longmapsto M^\un= M\big/\bigcap_{I}\ol{IM} \;,
$$ 
 where the latter is equipped with the  topology  induced by the family $\{N/\bigcap_{I}\ol{IM}\}$, for $N \in \cP^\un(M)$.  
The functor $M \longmapsto M^\un$ is
 left adjoint to the inclusion $\cSLMu_k \to \cSLM_k$. 
\item  
The uniformization functor induces a functor 
$\cCLM_k \to \cCLMu_k$,  
\beq \label{sepunif}
M \longmapsto  M^\un= \limit_{Q \in \cP^\un(M)}  M/Q  = \limit_{I \in \cP(k)}M/\ol{IM} \;,
\eeq
where quotients and limits  
are taken in $\cLM_k$.  
\een
\end{prop}
\begin{proof} Omitted.  
\end{proof}  
\begin{rmk} \label{sepunifrmk}  
 In formula~\ref{sepunif} the topology of $M/Q$ is discrete while the topology of $M/\ol{IM}$ is the quotient topology of the map $M \to M/\ol{IM}$.
\end{rmk}

\begin{prop} \label{leftadjcont}
For any object $M$ of $\LM_k$, 
let $\cI(M)$ denote the set of all maps $I:M\to\cP(k)$, $m \mapsto  I_m$.
Let us define 
$$
\cP^\co(M) := \Big\{ P + \sum_mI_mm \,|\, I \in \cI (M),\ P\in\cP(M)\, \Big\} 
$$
(notice that any element of $\cP^\co(M)$ is an open submodule of $M$), 
and set $M^\co$ to be the 
$k$-linearly topologized $k$-module $(M^\for,\cP^\co(M))$. 
\hfill \ben 
\item The elements of $\cP^\co(M)$ are sponges in $M$
and the $k$-linear topology having $\cP^\co(M)$ as a basis of open submodules 
is the maximal $k$-linear topology on $M$, weaker than the given one, which  makes $M$ a continuous module.  
In particular, $M^\co$ is an object of $\cLMc_k$ and 
the canonical bijective morphism $M\to M^\co$ 
in $\cLM_k$ is an isomorphism if and only if $M$ is an object of $\cLMc_k$. 
\item
The correspondence $M \mapsto M^\co$ extends to a functor named \emph{continuation}
$\cLM_k \to \cLMc_k$ which is left adjoint to the inclusion $\iota_\co:\cLMc_k \to \cLM_k$.  
Namely, for any  $M$ in $\cLM_k$ and $N$ in $\cLMc_k$, there are canonical 
bifunctorial  identifications
\beq \label{adjointc}
  \Hom_{\cLMc_k}(M^\co,N) = \Hom_{\cLM_k}(M,\iota_\co (N)) \; .
\eeq 
\item
The continuation functor induces a functor 
$\cSLM_k \to \cSLMc_k$, 
$$
M \longmapsto M^\co= M\big/\bigcap_{P \in \cP^\co(M)} P
$$ 
(equipped with the  topology  induced by the family $\{Q/\bigcap_{P \in \cP^\co(M)} P\}$, for $Q \in \cP^\co(M)$)
which is left adjoint to the inclusion $\cSLMc_k \to \cSLM_k$. 
\item
The continuation functor induces a functor 
$\cCLM_k \to \cCLMc_k$, 
$$
M \longmapsto M^\co= \limit_{Q \in \cP^\co(M)} M/Q 
$$
which is left adjoint to the inclusion $\cCLMc_k \to \cCLM_k$. 
\een
\end{prop}
\begin{proof} Omitted.
\end{proof}  
\begin{rmk} 
We observe that, for any $M$ in $\cLM_k$, 
the filters  $\cP^\un(M^\co)$ and $\cP^\un(M)$ of $k$-submodules of $M^\for$ are cofinal, 
so that the canonical morphism $M^\un \to (M^\co)^\un$ is an isomorphism. 
In other words, we have functorial morphisms in $\cLM_k$ (resp.~$\cSLM_k$, resp.~$\cCLM_k$) 
$$ 
M
\longrightarrow 
M^\co
\longrightarrow 
M^\un = (M^\co)^\un \;.
$$ 
\end{rmk} 
\begin{rmk} \label{projlimadj} From Propositions~\ref{leftadjunif} and \ref{leftadjcont} it follows that  
the  formation of limits in all vertices of the diagram \eqref{diagcat} is also compatible with the horizontal arrows, all full inclusions,   and   
with the  forgetful functor to $\Mod_k$. 
\end{rmk} 
\begin{rmk} \label{indlimadj}
The fact that $M \longmapsto M^\un$ (resp.  $M \longmapsto M^\co$) is the left adjoint to 
  the inclusion $\iota_\un:\cLMu_k \to \cLM_k$ (resp $\iota_\co:\cLMc_k \to \cLM_k$)
together with the existence of colimits $\colimit(-)$ in $\cLM_k$ shows that $\cLMu_k$ (resp $\cLMc_k$) admits all 
colimits, defined by  $\colimit^\un(-) = \colimit(-)^\un$ (resp. $\colimit^\co(-) = \colimit(-)^\co$). Similarly in the separated (resp. complete) case, where $M \longmapsto M^\un$ (resp.  $M \longmapsto M^\co$) denotes the left adjoint to 
  the inclusion $\iota_\un:\cSLMu_k \to \cSLM_k$ (resp $\iota_\co:\cCLMc_k \to \cCLM_k$), described in parts $\mathit 3$ and $\mathit 4$ of Propositions~\ref{leftadjunif} and \ref{leftadjcont}. A more explicit description of colimits will be given in subsection~\ref{colimsec}.
\end{rmk}
\end{subsection}
\begin{subsection}{Box products}
\label{counterbox} 
\begin{parag} Suppose $R$ is an object of $\cR$, 
and assume we have a family $\{M_\alpha\}_{\alpha \in A}$ of objects of 
$\cM_R$. 
The usual product 
\beq \label{normprod}
M := \PROD_{\alpha \in A} M_\alpha 
\eeq
is the product of the $R$-modules $M_\alpha$ equipped with the usual product topology, 
and it is in fact the product in the category  $\cM_R$. If every $M_\alpha$ is separately continuous then \eqref{normprod} is separately continuous hence it is the product in the category $\cMs_R$.
If $R$ is continuous (resp. uniform) and every $M_\alpha$ is continuous (resp. uniform), then \eqref{normprod} is the product in the category $\cMc_R$ (resp. $\cMu_R$). If $R \in \cCRu$ and $M_\alpha \in \cCLMc_R$ (resp. $M_\alpha \in \cCLMu_R$) for any $\alpha \in A$, then $M  \in \cCLMc_R$ (resp. $M \in \cCLMu_R$).    
\end{parag}
\begin{parag}
We will now equip the product $R$-module $M^\for$ with the finer
 \emph{box topology}, that is the  topology for which 
a filter basis of open subgroups consists of the subgroups
$U:= \PROD_{\alpha \in A} U_\alpha$, 
for any choice of the open subgroups $U_\alpha$ of $M_\alpha$, for any $\alpha$. 
This new topological $R$-module is an object of  $\cM_R$.  It will be called the  \emph{box-product} of the family 
$\{M_\alpha \}_{\alpha \in A}$ and will be denoted  
\beq \label{sqprod}
M^\square := \PRODsq_{\alpha \in A} M_\alpha \;.
\eeq 
If $M_\alpha$ is separated for any $\alpha \in A$, then \eqref{normprod}  is obviously separated. If all $M_\alpha$'s are complete, then \eqref{normprod} is complete and the subgroups $U$ as above are closed in it. It follows from Lemma~\ref{BourbComplLemma} that \eqref{sqprod} is complete, as well, and that there is a bijective morphism 
\beq
\label{sqbij}
M^\square = \PRODsq_{\alpha \in A} M_\alpha \map{(1:1)} \PROD_{\alpha \in A} M_\alpha = M\;.
\eeq
We easily see that if $R$ and all $M_\alpha$'s are $\op$ (resp. $\clop$), then  \eqref{sqprod}  is $\op$ (resp. $\clop$).  \par
If  $R$ is linearly topologized and  $\{M_\alpha\}_{\alpha \in A}$ is a family in 
$\cLM_R$, then \eqref{sqprod} is linearly topologized, as well.
On the other hand, even assuming that $R$ and all $M_\alpha$'s 
are $R$-linearly topologized and uniform, so that, by $(4)$ of Proposition~\ref{scalar-prod-prop}, multiplication by scalars is continuous at $(0,0)$, the topological $R$-module   \eqref{sqprod}  is not necessarily  separately
continuous. In fact, the previously defined $U$'s, with $U_\alpha \subset M_\alpha$ 
an open $R$-submodule for any $\alpha$, are $R$-submodules but not necessarily sponges.
\end{parag}
\begin{parag}
We consider only the following situations:
\begin{itemize}
\item $R$ is in $\cCRu$ and all $M_\alpha$'s are objects of $\cCLMc_R$; 
\item $R$ is in $\cCRu$ and all $M_\alpha$'s are objects of $\cCLMu_R$. 
\end{itemize}
\par   In the first case, we define the (complete) \emph{continuous box product} 
\beq \label{contsqprod}
\PRODsqcont_{\alpha \in A} M_\alpha 
\eeq  
of the family $\{M_\alpha\}_{\alpha \in A}$ to be 
the completion of the $R$-module $M^\for$ 
equipped with the $R$-linear topology for which 
  a basis of open $R$-submodules is given by the family (notation as in Proposition~\ref{leftadjcont})
 $$
\cP^{\square,\co}(M) = \{\cl_M(\sum_{m \in M} I_m m + \PROD_\alpha P_\alpha) \,|\, P_\alpha \in \cP(M_\alpha)\,,\, I \in \cI (M)\,\} \;.
$$    
It follows from Lemma~\ref{BourbComplLemma} that the natural morphism 
$$\PRODsqcont_{\alpha \in A} M_\alpha  \map{(1:1)} \PROD_{\alpha \in A} M_\alpha $$
is bijective. In the end, we 
 have  natural  bijective morphisms 
\beq \label{sqcontbij}
\PRODsq_{\alpha \in A} M_\alpha \map{(1:1)} \PRODsqcont_{\alpha \in A} M_\alpha \map{(1:1)}  \PROD_{\alpha \in A} M_\alpha\;.
\eeq 
 \endgraf
  In the second case, 
we define the \emph{uniform box product} 
\beq \label{unifsqprod}
\PRODsqu_{\alpha \in A} M_\alpha 
\eeq 
of the family $\{M_\alpha\}_{\alpha \in A}$ to be the completion of the $R$-module $M^\for$ 
equipped with the $R$-linear topology for which 
  a basis of open $R$-submodules is given by the family 
  \beq \label{usquare0}
\cP^{\square,\un}(M) :=  \{U((P_\alpha)_\alpha,J) \,|\,(P_\alpha)_\alpha  \in \PROD_\alpha \cP(M_\alpha)\,,\, J \in \cP (R)\,\} 
\eeq
where  
\beq \label{usquare1}
U((P_\alpha)_\alpha,J) := \PROD_\alpha (P_\alpha + JM_\alpha) = \cl_M( JM  + \PROD_\alpha P_\alpha) \;.
\eeq
By Lemma~\ref{BourbComplLemma}, the uniform box product 
\eqref{unifsqprod} identifies with $M^\for$ equipped with the $R$-linear  topology for which a basis of open $R$-submodules is the set of  $\prod_{\alpha \in A} P_\alpha$, for $P_\alpha \in \cP_R(M_\alpha)$, for which there exists an open ideal $I \in \cP(R)$ such that $P_\alpha \supset I M_\alpha$, for any $\alpha \in A$.
Again by Lemma~\ref{BourbComplLemma}  we have   bijective morphisms 
\beq \label{squnifbij}
\PRODsq_{\alpha \in A} M_\alpha \map{(1:1)} \PRODsqcont_{\alpha \in A} M_\alpha \map{(1:1)}  \PRODsqu_{\alpha \in A} M_\alpha \map{(1:1)}  \PROD_{\alpha \in A} M_\alpha\;.
\eeq
\end{parag}
\end{subsection}
\begin{subsection}{Clop, barrelled and pseudocanonical modules.}
\label{clopbarrell} 
We consider here full subcategories of $\cLM^\ast_k$, $\cSLM^\ast_k$, $\cCLM^\ast_k$, for $\ast =\emptyset, \un,\co$ whose full embedding admits a right adjoint. 
\begin{prop} \label{rightadjclop} We assume here that $k$ is an object of  $\cRuclop$.
For any object $M$ of $\cLMc_k$, we define 
$$
\cP^\clop(M) := \{ \ol{JP} \,|\, J \in \cP(k),\ P\in\cP_k(M)\,\} 
$$
and set $M^\clop$ to denote the 
$k$-linearly topologized $k$-module $(M^\for,\cP^\clop(M))$. 
The $k$-linear topology having $\cP^\clop(M)$ as a basis of open submodules 
is the minimal $k$-linear topology on $M$, finer than the given one, which makes $M$ a $\clop$ module.  
\endgraf 
Then $M^\clop$ is an object of $\cLMcclop_k$ and 
the canonical (bijective) morphism $M^\clop\to M$ 
in $\cLM_k$ is an isomorphism if and only if $M$ is an object of $\cLMcclop_k$. 
The correspondence $M \mapsto M^\clop$ extends to a functor 
$\cLMc_k \to \cLMcclop_k$ which is right adjoint to the inclusion $\iota_\clop:\cLMcclop_k \to \cLMc_k$.  
Namely, for any  $M$ in $\cLMc_k$ and $N$ in $\cLMcclop_k$, there are canonical 
bifunctorial  identifications
\beq \label{adjointclop}
  \Hom_{\cLMc_k}(\iota_\clop N,M) = \Hom_{\cLMcclop_k}(N,M^\clop) \; .
\eeq 
This functor restricts to a functor $\cLMu_k \to \cLMuclop_k$ (resp. $\cSLM^\ast_k \to \cSLM^{\ast,\clop}_k$, resp.~$\cCLM^\ast_k \to \cCLM^{\ast,\clop}_k$, for $\ast = \un,\co$) 
which is right adjoint to the respective inclusion of categories.
\end{prop}
\begin{proof} 
For any $k$-submodule $U$ of $M$, we denote by $\wtilde{U}$ (resp. $\ol{U}$) the closure of $U$ in $M^\clop$ (resp. in $M$). 
To show that $M^\clop$ is in fact clop, 
it suffices to show that, for any $I, J \in \cP(k)$ and $P \in \cP(M)$, 
$$
\wtilde{J\,(\ol{IP})} =  \ol{\ol{JI}P} \;.
$$
We recall that 
$$
\wtilde{J\,(\ol{IP})} = \bigcap \ol{HQ}
$$ 
where the intersection is taken over all $\ol{HQ}$, 
with $H \in \cP(k)$ and $Q \in \cP(M)$,  such that $\ol{HQ} \supset J \ol{IP}$. 
But then $\ol{HQ} \supset \ol{ J \ol{IP}} \supset \ol{JIP} =  \ol{\ol{JI} P}$. 
So,  $\ol{\ol{JI} P}$ is the smallest of the $\ol{HQ}$'s under consideration. 
Our assertion follows. We then conclude that $M^\clop$ is in fact $\clop$, and that $(M^\clop)^\clop = M^\clop$. 
\endgraf  
Let now $M'$ be an object of $\cLMcclop_k$, equipped with a $\cLMc_k$-morphism 
$M' \longrightarrow M$   whose underlying $k$-linear map is the identity of $M^\for$, and let 
$$\cP_k(M') 
  \supset \cP_k(M)
$$ 
be  the family of open $k$-submodules of $M'$. We claim that 
$$\cP_k(M')  
\supset \cP^\clop_k(M)\;.$$
In fact, let $P \in \cP_k(M)$, and let $I \in \cP(k)$. The closure $\cl_{M'}(IP)$ of $IP$ in $M'$ is contained in $\ol{IP} \in \cP^\clop(M)$, so that the latter is also open in $M'$.  
\endgraf
The remaining parts of the Proposition are easy 
(the stability of completeness  follows from Lemma~\ref{BourbComplLemma}). 
\end{proof}

\begin{rmk} \label{adjointcloprmk} 
\hfill \ben
\item 
By the  right adjoint property of $M \mapsto M^\clop$ as a functor $\cLM^\ast_k \to \cLM^{\ast,\clop}_k$ (resp. $\cSLM^\ast_k \to \cSLM^{\ast,\clop}_k$, resp.~$\cCLM^\ast_k \to \cCLM^{\ast,\clop}_k$, for $\ast = \un,\co$) 
we  deduce the existence of projective limits in the target categories, 
 calculated by applying the functor $M \mapsto M^\clop$ 
to  projective limits in the source categories. 
 \item
 If $k$ is in $\cRuop$, the proof of the previous proposition  simplifies to prove the existence of the right adjoint 
$M \mapsto M^\op$ of the inclusion $\iota_\op: \cLMcop_k \to \cLMc_k$. 
The separated and uniform variants hold, as well. The complete variant does not in general hold. 
\item As was proven in remark~\ref{indlim-clop} for colimits in $\cRu$ 
of an inductive system  in $\cRuop$ (resp. $\cRuclop$) and for colimits in $\cCRu$ of inductive systems in $\cCRuclop$, we may show that for $k \in \cRuop$ (resp. for $k \in \cRuclop$) $\cLMuop_k$ (resp. $\cLMuclop_k$) is stable under colimits in $\cLMu_k$ and that, for $k \in \cCRuclop$, $\cCLMuclop_k$  is stable under colimits in $\cCLMu_k$.  
\een
\end{rmk} 

\begin{defn} \label{barreldef} Let $k \in \cRu$. An object $M$ of $\cLMc_k$ is  \emph{barrelled} if any closed 
$k$-submodule of $M$ which is a sponge is open.
\end{defn}  

\begin{rmk} \label{barrpscan} Let $k \in \cRu$.  \hfill \ben 
\item Any barrelled $M \in \cLMu_k$ is pseudocanonical. In fact,  for any $I \in \cP(k)$,  $\ol{IM}$ is a closed sponge in $M$, so that it is open.
\item   Let  $k \in \cRuclop$. Then any barrelled $M \in \cLMc_k$
 is clop. In fact,   let $P \in \cP_k(M)$  and $I \in \cP(k)$. Since $M$ is continuous, $P$ is a sponge. Then $\ol{IP}$ is a closed sponge in $M$~:
let $x \in M$ and let $J \in \cP(k)$ be such that $Jx \subset P$. Then $JIx \subset IP$ and $\ol{JI}x \subset \ol{IP}$, but $\ol{JI} \in \cP(k)$ since $k$ is clop.  Since $M$ is barrelled,  $\ol{IP}$  is open in $M$. 
\item  Any $M \in \cLMuclop_k$
is pseudocanonical. If $k \in \cRuclop$, the converse holds, as well.   This is because, if $k \in \cRuclop$, $M$ is pseudocanonical and $I,J \in \cP(k)$, 
$$
\ol{I\,\ol{JM}} = \ol{\ol{IJ}\,M} 
$$
is open in $M$.   
\item If $k \in \cRuclop$, then
$$\cLMubarrell_k \subset \cLMpscan_k =  \cLMuclop_k  \;.
$$
\item Let $\{M_\alpha\}_{\alpha \in A}$ be an inductive system in $\cLMubarrell_k$ and let $M =\colimit_{\alpha \in A}M_\alpha$ be its colimit in $\cLMu_k$. Let $C \subset M$ be a closed $k$-sponge and, for any $\alpha \in A$, let $i_\alpha:M_\alpha \map{} M$ be the canonical morphism. Then $i_\alpha^{-1}(C)$ is a closed sponge in $M_\alpha$, hence an open $k$-submodule. It follows that $C$ is an open $k$-submodule of $M$. So, $\cLMubarrell_k$ is cocomplete. 
\item Let $k \in \cCRu$ and $M \in \cLMubarrell_k$. Then the completion $\what{M} \in \cLMubarrell_k$. In fact, let $j: M \map{} \what{M}$ be 
the canonical morphism and let $C \subset \what{M}$ be a closed $k$-sponge. Then $j^{-1}(C)$  is a closed $k$-sponge in $M$. Therefore $j^{-1}(C)$ is open in $M$ and $C \subset \what{M}$ is then an open submodule. 
\een 
\end{rmk}
\begin{prop} \label{clopbarrel}  Let $k \in \cRou$ and   $M \in \cCLMouclop_k$. Then $M$ is barrelled. 
\end{prop}
\begin{proof}
Let   $N \subset M$ be a closed $k$-submodule which is a sponge. Let $\{I_n\}_{n \in \N}$ be a countable fundamental system of open ideals of $k$. Then, for any $n \in \N$, 
$$(N : I_n) = \{m \in M\,|\, I_nm \subset N\,\} = \bigcap_{x \in I_n} \{m \in M\,|\, x m \in N\,\}
$$
is closed $k$-submodule of $M$ and therefore 
$$
M = \bigcup_{n \in \N} (N : I_n)
$$
is a countable union of closed $k$-submodules. As recalled in Remark~\ref{baire1},  $M$ is a Baire space, so that we have  $(N : I_n) \in \cP_k(M)$, for some $n$. Therefore, $\ol{I_n(N : I_n)} \in \cP_k(M)$ and then $N \in \cP_k(M)$.  So, $M$ is barrelled. 
\end{proof} 
\begin{prop} \label{barrellif} Let $k \in \cRu$. 
For any object $M$ of $\cLMc_k$, we define 
$$
\cP^\barrell(M) := \{ \,\mbox{\rm closed $k$-sponges in $M$}\,\} 
$$
and set $M^\barrell$ to denote the 
$k$-linearly topologized $k$-module $(M^\for,\cP^\barrell(M))$. The $k$-linear topology having $\cP^\barrell(M)$ as a basis of open submodules 
is the minimal $k$-linear topology on $M$, finer than the given one, which makes $M$ a barrelled $k$-module. 
The correspondence $M \mapsto M^\barrell$ induces a functor 
$\cLM^\ast_k \to \cLM^{\ast, \barrell}_k$ (resp. $\cSLM^\ast_k \to \cSLM^{\ast, \barrell}_k$, resp. $\cCLM^\ast_k \to \cCLM^{\ast, \barrell}_k$), for $\ast = \un,\co$,   right adjoint to the natural inclusion of categories deduced from  $\iota_\barrell: \cLM^{\ast, \barrell}_k \to \cLM^\ast_k$. The natural morphism $M^\barrell \longrightarrow M$ is a bijection (the only non-trivial case being the one of complete modules which follows from Lemma~\ref{BourbComplLemma}).  
\end{prop}
\begin{proof}
Similar to the one of Proposition~\ref{rightadjclop}.
\end{proof}
 Similar to Remark~\ref{adjointcloprmk} we have
\begin{rmk} \label{adjointbarrellrmk}  
By the  right adjoint property of $M \mapsto M^\barrell$ as a functor $\cLM^\ast_k \to \cLM^{\ast,\barrell}_k$ (resp. $\cSLM^\ast_k \to \cSLM^{\ast,\barrell}_k$, resp.~$\cCLM^\ast_k \to \cCLM^{\ast,\barrell}_k$, for $\ast = \un,\co$) 
we  deduce the existence of projective limits in the target categories, 
 calculated by applying the functor $M \mapsto M^\barrell$ 
to  projective limits in the source categories.  
\end{rmk} 
 
\begin{prop} \label{pseudocan3} Let $k \in \cRu$.
For any object $M$ of $\cLMu_k$, we define 
$$
\cP^\pscan(M) := \{ \ol{IM} \,|\, I \in \cP(k)\,\} 
$$
and set $M^\pscan$ to denote the 
$k$-linearly topologized $k$-module $(M^\for,\cP^\pscan(M))$.  
The $k$-linear topology having $\cP^\pscan(M)$ as a basis of open submodules 
is the minimal $k$-linear topology on $M$, finer than the given one, which makes $M$ a pseudocanonical $k$-module. 
The correspondence $M \mapsto M^\pscan$ induces a functor 
$\cLMu_k \to \cLMpscan_k$ which is right adjoint to the inclusion $\iota_\pscan:\cLMpscan_k \to \cLMu_k$; the natural morphism $M^\pscan \longrightarrow M$ is a bijection. Similarly in the separated and complete cases.  
\end{prop}
\begin{proof} Omitted. 
\end{proof}
 Similar to Remarks~\ref{adjointcloprmk} and \ref{adjointbarrellrmk} we have
\begin{rmk} \label{adjointpscanrmk}   
By the  right adjoint property of $M \mapsto M^\pscan$ as a functor $\cLMu_k \to \cLMpscan_k$ (resp. $\cSLMu_k \to \cSLM^\pscan_k$, resp.~$\cCLMu_k \to \cCLM^\pscan_k$) 
we  deduce the existence of projective limits in the target categories, 
 calculated by applying the functor $M \mapsto M^\pscan$ 
to  projective limits in the source categories.  
\end{rmk} 
\begin{rmk} \label{pscancount} For $k \in \cRou$ and   $M \in \cLMu_k$, $M^\pscan \in \cLMou_k$.
\end{rmk}
\begin{rmk} \label{ouclop} 
Let $k \in \cRu$. \hfill \ben \item 
Part 1 of Remark~\ref{barrpscan} implies that, for any $M \in \cLMu_k$,  the natural morphism $M^\barrell \map{(1:1)} M$ factors as 
$$
M^\barrell \map{(1:1)} M^\pscan \map{(1:1)} M \;.
$$
\item Part 2 of Remark~\ref{barrpscan} implies that, if $k \in \cRuclop$ and  $M \in \cLMc_k$,   the natural morphism $M^\barrell \map{(1:1)} M$ factors as 
$$
M^\barrell \map{(1:1)} M^\clop \map{(1:1)} M \;.
$$
\item Part 3 of Remark~\ref{barrpscan} implies that, if $k \in \cRuclop$,  the identity map of $M^\for$ is an isomorphism between $M^\clop$ and $M^\pscan$. 
\een
\end{rmk}
We conclude that 
\begin{cor}  \label{barrpscan20}
If $k \in \cRuclop$ and  $M \in \cLMu_k$, then the identity map of $M^\for$ identifies  $M^\pscan$ to $M^\clop$ and the
natural morphism $M^\barrell \longrightarrow M$ factors through 
  \beq \label{barrpscan22}
M^\barrell \map{(1:1)}  M^\clop = M^\pscan \map{(1:1)} M \;.
\eeq
So, for $k \in \cRuclop$,  the two categories 
$$ \cCLMuclop_k = \cCLM^\pscan_k  
$$ 
coincide. This category may be described as the full subcategory of $\cCLMu_k$ of objects which carry a minimal clop structure or, equivalently, a minimal pseudocanonical structure. 
\end{cor}
\begin{proof} 
The first part of statement  follows from Remark~\ref{ouclop}. The second is  an obvious  consequence. 
\end{proof} 
\begin{cor}  \label{barrpscan2} 
Let $k \in \cCRouclop$. Then, for any   $M \in \cCLMou_k$ the identity map of $M^\for$ identifies 
  $M^\pscan$, $M^\clop$ and $M^\barrell$ 
  and we have a single natural morphism
  \beq \label{barrpscan222} M^\barrell  = M^\clop  =  M^\pscan  \map{(1:1)} M
\;.\eeq
\end{cor}
\begin{proof}
The statement   follows from Corollary~\ref{barrpscan20} and Proposition~\ref{clopbarrel}.
\end{proof}  
\begin{cor}  \label{pseudocanbarr2} 
If $k \in \cCRouclop$, then 
$$ \cCLMouclop_k = \cCLM^\pscan_k = \cCLMoubarrell_k  
$$ 
is a full additive subcategory of $\cCLMou_k$  with  countable limits, calculated by applying any of the functors 
$(-)^\clop = (-)^\pscan = (-)^\barrell$ to the limit in $\cCLMou_k$, and finite colimits, calculated in $\cCLMou_k$.  
\end{cor}  
 \end{subsection}  
  \end{section} 
  \begin{section}{Colimits of topological modules} 
  As in the previous  section, $k$ is any object of $\cRu$ but, when  a statement involves separated (resp. 
  complete) $k$-modules, $k$ will be understood to be in $\cSRu$ (resp. $\cCRu$). 
\begin{subsection}{Colimits~: explicit description} \label{colimsec}
We already observed in Remark~\ref{indlimadj} that $\cLM^\ast_k$,  for $\ast = \un,\co$, admits all 
colimits, defined by  $\colimit^\ast(-) = \colimit(-)^\ast$,  where $\colimit(-)$ denotes the colimit in $\cLM_k$ and the apex $\ast$ indicates both the colimit in $\cLM^\ast_k$ and the functor $M \longmapsto M^\ast$. Similarly for  $\cSLM^\ast_k$ and $\cCLM^\ast_k$.   
 It is however useful to have a more explicit description of 
those colimits at our disposal. We do so only in the cases of interest to us. 
\endgraf
For an inductive system $\{M_\alpha\}_{\alpha \in A}$ in $\cLM_k$,   indexed by the 
preordered set $A$,  the inductive  limit  in $\cLM_k$ is calculated as follows. Let
$$M := \colimit_\alpha (M_\alpha, \{\iota_\alpha : M_\alpha \to M\}_\alpha)$$
 be the familiar colimit in $\Mod_k$.  
We may give to $M$ the finest $k$-linear topology  such that all maps 
$\iota_\alpha: M_\alpha \to M$ are continuous. 
So, a sub-basis of open $k$-submodules in $M$ consists of 
the $k$-submodules $U$ of $M$  such that
$U_\alpha := \iota_\alpha^{-1}(U)$ is open in $M_\alpha$, for any $\alpha \in A$. 
Then $M$, endowed with this topology and with the natural morphisms $\iota_\alpha$, 
represents the colimit of $\{M_\alpha\}_\alpha$ in $\cLM_{k}$, written $\colimit^{\cLM_{k}}_\alpha M_\alpha$.
\par 
For an inductive system $\{M_\alpha\}_{\alpha \in A}$ in $\cLM^\ast_k$, for $\ast = \un,\co$,  let  
$\colimit^{\cLM^\ast_k}_\alpha M_\alpha$
be the colimit in 
$\cLM^\ast_k$.  It  is $M$ equipped with the $k$-linear topology for which a basis of open $k$-submodules consists of the \emph{sponges} $P \subset M$ such that $\iota_\alpha^{-1}(P) \in \cP(M_\alpha)$ for any $\alpha$, and, in case $\ast =\un$,  \emph{such that there exists $I \in \cP(k)$ with  $IM \subset P$}.  If $k$ and all $M_\alpha$'s are discrete, then 
$\colimit^{\cLM^\ast_k}_\alpha M_\alpha$ is discrete, as well. 
\par 
For an inductive system $\{M_\alpha\}_{\alpha \in A}$ in $\cCLM^\ast_k$, for $\ast = \un,\co$,  the colimit in 
$\cCLM^\ast_k$,  will be denoted $\colimit^\ast_\alpha M_\alpha$ for short. It  is 
  the (separated) completion of 
  $\colimit^{\cLM^\ast_k}_\alpha M_\alpha$.
    If $k$ and all $M_\alpha$'s are discrete, then 
$\colimit^\ast_\alpha M_\alpha$ is discrete, as well. 
\par
The cokernel of a morphism $f: N \to M$ in $\cSLM^\ast_k$ is $M/\ol{f(N)}$ equipped with the quotient topology. 
If $f$ is a morphism in $\cCLM^\ast_k$ its cokernel is the completion of the quotient space $M/\ol{f(N)}$, namely 
$$
 \Coker^\ast (f) = \limit_{P \in \cP(M)} M/(f(N) + P)
$$
where $M/(f(N) + P)$ is calculated in $\Mod_k$ and is endowed with the discrete topology.  
If $M \in \cCLM^{\omega,\ast}_k$, then $\Coker^\ast (f) = M/\ol{f(N)}$ is already complete for the quotient topology. In particular, if $M \in \cCLM^{\omega,\ast}_k$ and $f$ is closed, then 
$(\Coker^\ast (f))^\for$ coincides with $\Coker (f^\for)$. 
For any injective $\cCLM^\ast_k$-morphism $i: N \hookrightarrow M$, we write $(M/N)^\ast$ for $\Coker^\ast (i)$.
\par \medskip
The next result is self-explanatory.
\begin{prop} \label{induprop} For any inductive system 
$$M_\bullet = \{M_\alpha\}_{\alpha \in A} $$
in $\cCLMu_k$, let 
$$M  = \colimit^{\cLMu_k}_{\alpha \in A} M_\alpha \;.$$ 
For any $\alpha \in A$, we write $j_\alpha : M_\alpha \to M$ for the canonical morphism.  
For any $I \in \cP(k)$, let $\cB_I$ be the 
cofiltered set of $Q \in \cP(M)$ such that $IM \subset Q$. For any $\alpha \in A$, let $Q_\alpha := j_\alpha^{-1}(Q) \in \cP(M_\alpha)$. Then 
\beq \label{indu}
\colimit^\un_{\alpha \in A} M_\alpha = \limit_{I \in \cP(k)} \limit_{Q \in \cB_I}  \colimit^\un_{\alpha \in A} 
 M_\alpha / Q_\alpha \;,
\eeq 
where $\colimit^\un_{\alpha \in A} M_\alpha / Q_\alpha$ coincides with the colimit in $\Mod_{k/I}$ equipped with 
the discrete topology.
\end{prop}
\end{subsection}
 \begin{subsection}{Direct sums in $\cCLMc_k$ and $\cCLMu_k$ and box-products}  \label{compl-dirsums}
The direct sum $M':= {\SUM}^{\cM_k}_\alpha M_\alpha$ of a family $\{M_\alpha\}_\alpha$  in the category $\cM_{k}$ is 
the algebraic direct sum ${\SUM}_\alpha M_\alpha$ of the $k$-modules $M_\alpha^\for$ in $\Mod_k$, equipped with the 
family of open subgroups  $\{{\SUM}_\alpha P_\alpha\,|\, P_\alpha \in \cP(M_\alpha)\,,\,\forall \alpha \}$.
This is the subspace topology induced by the natural 
inclusion in the box-product \eqref{sqprod} of the $M_\alpha$'s.  
The following Lemma~\ref{complsum} and, more precisely, part $\mathit 1$ of Proposition~\ref{complsum2} should be compared with \cite[Chap. I, Lemma 7.8]{schneider}. 
\begin{lemma} \label{complsum} Let $\{M_\alpha\}_{\alpha \in A}$ be a family in $\cSM_k$. Then 
the algebraic sum 
${\SUM}_\alpha M_\alpha$ is   closed in the  box-product \eqref{sqprod}. 
 In particular, if the $M_\alpha$'s are complete
\beq  \label{complsum1}
{\SUM}^{\cM_k}_\alpha M_\alpha = {\SUM}^{\cCM_k}_\alpha M_\alpha
\eeq
is complete.
 If moreover the $M_\alpha$'s are continuous
\beq  \label{complsumcont}
{\SUM}^{\cM_k}_\alpha M_\alpha = {\SUM}^{\cCMc_k}_\alpha M_\alpha  
\eeq
 is continuous.  
\end{lemma}
\begin{proof}
To prove the first part of the statement, let 
$x = (x_\alpha)_\alpha \in \PRODsq_{\alpha \in A} M_\alpha \setminus {\SUM}_\alpha M_\alpha$.  Then $x_\alpha \neq 0$, for $\alpha$ in an infinite subset $A' \subset A$. For any $\alpha \in A'$, let $P_\alpha \in \cP(M_\alpha)$ be such that $x_\alpha \notin P_\alpha$; for $\alpha \in A \setminus A'$, let $P_\alpha = M_\alpha$. Then 
\beq \label{closedness} ( {\SUM}_\alpha M_\alpha ) \cap (x + \prod_{\alpha \in A} P_\alpha) = \emptyset \;.
\eeq
We are left to prove the last part of the statement, namely that, if all $M_\alpha$'s are continuous, then  any ${\SUM}_\alpha P_\alpha$, for $P_\alpha \in \cP(M_\alpha)$, is a sponge in ${\SUM}_\alpha M_\alpha$. But for any $x = (x_\alpha)_\alpha \in {\SUM}_\alpha M_\alpha$, $x_\alpha = 0$  for almost all $\alpha$'s, so the assertion is clear.
\end{proof} 
As in Example~\ref{counterbox} we only consider 
 the following situations:
 \begin{prop} \label{complsum2} Let $k \in \cCRu$. Notation as in Lemma~\ref{complsum}. Then  
\ben 
\item
Let all $M_\alpha$'s be objects of $\cCLMc_k$. Then  
${\SUM}^{\cCLMc_k}_\alpha M_\alpha = {\SUM}^{\cM_k}_\alpha M_\alpha$ is 
    the algebraic direct sum of $k$-modules ${\SUM}_\alpha  M^\for_\alpha$ equipped with the $k$-linear topology determined by the basis of open $k$-submodules $\{{\SUM}_\alpha P_\alpha\,|\, P_\alpha \in \cP(M_\alpha)\,\}$.
\item
Let all $M_\alpha$'s be objects of $\cCLMu_k$.  Then 
$M := {\SUM}^{\cCLMu_k}_\alpha M_\alpha$ is the completion of ${\SUM}^{\cLMu_k}_\alpha M_\alpha = 
({\SUM}^{\cM_k}_\alpha M_\alpha)^\un$ or, equivalently, the closure of the algebraic direct sum of $k$-modules 
${\SUM}_\alpha  M^\for_\alpha$ in $\PRODsqu_\alpha M_\alpha$, endowed with the subspace topology. 
Set-theoretically, $M$ is  the subset  of $\PROD_\alpha M_\alpha$  consisting of the $(x_\alpha)_\alpha$'s  
such that, for any $U((P_\alpha)_\alpha,J)$ as in \eqref{usquare1}, $x_\alpha \in P_\alpha + J M_\alpha$ for all but a finite number of $\alpha \in A$.  
 \een
\end{prop}
\begin{proof}  Part $\mathit 1$ follows from Lemma~\ref{complsum}. Part $\mathit 2$ is clear.
  \end{proof}   
  \begin{rmk} \label{locconvdirsum} Let $k \in \cRu$
and let all $M_\alpha$'s be objects of $\cLMc_k$.  
Then  
${\SUM}^{\cLMc_k}_\alpha M_\alpha$ is a vast generalization of the notion of \emph{locally convex direct sum} 
 of locally convex $K$-vector spaces, for $K$ a nonarchimedean field,  given in \cite[Chap. I, \S 5,  E1]{schneider}. To reconcile our definition with the one of \lc, we let  $k = K^\circ$ and assume  the $M_\alpha$'s to be locally convex $K$-vector spaces.  
 If the $M_\alpha$'s are separated, \cite[Chap. I, Lemma 7.8]{schneider} is a special case of 
 \eqref{complsumcont} and, more precisely, of part $\mathit 1$ of Proposition~\ref{complsum2}.
 \par 
 In strong contrast, when all $M_\alpha$'s be objects of $\cCLMu_k$, part $\mathit 2$ of Proposition~\ref{complsum2} illustrates the difference between ${\SUM}^{\cLMu_k}_\alpha M_\alpha$, whose underlying $k$-module is the algebraic direct sum 
 ${\SUM}_\alpha M^\for_\alpha$,  and its completion ${\SUM}^{\cCLMu_k}_\alpha M_\alpha$.
  \end{rmk}  
\begin{notation} 
\label{lcdirsum} For any set $A$, any ring $R$,  and any $M \in \Mod_R$ it is customary to indicate by $M^{(A)}$ (resp. $M^A$) the direct sum (resp. product) in $\Mod_R$ of a family indexed by $A$ of copies of $M$. For $k \in \cCRu$, we introduce a similar notation in our topological categories $\cCLMc_k$ and  $\cCLMu_k$, as follows.
\par
 For a family $\{M_\alpha\}_{\alpha \in A}$ in $\cCLMc_k$ (resp. in $\cCLMu_k$), we often shorten  
$$
{\SUM}_{\alpha \in A}^{\cCLMc_k}M_\alpha \;\;\mbox{(resp.} \; {\SUM}_{\alpha \in A}^{\cCLMu_k}M_\alpha\;\mbox{)} \;\;\mbox{into}\; \;{\SUM}_{\alpha \in A}^\co M_\alpha \;\mbox{(resp.}\; {\SUM}_{\alpha \in A}^\un M_\alpha \;\mbox{)}\;.
$$
Moreover, for any set $A$ and $M \in \cCLMc_k$ (resp. $\in \cCLMu_k$) we set 
\beq
\label{cudirsum}
 M^{(A,c)} = {\SUM}^\co_{\alpha \in A} M_\alpha \;\;\mbox{(resp.} \; M^{(A,\un)} = {\SUM}^\un_{\alpha \in A}M_\alpha 
 \;\mbox{),} \; M_\alpha = M \;, \;\forall \; \alpha \in A\;.
\eeq 
  We also set, for $M_\alpha = M \in \cCLMc_k$  (resp. $\in \cCLMu_k$),  $\forall \, \alpha \in A$,
         $$ M^A =  \PROD_{\alpha \in A} M_\alpha \in \cCLMc_k  \;\mbox{(resp.} \in \cCLMu_k \;\mbox{)}\;,$$
  $$M^{A,\square} =  \PRODsq_{\alpha \in A} M_\alpha  \; \in \cCLM_k \;,$$
\beq
\label{cuboxprod}
 M^{A,\square,\co}  =  \PRODsqcont_{\alpha \in A} M_\alpha \in \cCLMc_k   
 \;\;\mbox{(resp.} \;M^{A,\square, \un} =  \PRODsqu_{\alpha \in A} M_\alpha \in
 \cCLMu_k \;\mbox{)}\;.
 \eeq
 If $B$ is another set, we have 
  $$
 (M^{(A,c)})^{(B,c)} = M^{(A \times B,c)} \;\;\mbox{(resp.} \; (M^{(A,\un)})^{(B,\un)} = M^{(A \times B,\un)}  \;\mbox{)}\;.
 $$
 $$
 (M^A)^B = M^{A\times B} \;,\; (M^{A,\square})^{B,\square} = M^{A\times B,\square} \;,
 $$
  $$
(M^{A,\square,\co})^{B,\square,\co} = M^{A\times B,\square,\co}\;,\; (M^{A,\square,\un})^{B,\square,\un} = M^{A\times B,\square,\un} \;.
 $$
\end{notation}
 \begin{rmk} \label{dirsumdescr} Let $k \in \cCRu$ and $M \in \cCLMu_k$. We  describe  explicitly the uniform  box products (resp. direct sums)  appearing in \eqref{cuboxprod} (resp. in  \eqref{cudirsum}) and defined in \eqref{unifsqprod}  (resp. in part $\mathit 2$ of  Proposition~\ref{complsum2}). 
\ben \item For  any set $A$,  $M^{A,\square,\un}$ (resp. $M^{(A,\un)}$) is the set of functions $x : A \to M$, $x = (x_\alpha)_{\alpha \in A}$ (resp.  such that, for any $J \in \cP(k)$, and any family $(P_\alpha)_{\alpha \in A} \in  \cP(M)^A$, $x_\alpha \in P_\alpha +J M$ for all but a finite number of $\alpha$). 
A basis of open $k$-submodules of $M^{A,\square,\un}$ (resp. $M^{(A,\un)}$) consists of the family 
$$\{U((P_\alpha)_\alpha,J)\}_{\cP(M)^A \times \cP(k)}
$$ 
where $U((P_\alpha)_\alpha,J)$ is defined in \eqref{usquare1} (resp. 
$$\l\{U((P_\alpha)_\alpha,J) \cap M^{(A,\un)}  = \l(\PROD_\alpha (P_\alpha + JM)\r) \cap M^{(A,\un)}  \r\}_{\cP(M)^A \times \cP(k)} \;\mbox{)}\;.
$$ 
\item 
If $M$ is pseudocanonical, so that $\{ \ol{IM}\}_{I \in \cP(k)}$ is a basis of open $k$-submodules of $M$,  then, in the description of the previous point 1 for any given $J \in \cP(k)$ we may take  $P_\alpha =  \ol{JM}$ for any $\alpha$, 
so that 
$$ U((P_\alpha)_\alpha,J) =  (\ol{JM})^A = \ol{JM^{A,\square,\un}}
$$ 
and $M^{(A,\un)}$ is the set of $(x_\alpha)_\alpha \in M^A$   such that, 
  for any $J \in \cP(k)$,  $x_\alpha \in \cl_{M}(JM)$ for almost all  $\alpha$'s.  Moreover, 
\beq \label{dirsumdescr1}
(\ol{JM})^A \cap M^{(A,\un)} =  \cl_{M^{A,\square,\un}}(J M^{(A,\un)}) =  
   \cl_{M^{(A,\un)}}(J M^{(A,\un)})\;.
\eeq
Therefore, if $M$ is pseudocanonical,  $M^{A,\square,\un}$ (resp. $M^{(A,\un)}$) is pseudocanonical, as well.
\item
If  $M$ is pseudocanonical 
 $M^{A,\square,\un}$ (resp. $M^{(A,\un)}$) is the set of functions $x : A \to M$ (resp. which tend to $0$ along the filter of cofinite subsets of $A$) equipped with the topology of uniform convergence on $A$. 
Equivalently,   we may  identify  $M^{(A,\un)}$ with the completion of the $k$-module of functions $x: A \to M$ with finite support in the topology of uniform convergence on $A$. 
 \item The discussion of the previous point $2$ of this remark applies in particular to $M=k$ which is always pseudocanonical. So, for any (small) set $A$, $k^{(A,\un)}$ is pseudocanonical. When $k$ is clop, this also applies to $I \in \cP(k)$, equipped with the subspace topology of $I \subset k$,  because in this case $\{\ol{JI}\}_{J \in \cP(k)}$ is a filter basis of $\cP(I)$, \ie  $I$ is a pseudocanonical  or, equivalently, a clop $k$-module (see 3 of Remark~\ref{barrpscan}).     
 \een
\end{rmk}  
\begin{prop} \label{pscansumbox} Let $A$ be any set, $k \in  \cCRu$, $M \in \cCLMpscan_k$. Then~: \ben
\item
$$M^{A,\square, \un} = \nwhat{M^A} = (M^A)^\pscan \;,$$
\item
$$ M^{(A,\un)} = \nwhat{M^{(A)}} = \what{(M^{(A)})^\naive}\;,$$
\een 
and the natural morphism $M^{(A,\un)} \map{} M^{A,\square, \un}$ is a closed embedding.
\end{prop}
\begin{proof} $\mathit 1.$ Is easily deduced from the description of  $\prod^{\square,\un}_{\alpha \in A} M_\alpha$ following  \eqref{unifsqprod}, taking into account the fact that 
$M_\alpha = M$ is 
pseudocanonical for any $\alpha \in A$. 
\par $\mathit 2.$ Follows directly from the description of ${\SUM}^{\cCLMu_k}_\alpha M_\alpha$ in 
$\mathit 2$ of Proposition~\ref{complsum2}, again taking into account the fact that $M_\alpha = M$  
for any $\alpha \in A$. 
\par The last part of the statement is clear. 
\end{proof}
\begin{rmk} \label{smallness} Let $k$ be in $\cCRu$. 
It follows from Lemma~\ref{complsum} that for any family $\{M_\alpha\}_{\alpha \in A}$ in $\cCLMc_k$
$$
\Hom_{\cLM_k} (k,  {\SUM}^\co_\alpha M_\alpha) = {\SUM}^{\Mod_k}_\alpha M_\alpha^\for \;.
$$
So, the object $k$ of $\cCLMc_k$ is \emph{small} in the sense of \cite[Defn. 2.1.1 (a)]{schneiders}. Notice that $k$ is not a small object of $\cCLMu_k$, in general.
\end{rmk}
\begin{exa}  \label{comparBBK} Let $K$ be a non-archimedean non trivially valued field, and let $k = K^\circ$. The category of locally convex $K$-vector spaces \cite{schneider} is the full subcategory of $K$-modules of $\cLMc_k$.
Let $\Ban_K$ (resp. $\Ban^{\leq 1}_K$) be the category of $K$-Banach spaces and bounded (resp. contractive) morphisms \cite[Appendix A]{BBK}.  Then $\Ban_K$ 
is a full subcategory of $\cCLMc_k$. The map 
$$(M,||~||) \longmapsto \beta(M,||~||) := \{x \in M \, :\, ||x|| \leq  1\}$$
induces a fully faithful functor ``unit ball''
$$\beta: \Ban^{\leq 1}_K \longrightarrow \cCLMu_k\;.$$  
Let $\{M_\alpha\}_{\alpha \in A}$ be a family of locally convex $K$-vector spaces. 
\ben
\item
${\SUM}_{\alpha \in A}^\co M_\alpha = {\SUM}_{\alpha \in A}^{\cLMc_k} M_\alpha$ coincides with the locally convex direct sum of \cite{schneider}. In particular, part 1 of Proposition~\ref{complsum2}  is a generalization of  Lemma 7.8 of Chapter I of  \lc.  
\item $\Ban_K$ is quasi-abelian but has no infinite product or coproduct of non-zero objects (Lemma A.26 of \cite{BBK}).
\item $\Ban^{\leq 1}_K$ is quasi-abelian bicomplete (Appendix A.4 of \lc) and the unit ball functor $\beta$ commutes with limits and colimits.
\item Assume $M_\alpha \in \Ban_K$ for any $\alpha \in A$. Then the product and coproduct of $\{M_\alpha\}_{\alpha \in A}$ in $\LMc_k$ exist and are complete 
but, unless $M_\alpha = (0)$ $\forall \,\forall \alpha \in A$, they  are not in the essential image of  $\Ban_K$ in $\LMc_k$.
\een
\end{exa}
\end{subsection}
\begin{subsection}{Strict inductive limits} \label{strictindlim}   
We slightly generalize the notion of strict inductive limit of  \cite[E2]{schneider}.
\begin{defn} \label{strictindlimdef}   
A filtered inductive system $(M_\alpha)_{\alpha \in A}$ in $\cLMc_k$ 
and its colimit in $\cLMc_k$, 
are  \emph{topologically strict} 
(resp. \emph{strictly closed}, resp.  \emph{strictly open}) if, for any $\alpha \leq \beta$, the morphism $\pi_{\alpha,\beta}:M_\alpha \to M_\beta$ is a topological embedding 
(resp. a closed embedding, resp. an open embedding) in $\cLMc_k$.  
An  inductive system as above is \emph{countable} if there exists a cofinal increasing map $(\N,\leq) \to (A,\leq)$. 
\end{defn}
\begin{prop} \label{indclosed} 
Let $\{V_n\}_{n \in \N}$  be a countable topologically strict  inductive system in  $\cLMc_k$ 
and let 
$$
V :=  \colimit^{\cLMc_k}_n V_n  \;.
$$
Then,     
\hfill \ben
\item the canonical $\cLMc_k$-morphisms $j_n : V_n \hookrightarrow V$ are topological embeddings; 
\item if $V_n$ is separated for any $n$, then $V$ is separated;
\item if the morphisms $V_n \to V_{n+1}$ are closed embeddings for any $n$, then so are 
the morphisms $j_n : V_n \hookrightarrow V$ for any $n$;
\item if all $V_n$ are complete, $V$ is complete, so that it coincides with
$$ \colimit^{\cCLMc_k}_n V_n   \; .
$$ 
\een
\end{prop}
\begin{proof}  
The reader might follow  word-by-word  
the arguments of \cite[Prop. 5.5 and Lemma 7.9]{schneider}. 
For her/his convenience, we prefer to render the arguments of \lc in our notation. 
\par \smallskip $\mathit 1.$ Fix an $n \in \N$; we first show that $j_n$ is an embedding. So, let  $L_n \in \cP_k(V_n)$. Since 
$V_m \hookrightarrow V_{m+1}$ is an  embedding, for any $m$, we inductively determine  $L_{n+m} \in \cP_k(V_{n+m})$ such that $L_{n+m+1} \cap V_{n+m} = L_{n+m}$, for any $m \in \N$. So, $L:= \bigcup_m  L_{n+m} = \sum_m  L_{n+m}$ is a $k$-submodule of $V$, and it is open because $L \cap V_{n+m} = L_{n+m}$, for any $m \in \N$. 
\par \smallskip $\mathit 2.$ We need to show that for any nonzero element $v \in V$ there exists $L \in \cP(V)$ such that $v \notin L$. The construction of $L$ is as in the previous point (see \cite[Prop. 5.5 $ii)$]{schneider}).
\par \smallskip $\mathit 3.$  We now show that $j_n$ is closed. So, let $v \in V \setminus V_n$. We have $v \in V_m$, for
some $m >n$. Since $V_n$ is closed in $V_m$ by assumption, we find  
$L_m \in \cP_k(V_m)$ such that $ (v + L_m) \cap V_n = \emptyset$.
  Applying the previous inductive construction again, there is  $L \in \cP_k(V)$ such that 
  $L_m = V_m \cap L$.   It follows
that $(v +L) \cap V_n = ((v +L) \cap V_m)  \cap V_n = (v +L_m) \cap V_n =\emptyset$.  
\par \smallskip $\mathit 4.$ 
Let $\{v_i\}_{i\in I}$ be a Cauchy net in $V$. In a first step we show that there is an
$m \in \N$ such that for any $i \in I$ and any $L \in \cP(V)$ there is a $j \geq i$ such
that $v_j \in V_m + L$. Assume that, for any $h \in \N$, there is  $L_h \in \cP(V)$  
and an $i(h) \in I$ such that
$$v_j \notin V_h + L_h \;\;,\;\; \forall \; j \geq i(h) \;. $$
We certainly may assume that the $L_1 \supseteq L_2 \supseteq \dots$ are decreasing. Consider 
$$ L := \sum_{n \in \N} V_n \cap L_n \in  \cP(V) \;.$$
\par We claim that $L \subseteq V_h +L_h$, for any $h \in \N$. 
  It suffices to show that $V_n \cap L_n \subseteq V_h + L_h$ for any $n$ and $h$.
  But if $n \leq h$ then $V_n \subseteq V_h$, and if $n \geq h$ then $L_n \subseteq L_h$. 
 It follows that $V_h + L \subseteq V_h + L_h$ for any $h \in \N$. 
   Choose now an index $i \in I$ such that $v_{i_1} - v_{i_2} \in L$ for any $i_1, i_2 \geq i$.
Letting $h \in \N$ be such that $v_i \in V_h$ we arrive 
 at the contradiction that $v_j \in V_h + L \subseteq V_h + L_h$ for any $j \geq i$. This proves the claim and therefore the existence of $m \in \N$ as described above.
 \par \smallskip
We introduce the set $I \times \cP(V)$ directed by the partial order $(i,L) \leq (j,P)$ if $i \leq j$ and $P \subseteq L$. Fix a
natural number $m$ with the property which we have established above. For any
pair $(i,L) \in I \times \cP(V)$ we then have an index $i'(L) \geq i$ and an element $v(i,L) \in V_m$
  such that $v(i,L) - v_{i'(L)} \in L$.
  \par \medskip
In the next step we show that  $\{v(i,L)\}_{(i,L) \in I \times \cP(V)}$ in fact is a Cauchy net in $V_m$. 
  Any
$L_m \in \cP(V_m)$  is of the form $L_m =  V_m \cap L$, for some $L \in \cP(V)$. Fix an
 $i \in I$  such that $v_h - v_\ell \in L$ for any $h,\ell \geq i$.
 Consider now any two pairs $(h,P), (\ell,Q) \geq (i,L)$.
We have $v(h,P) - v_{h'(P)} \in P$ and $v(\ell,Q) - v_{\ell'(Q)} \in Q$. Since $h'(P) \geq h \geq i$, 
 $\ell'(Q) \geq \ell \geq i$, and $P +Q \subseteq L$ we obtain
 $v(h,P) - v(\ell,Q) = (v(h,P) - v_{h'(P)}) + (v_{h'(P)} - v_{\ell'(Q)}) + (v_{\ell'(Q)} - v(\ell,Q))
\in P + L + Q \subseteq L$.  
Since $V_m$ by assumption is complete the Cauchy net 
$\{v(i,L)\}_{(i,L) \in I \times \cP(V)}$
  converges to some element $v \in V_m$.
    \par \medskip
 We conclude the proof by showing that the original Cauchy
net $\{v_i\}_{i \in I}$ also converges to $v$. So, let us pick $L \in \cP(V)$; 
we find a pair $(h,P) \in I \times \cP(V)$ such that 
\ben
\item $v(\ell,Q) -v \in L$ for any $(\ell,Q) \geq (h,P)$, and
\item $v_{h_1} - v_{h_2} \in L$ for any $h_1, h_2 \geq h$.
\een  
As a special case of $(1)$ we have $v(h,P \cap L) - v \in L$. 
 Since, by construction, $v(h,P \cap L) - v_{h' (P \cap L)} \in P \cap L$
it follows that $v_{h'(P \cap L)} - v \in L$. 
  Using $(2)$ we
finally obtain that 
  $v_\ell - v \in L$, for any $\ell \geq h$.  
\end{proof}
\begin{rmk} \label{indopen} \hfill \ben
\item 
Let $(M_\alpha)_\alpha$ be a  strictly open inductive system
in $\cCLMc_k$ and let $M := \colimit^{\cLMc_k}_\alpha M_\alpha$. Then  
 $M$ is complete and the canonical morphisms $j_\alpha : M_\alpha \to M$ are open embeddings in $\cCLMc_k$.
So, any $M_\alpha$ may be identified with its image in $M$; a $k$-submodule $P \subset M$ is open if and only if $P \cap M_\alpha$ is open in $M_\alpha$, for any $\alpha$.  
\item 
Consider the strictly open inductive system in $\cCLMu_{\Z_p}$
$$
(\Z_p,p\cdot) := \Z_p \map{p \cdot} \Z_p \map{p \cdot} \Z_p \map{p \cdot} \dots \; ,
$$
where $p \cdot : \Z_p \to \Z_p$ is multiplication by $p$.  
Then  
$$
\colimit^{\cLMc_{\Z_p}} \, (\Z_p,p\cdot) = (\Q_p, \mbox{$p$-adic})
$$
while 
$$
\colimit^{\cLMu_{\Z_p}} \, (\Z_p,p\cdot) = (\Q_p, \mbox{trivial}) \;.
$$
So
$$\colimit^{\cCLMc_{\Z_p}} \, (\Z_p,p\cdot) = \colimit^{\cLMc_{\Z_p}} \, (\Z_p,p\cdot) = (\Q_p, \mbox{$p$-adic})\;,$$
as predicted by part $1$ of this Remark, while 
$$\colimit^{\cCLMu_{\Z_p}}  (\Z_p,p\cdot) = (0) \;.
$$
This shows that parts 1, 2 and 3 of Proposition~\ref{indclosed} fail if 
  $\{V_n\}_{n \in \N}$ is a strictly closed (or even strictly open) inductive system in  $\cCLMu_k$ but the colimit is taken 
  in $\cLMu_k$. 
\een
\end{rmk}

\end{subsection}

\begin{subsection}{Structure of continuous complete modules}
In this subsection $k$ is an object of  $\cCRu$.
\begin{prop} \label{prodiscrete} 
 Any object $M$ in $\cCLM_{k}$ is a projective limit in $\cLM_{k}$
 of a cofiltered projective system of discrete $k$-modules and surjections
 \beq \label{prodiscrete-c}
 M  = \limit_{P \in \cP_k(M)} M/P
 \;.
 \eeq 
 \ben
\item $M$ is an object of $\cCLMc_{k}$ iff, for any $P \in \cP_k(M)$, 
$M/P$ is a filtered inductive limit in $\cLM_{k}$
 \beq \label{prodiscrete-cc0}
M/P = \colimit_{I \in \cP(k)} (M/P)_{[I]}
 \;,
 \eeq
  where $(M/P)_{[I]}$ is defined in Remark~\ref{discrete-mod} and is a  discrete $k/I$-module. 
So, $M$ is an object of $\cCLMc_{k}$ if and only if
 \beq \label{prodiscrete-cc}
M = \limit_{P \in \cP_k(M)}\colimit_{I \in \cP(k)} (M/P)_{[I]}
 \eeq
 where limits and colimits are taken in $\cLM_k$. 
\item 
 $M$ is an object of $\cCLMu_{k}$ iff  
 there exists a filter basis $\cP$ of $\cP_k(M)$ and an increasing function $\cP \to \cP(k)$, $P \mapsto I_P$, such that 
 $\{I_P\}_{P\in \cP}$ is a basis of open ideals of $k$, and $I_P M \subset P$ \ie $M/P$ in \eqref{prodiscrete-c} is a  discrete $k/I_P$-module.   So, $M$ is an object of $\cCLMu_{k}$ if and only if there is $I: \cP \map{} \cP(k)$, $P \longmapsto I_P$, as before, such that 
 \beq \label{prodiscrete-uu}
M = \limit_{P \in \cP}  M/P 
 \eeq
 where $M/P$ is a discrete $k/I_P$-module and the limit  is taken in $\cLM_k$. 
 \item For any $I \in \cP(k)$, let $M_I$  be a discrete  faithful $k/I$-module and, for any $J \leq I$, let 
$\pi_{I,J}: M_J \to M_I$ be a surjective  morphism of modules over the morphism of rings $k/J \to k/I$. Let  
\beq \label{canrepr}
M :=  \limit_{I \in \cP(k)} M_I \;,
\eeq
where the limit is taken in $\cLM_k$. Then the kernel  of the projection $\pi_I: M \longrightarrow M_I$ is $\ol{IM}$ and  $M \in \cCLMpscan_k$. Conversely, any $M \in \cCLMpscan_k$ admits a representation of the form \eqref{canrepr} where, for any $I \in \cP(k)$, $M_I = M/\ol{IM}$.  
 \item If $k$ is in $\cCRuclop$ then $M$ is an object of $\cCLMuclop$ if and only if it admits a representation 
 \eqref{canrepr}. 
 \een \end{prop} 
\begin{proof} 
The first two assertions  follow  from Remark~\ref{discrete-mod}. For $\mathit 3$, let $M$ be as in \eqref{canrepr}. Then, for any open ideal $J \leq I$, we have the exact sequences of $k/J$-modules
$$
0 \longrightarrow (I/J) M_J  \longrightarrow M_J  \longrightarrow M_I  \longrightarrow 0 \;.
$$
Taking limits for $J \leq I$ we get the first equality in  
$$
\ker \pi_I =  \limit_{J \leq I} (I/J)  M_J = \ol{IM}
$$
where the second equality follows from part $\mathit 2$ of 
Lemma~\ref{clopchar0}. Conversely, if $M \in \cCLMpscan_k$, \eqref{canrepr} holds with $M_I = M/\ol{IM}$. 
Finally, the last assertion follows from the previous point together with 3 of Remark~\ref{barrpscan}.
\end{proof}  

\begin{rmk}
\label{limitrepn1} 
  Not all uniform modules are of the form \eqref{canrepr}. For example, if $k$ is discrete, then any $M$ of the form 
 \eqref{canrepr} is discrete, as well.  The direct product $M$ in $\cLM_k$ of an infinite family of copies of $k$ is a non-discrete object $M$ of $\cCLMu_k$ which is therefore not of the form \eqref{canrepr}. Of course any $M$ is a projective limit of the form \eqref{prodiscrete-c}.  
 \end{rmk}
\begin{defn}\label{pro-flat} An  object  $M$ of $\cCLMu_k$   is  \emph{pro-flat}  if it is pseudocanonical, \ie of the form     
\eqref{canrepr},  where, for any open ideal $I$ of $k$, $M_I$ is a  flat  $k/I$-module. 
\end{defn} 
\begin{cor} \label{hom-prodiscrete} 
Let $M, N$ be objects of 
$\cCLMc_k$ and let $\cP_k(M)$, $\cP_k(N)$ be 
fundamental systems of open $k$-submodules in $M$ and $N$, respectively. Then, in $\Mod_k$,
\beq \label{modhom-c}
\Hom_{\cCLMc_k}(M,N) = 
\limit_{Q \in \cP_k(N)}  \colimit_{P \in \cP_k(M)} \limit_{I \in \cP(k)} 
\Hom_{k/I}((M/P)_{[I]},(N/Q)_{[I]}) 
\eeq
as a $k^\for$-module. Notice that 
$$ 
\Hom_k(M/P,N/Q)=\limit_{I \in \cP(k)} 
\Hom_{k/I}((M/P)_{[I]},(N/Q)_{[I]}) \;, 
$$ 
where for any $I\subseteq J$ 
we have inclusions $(M/P)_{[J]}\subseteq (M/P)_{[I]}$ and 
the morphism of the projective system is the 
restriction map 
$$
\Hom_{k/I}((M/P)_{[I]},(N/Q)_{[I]})\longrightarrow \Hom_{k/J}((M/P)_{[J]},(N/Q)_{[J]})\;. 
$$
\endgraf 
Let $M, N$ be objects of 
$\cCLMu_k$ and let $\cP_k(M)$, $\cP_k(N)$ be 
fundamental systems of open $k$-submodules in $M$ and $N$, respectively. Then 
\beq \label{modhom-u} \begin{split}
\Hom_{\cCLMu_k}&(M,N)   =  
\limit_{Q \in \cP_k(N)}  \colimit_{P \in \cP_k(M)} 
\Hom_k(M/P,N/Q) = \\ &\limit_{Q \in \cP_k(N)}  \colimit_{P \in \cP_k(M)} 
\Hom_{k/J_P}(M/P,(N/Q)_{[J_P]}) \;,\end{split}
\eeq 
as a $k^\for$-module, where  $J_P \in \cP(k)$ is the annihilator of $M/P$ in $k$. 
\end{cor} 
\end{subsection} 
\end{section} 
\begin{section}{Categories of topological modules}  \label{cattop}
We keep the assumptions on $k$ made in the previous two sections, but we will indicate further requirements according to 
our needs. 
\begin{subsection}{Quasi-abelian categories of topological modules} \label{quasab}
\begin{rmk} \label{coker-coim} 
Let  $\ast\in\{ {\varnothing, \rm c,u} \}$ and let $f$ be a morphism in one of the categories $\cLM^\ast_k$. 
Then $\Ker(f)$ (resp. $\Coker(f)$) is
$\Ker(f^\for)$ (resp. $\Coker(f^\for)$) 
endowed with the subspace (resp. quotient) topology of the source (resp. of the target).  
\endgraf 
For a morphism $f$ in   the category $\cSLM^\ast_k$, 
$\Ker(f)$ is calculated in $\cLM^\ast_k$, while  
$\Coker(f)$ is obtained from the cokernel of $f$  in $\cLM^\ast_k$
by application of the separation functor
$(-)^\sep$.   For a morphism $f:M \to N$ in   $\cCLM^\ast_k$ 
we have that $\Ker(f)=\Ker(f^\for)$ with the induced topology, 
$\Coker(f)=\what{\Coker(f^\for)}$ (separated completion of $\Coker(f^\for)$ with the quotient topology),
$\Im(f)=\ol{\Im(f^\for)}$ (closure in $N$ with the induced topology, 
which is isomorphic to the completion of $\Im(f^\for)$ with the topology induced by $N$), 
$\Coim(f)=\what{\Coim(f^\for)}$ (completion of $\Coim(f^\for)$ equipped with the quotient topology). 
\end{rmk}
\begin{rmk} \label{coker-coim} Let again $\ast\in\{ {\varnothing, \rm c,u} \}$ and 
let $f:M\to N$ be a morphism in $\cLM^\ast_k$ (resp. in $\cSLM^\ast_k$). 
Then $f$  is a kernel if and only if it is 
an embedding (resp. a closed embedding)
while $f$ is a cokernel if and only if it is surjective and $N$ has the quotient topology 
(that is, $f$ is surjective and open). 
\endgraf  
A kernel 
in the category $\cCLM^\ast_k$ is a closed embedding. Let $f:M \longrightarrow N$ be a morphism of $\cCLM^{\ast}_k$ and assume $M \in \cCLM^{\omega,\ast}_k$. Then, 
by Corollary~\ref{quotclop2},   the cokernel of $f$ in $\cLM^{\ast}_k$ is complete, hence coincides with  the cokernel of $f$ in $\cCLM^{\ast}_k$. So, the canonical morphism $\tilde{f}: \Coker (f) \map{} \Im (f)$ in 
$\cCLM^{\ast}_k$ is bijective.  
 \end{rmk} 
The category $\cLM_k$,
as well as its subcategories $\cLMc_k$
and $\cLMu_k$, 
are  additive ($k$-linear, in fact) categories which are bicomplete (\ie  have inductive and projective limits), 
but  are not in general abelian. We have however
\begin{thm} \label{quasiabelian}  Let $k \in \cRu$.
The categories  $\cLM_k$, $\cLMc_k$  
and $\cLMu_k$  are (bicomplete and) quasi-abelian.
\end{thm}
\begin{proof}  By Proposition~\ref{exactcheck}, 
we need to prove that kernels, \ie  embeddings,  are stable under push-out, 
and cokernels, \ie open surjections, are stable under pull-back. 
 \endgraf 
Let $i: N\hookrightarrow  M$ be a  
embedding in $\cLM_k$
and let $f: N\to N'$ be any morphism in the category. 
We construct the push-out square 
\beq \label{pushoutmono}
\begin{tikzcd}[column sep=2.5em, row sep=2.5em] 
N \arrow[hook]{r}{i} \arrow{d}[']{f}
& M  \arrow{d}{f'} 
\\ 
N' \arrow{r}[']{i'}  
& M'
\end{tikzcd}
\eeq
We need to show that $i': N' \hookrightarrow  M'$ is also an embedding.
The square \eqref{pushoutmono} is cocartesian in the category $\Mod_k$, as well. 
We describe 
$i'$ in  $\Mod_k$, and then specify its topology.
The push-out $M':=N'\oplus_NM$ is canonically isomorphic to the cokernel 
 of the morphism $(f,-i): N \to N'\oplus M$ induced by $f$ and $-i$. 
 Let $R = \Im(f,-i)$. 
 Then $M'$ is just the quotient of $N'\oplus M$ modulo $R$, with the quotient topology. 
 A basis of open $k$-submodules of $M'$ consists of the submodules 
 $$ ((P \oplus Q) + R)/R   
 $$
 for $P \in \cP(N')$ and $Q \in \cP(M)$. Now, 
 the morphism $i':N'\to M'$ is injective because it is a monomorphism in 
 the abelian category $\Mod_k$. We have to show that $N'$ carries the weak topology of  $i'$. This is true if, for any $P \in \cP_k(N')$, there exists $Q \in \cP_k(M)$ such that 
 $$
 (i')^{-1} (((P \oplus Q) + R)/R ) \subset P\;.
 $$
 In fact, $N$ carries the weak topology of  $i$ so that there exists $Q \in \cP_k(M)$ such that 
 $f^{-1}(P) \supset i^{-1}(Q)$. In other words, for any $y \in N$, 
 $$
 i(y) \in Q \Rightarrow f(y) \in P \;.
 $$
So, suppose $x \in N'$ is such that $i'(x) \in ((P \oplus Q) + R)/R$. This means that there exists $p \in P$, $q \in Q$, and $y \in N$ such that
$$
(x,0) = (p,q) + (f(y),-i(y))
$$
in $N' \oplus M$. Then $q= i(y)$ implies $f(y) \in P$ and therefore $x = p + f(y) \in P$. 
We conclude that $i'$ is an embedding. 
 \endgraf 
 Let now $p: M\twoheadrightarrow N$ be a cokernel in $\LM_k$, so in particular a surjective map, 
 and let $f: N'\to N$ be any morphism in that category. 
 We construct the pull-back square 
$$ 
\begin{tikzcd}[column sep=2.5em, row sep=2.5em] 
M' \arrow{r}{p'} \arrow{d}[']{f'}
& N'  \arrow{d}{f} 
\\ 
M \arrow[two heads]{r}[']{p}  
& N
\end{tikzcd}
$$ 
Then the   pull-back $M'=M\times_NN'$ is just the pull-back in the 
 abelian category $\Mod_k$ and $p$ is surjective, so that the map $p':M'\to N'$ 
 is surjective. 
 Since $p$ is a cokernel in the category  $\LM_k$, 
 it is an open map. It follows that also $p'$ is open. 
 In fact let $P\times_NQ'$ be an open submodule of $M'$, where $P\in\cP(M)$ and  $Q'\in\cP(N')$. 
 Then $Q=p(P)\in\cP(N)$ and we may restrict the previous pull-back square over $Q$ to get the open sub-diagram, also a pull-back square, 
 $$ 
\begin{tikzcd}[column sep=2.5em, row sep=2.5em] 
P \times_Qf^{-1}(Q) \arrow{r}{p'_{Q}} \arrow{d}[']{f'_{Q}}
& f^{-1}(Q) \arrow{d}{f_{Q}} 
\\ 
P \arrow[two heads]{r}[']{p_{Q}}  
& Q
\end{tikzcd}
$$ 
where $P \times_Qf^{-1}(Q)$ contains the open $k$-submodule 
$$(P \times_Qf^{-1}(Q)) \cap (P\times_NQ' )= P\times_Q(Q' \cap f^{-1}(Q))
$$  
 whose image via $p'_Q$ is $Q'\cap f^{-1}(Q)$. The latter coincides with the image $p'(P\times_NQ')$ in $N'$, so that we have proven that the latter is  an open submodule of $N'$. 
 Now, since $p'$ is surjective and open, $N'$ carries the quotient topology of $M'$, 
 and $p'$ is a strict epi. 
 \endgraf 
 Since the categories $\cLMc_k$ and $\cLMu_k$ are full subcategories of $\cLM_k$ 
 stable under finite limits and colimits, the result follows. 
 \end{proof}

\begin{thm} \label{quasiabelian-sep} Let $k \in \cSRu$. 
The categories  $\cSLM_k$, $\cSLMc_k$  
and $\cSLMu_k$ are (bicomplete and) quasi-abelian. 
\end{thm}
\begin{proof} The only difference with respect to the proof of the previous theorem is in the first part, where $i': N' \hookrightarrow  M'$ is shown to be an embedding. We only discuss the case of $\cSLM_k$.
The square \eqref{pushoutmono} is not any more cocartesian in the category $\Mod_k$.
The push-out $M':=N'\oplus_NM$ is   the cokernel 
 of the morphism $(f,-i): N \to N'\oplus M$ induced by $f$ and $-i$, in the category  $\cSLM_k$. 
So, $M'$ is set-theoretically the quotient of  $N'\oplus M$ modulo  the \emph{closure} $R$ of the set-theoretic image of the morphism $(f,-i)$, and carries the quotient topology.
The morphism $i':N'\to M'$ is injective because if   $i'(y) = 0$, then $(y,0) \in N' \oplus M$ belongs to $R$. So, 
there exists a net $(x_\alpha)_\alpha$ in $N$ such that $i(x_\alpha)_\alpha$ converges to $0$ in $M$, while $(f(x_\alpha))_\alpha$  converges to $y$. Since the topology of $N$ is its subspace topology in $M$, we deduce that  $(x_\alpha)_\alpha$ converges to $0$ in $N$, so that $y=0$. 
 Then as in the proof of the previous theorem $M'$ is just the quotient of $N'\oplus M$ modulo $R$, with the quotient topology. 
 A basis of open $k$-submodules of $M'$ consists of the submodules 
 $$ ((P \oplus Q) + R)/R   
 $$
 for $P \in \cP(N')$ and $Q \in \cP(M)$. From the injectivity of $i'$, we still have 
 $$
 (i')^{-1} (((P \oplus Q) + R)/R ) = P\;.
 $$
 So,   $i'$ is an embedding. 
\end{proof}
 \begin{rmk} \label{strictmonopushout} It is instructive to observe that the proof of stability of kernels by push-outs in the case of Theorem~\ref{quasiabelian} 
makes crucial use of the fact that in diagram~\ref{pushoutmono} 
the topology of $N$ is the relative topology as a subspace of $M$.  In the following the topology of the kernel of a morphism $f$ will in general be finer than its relative topology as a submodule of the source of $f$, but we will take advantage of other features of our categories.
\end{rmk}

In the subcategories of complete modules the main problem is the following: 
in general cokernels are not open surjective maps, 
and the previous argument does not work. 
But if the topology admits a countable basis of open submodules, 
then in fact, by Corollary~\ref{quotclop2} and Remark~\ref{coker-coim}, cokernels are open surjective maps, 
and the above proof gives the first part of the following  
\begin{thm} \label{quasiabelian-compl}  
Let $k$ be an object of $\cCRou$. 
The categories  $\cCLMoc_k$  
 and $\cCLMou_k$ are quasi-abelian and have countable limits.   The category $\cCLMou_k$ has enough injectives. 
\end{thm} 
\begin{proof}
 \par \medskip
 We are left to prove that $\cCLMou_k$ has enough injectives. We  recall that an object $J$ of $\cCLMou_k$ is   \emph{injective} if, for any   strictly closed subobject   $X \hookrightarrow Y$,  
 and any morphism $g: X \longrightarrow J$, there exists a morphism $\ell: Y  \longrightarrow J$ such that the diagram
 \beq \label{injdef}
\begin{tikzcd}[column sep=2.5em, row sep=2.5em] 
X \arrow[hook]{r}{} \arrow{d}[']{g}  
&Y \arrow[dashed]{dl}{\exists \, \ell} 
\\ 
J       
& 
\end{tikzcd}
\eeq
commutes. 
 So, let $M$ be an  object  of $\cCLMou_k$, and let 
 $(P_n)_{n \in \N}$, with $\dots \supset P_n \supset P_{n+1} \supset \dots$ be a sequence of open $k$-submodules of $M$ which is a filter basis of  $\cP_k(M)$. Since $M$ is uniform, for any  $n \in \N$, there exists a decreasing basis  of open ideals  $I_n \in \cP(k)$ with $I_n \supseteq I_{n+1}$ for any $n$, such that $I_n M \subset P_n$. So, $M = \limit_n M_n$, where $M_n := M/P_n$ is a discrete $k/I_n$-module. By \cite[Thm. 1.10.1]{To} for any $n$ there exists an injective $k/I_n$-module $J_n$ and a $k/I_n$-linear monomorphism $M_n \to J_n$. We then have commutative squares 
 
\beq \label{injres}
\begin{tikzcd}[column sep=2.5em, row sep=2.5em] 
M_{n+1} \arrow[hook]{r}{i_{n+1}} \arrow{d}[']{} \arrow{dr}[']{}
&J_{n+1} \arrow[dashed]{d}{\exists } 
\\ 
M_n   \arrow[hook]{r}{i_n}   
& J_n
\end{tikzcd}
\eeq
 We set $J := \limit_n J_n$, and call $j_n : J \longrightarrow J_n$ the projection.  We have 
\begin{lemma}  \label{quasiabelian-compl1}  $J$ is an injective object of $\cCLMou_k$, and  the canonical morphism 
 $$i:=\limit_n i_n : M \longrightarrow J
 $$ 
 is a  closed embedding (hence a kernel) in $\cCLMou_k$. 
 \end{lemma}
 \begin{proof} We prove the second part of the statement first. Since limits in $\Ab$ are left exact, $i$ is a 
 monomorphism. Then $i(M)$ is closed in $J$, since it coincides with $\bigcap_n j_n^{-1}(i_n(M_n))$, and $j_n^{-1}(i_n(M_n))$ is an open $k$-submodule of $J$, for any $n$. Then $i$ is a homeomorphism of $M$ onto $i(M)$ equipped with the relative topology of the inclusion in $J$ because, for any $n$, 
 $$P_n = \Ker (M \to M_n) = \Ker(j_n) \cap M\;.$$
 We now show that  $J$ is  injective. So, we consider diagram~\ref{injdef} where $X$ is a strictly closed subobject of $Y$, and, for any $n$, we let $P_n := \Ker (j_n \circ g)$. Since the topology of $X$ is the subspace topology of the inclusion in $Y$, for any $n$ there is an open $k$-submodule $Q_n$ of $Y$ such that $X \cap Q_n = P_n$. We let $X_n := X/P_n$ and $Y_n := Y/Q_n$. Since, for any $n$, $J_n$ is an injective $k/I_n$-module, 
 we inductively obtain a sequence of  commutative diagrams, 
  \beq \label{injdefn}
\begin{tikzcd}[column sep=2.5em, row sep=2.5em] 
X_n \arrow[hook]{r}{} \arrow{d}[']{g_n}  
&Y_n \arrow[dashed]{dl}{\exists \, \ell_n} 
\\ 
J_n   & \quad ,
\end{tikzcd}
\eeq
such that the natural projections of triangles 
  \beq \label{injdefn2}
\begin{tikzcd}[column sep=2.5em, row sep=2.5em] 
X_{n+1} \arrow[hook]{r}{} \arrow{d}[']{g_{n+1}}  
&Y_{n+1} \arrow[dashed]{dl}{\exists \, \ell_{n+1}} 
\\ 
J_{n+1}   & 
\end{tikzcd}
\quad
 \longrightarrow \quad \begin{tikzcd}[column sep=2.5em, row sep=2.5em] 
X_n \arrow[hook]{r}{} \arrow{d}[']{g_n}  
&Y_n \arrow[dashed]{dl}{\exists \, \ell_n} 
\\ 
J_n   & 
\end{tikzcd}
\eeq
make commutative prisms. 
Taking limits, we then obtain a morphism 
$$\limit_n \ell_n : \limit_n Y_n \longrightarrow \limit_n J_n = J
$$
and therefore, composing with the morphism $Y  \longrightarrow \limit_n Y_n$, a morphism $\ell$ as required in 
diagram~\ref{injdef}. 
 \end{proof}
  \end{proof}  
  \begin{cor} \label{strong-can-canCOR}   Let $k \in \cCRou$ and $P \in \cCLMou_k$. Then  (see Remark~\ref{coker-coim})  for a morphism $f: P \map{} M$ in $\cLMu_k$,  T.F.A.E.
\ben
\item $f$ is an open surjective map;
\item $f$ is a strict epi in $\cLMu_k$; 
\item $f$ is a strict epi in $\cCLMou_k$;
\item $f$ is a cokernel in $\cCLMu_k$.
\een
\end{cor}
\begin{proof} It is clear that  ${\mathit 1} \Rightarrow {\mathit 2} \Rightarrow {\mathit 3}\Rightarrow {\mathit 4}$. 
Let us check that  ${\mathit 4} \Rightarrow {\mathit 1}$. So, assume  
$f$ is the cokernel of a morphism $g:N \map{} P$ in $\cCLMu_k$.  Then $f$ is the natural morphism $P \to M$, where $M$  is the completion of $P/\ol{g(N)} \in \cLMou_k$.  But then $P/\ol{g(N)}$ is already complete and $f$ is open surjective.  
\end{proof}
\end{subsection} 
\begin{subsection}{Naive canonical topology}
 \label{naivesec}

\begin{defn}  \label{CLMnaive} Let $R$ be an object of $\cRu$. 
We denote by $\cLM^\naive_R$   the full subcategory of objects $\cLMu_R$ whose topology is the naive canonical one. If $R \in \cCRu$, we set  
$$\cCLM^\naive_R := \cLM^\naive_R \cap \cCLMu_R\;,$$
a full subcategory of $ \cCLMu_R$. 
\end{defn} 
\begin{prop} \label{naiveadj} 
The correspondence $M \mapsto M^\naive$ extends to a functor 
$$(-)^\naive : \Mod_R \map{} \cLMu_R
$$ 
which is left adjoint to the forgetful functor $(-)^\for: \cLMu_R \to \Mod_R$. 
\end{prop} 
\begin{proof} Let $N \in \cLMu_R$ and $M \in \Mod_R$. Then, obviously,
$$
\Hom_{\Mod_R}(M, N^\for) = \Hom_{\cLMu_R}(M^\naive, N) \;.
$$
\end{proof}

\begin{rmk} \label{naivermk} Let $\phi:R \to S$ be a morphism in $\cRu$ as in Definition~\ref{clop-adic}.
\hfill \ben 
\item  The morphism $\phi$  makes  $S$ into an object of $\cLMu_R$. In fact, let $J$ be any open ideal 
of $S$. Then, there exists $I \in \cP(R)$ such that $\phi(I) \subset J$, hence $\phi(I)S \subset J$. 
This proves that the naive $R$-canonical topology of $S$ is finer than the topology of $S$. Clearly, $\phi$ 
is $\op$-adic  if and only the naive $R$-canonical topology of $S$ is coincides with the topology of $S$, \ie iff 
$S \in \cLM^\naive_R$. 
\item  
 If $R \in \cRuop$, then, obviously,    
 $\cLM^\naive_R = \cLMuop_R$. If moreover $\phi$  is op-adic, then both $S \in \cLMuop_R$ and $S \in \cRuop$. To prove the last assertion, 
let $J,J_1 \in \cP_S(S)$ and let $I,I_1 \in \cP_R(R)$ be such that $\phi(I)S \subset J$,  $\phi(I_1)S \subset J_1$.  Then $JJ_1\supset \phi(I)\phi(I_1)S = \phi(I_1I)S$ is an open ideal of $S$.
  \item   If $R \in \cRuclop$, then  $\phi$
is $\clop$-adic if and only if 
$S$ is an object of $\cLMuclop_R$. From point 1 we know  that $S \in \cLMu_R$.
So,   assume $\phi$ is clop-adic and let us prove that $S \in \cLMuclop_R$.
Pick  any open $R$-submodule $H$ of $S$ and  $I  \in \cP_R(R)$. 
Since the topology of $S$ is $S$-linear,  there exists an open ideal $J$ of $S$ such that $J \subset H$. 
Let $I_1$ be an open ideal  of $R$ such that $\phi(I_1) \subset J$. Then,     
$\phi(I)\phi(I_1)S \subset \phi(I)J \subset \phi(I)H$. Therefore, $\phi(\,\ol{I\,I_1}\,) \subset \ol{\phi(I) \phi(I_1)}\subset  \ol{\phi(I)\phi(I_1)S} \subset \ol{\phi(I)H}$ shows that the latter is open. 
Conversely, if $S$ is an object of $\cLMuclop_R$, then, for any open ideal $I$ of $R$, the closure of $\phi(I)S$,  is open in $S$, so that $\phi$ is $\clop$-adic. As in point 2, we prove that, if this is the case, $S$ is an object of  $\cRuclop$.
\item From $1$ and $2$ (resp. $3$) above we deduce that: If $R$ is $\op$ (resp. $\clop$) and $\phi$ 
is $\op$-adic (resp. $\clop$-adic) then $S$ is $\op$ (resp. $\clop$).
\een
\end{rmk}

\begin{cor}\label{cor-adj-naive} For $R \in \cRu$ and any $M \in \cLMu_R$
we will use the notation $M^\naive$ 
as a shortcut for $(M^\for)^\naive$. 
This position defines a functor
\beq \label{cor-adj-naive0}
(-)^\naive : \cLMu_R \map{} \cLM^\naive_R
\eeq
which is right adjoint to the inclusion $\iota: \cLM^\naive_R \hookrightarrow \cLMu_R$
\end{cor}
\begin{proof} 
The  identity of $M^\for$ induces a functorial   $\cLMu_R$-morphism 
\beq \label{naivemap}   M^\naive \longrightarrow M\;.
\eeq 
For any $M \in \cLMu_R$ and $N \in \cLM^\naive_R$ we have a functorial isomorphism
$$
\Hom_{\cLM^\naive_R}( N, M^\naive) = \Hom_{\cLMu_R}(\iota(N), M) \;.
$$
\end{proof}
\begin{defn}  \label{naivedef} Let $R$ be in $\cCRu$.
For  any $N$ in $\Mod_R$, we define the object 
$\nwhat N$ of 
$\cCLMu_R$  
to be the completion of $N^\naive$, 
that is the completion of $N$ in its naive canonical topology, \ie the $R$-module
\beq
\label{naive}
\nwhat{N} = \what{N^\naive}= \limit_{I \in \cP(R)} (N/IN)^\discr \;,
\eeq 
equipped with the weak topology of the projections to the discrete $R/I$-modules $N/IN$.  
The notation also applies  to $N$ a topological $k$-module to mean $\nwhat{N} = \what{N^\naive} = \nwhat{N^\for}$.
\end{defn} 
 \begin{rmk} \label{cloprmk} Let $R \in \cCRu$. It follows from Remark~\ref{compl-adj} that, for  any $R$-module $N$, 
 $\nwhat{N}$ is pseudocanonical. If $R$ is clop, it then follows from  3 of Remark~\ref{barrpscan}  that $\nwhat{N} \in \cCLMuclop_R$. We will later give an example of $N=I \in \cP(R)$, with $R \in \cCRouclop$, such that $\nwhat{I} \notin \cCLMnaive_R$ (see (2) of Remark~\ref{naive-pbl} below).
 \end{rmk}
 In the positive direction we have~:
\begin{lemma} \label{8.3.3} Let $k$ be an object of $\cCRoufop$ and 
let  $M \in \Mod_k$. Then $\nwhat{M} \in \cCLMnaive_k$.  
\end{lemma}
\begin{proof}
See   \cite[Rmk. 8.3.3 ]{GR} or (in a particular case) Lemma~10.96.3 of  \cite[Tag 05GG]{stacks}. 
\end{proof}

\begin{prop} \label{indu2}
Let $k \in \cCRu$ and  $M \in \cCLMpscan_k$. Then, by $\mathit 2$ of  Remark~\ref{dirsumdescr}, 
$M^{(A,\un)}$ is pseudocanonical. Moreover, for any $J \in \cP(k)$, 
\beq \label{indu0}
M^{(A,\un)}/ \cl_{M^{(A,\un)}}(J M^{(A,\un)}) = (M/\ol{JM})^{(A)} \;.
\eeq
Then
\beq \label{indu1}
 M^{(A,\un)} = \limit_{J\in \cP(k)} (M/\ol{JM})^{(A)} \;,
\eeq
 where $M/\ol{JM}$ is discrete and the algebraic direct sum $(M/\ol{JM})^{(A)}$ is also equipped with the discrete topology. 
 \end{prop}
 \begin{proof}
By \eqref{dirsumdescr1},
$$
M^{(A,\un)}/ \cl_{M^{(A,\un)}}(J M^{(A,\un)}) = M^{(A,\un)}/ (\ol{JM})^A \cap M^{(A,\un)} = 
M^{(A)}/ (\ol{JM})^A \cap M^{(A)} = (M/\ol{JM})^{(A)} \;.
$$
Notice that \eqref{indu1} may also be deduced from \eqref{indu} since 
 $$ M^{(A,\un)} = \colimit^\un_{F \in \cF(A)} M^F = \limit_{J \in \cP(k)}  \colimit^\un_{F \in \cF(A)} (M/ \ol{J M})^F =
 $$
 $$
 \limit_{J\in \cP(k)} (M/\ol{JM})^{(A)} \;,
 $$
 where $\cF(A)$ denotes  the set of finite subsets of $A$ and $(M/\ol{JM})^{(A)}$ is equipped with the 
 discrete topology.
 \end{proof}

We summarize our conclusions~:  
\begin{prop} \label{dirsumideal} Let $k \in \cCRuclop$ and $A$ be any (small) set.  Then, $k^{(A,\un)} \in \cCLMuclop_k = \cCLMpscan_k$. For any $I \in \cP(k)$, 
$I^{(A,\un)} = \nwhat{I^{(A)}}$ is the completion of the algebraic direct sum $I^{(A)} = I k^{(A)}$ and of $I k^{(A,\un)}$ in the naive $k$-linear topology. 
 It consists of the elements $x = (x_\alpha)_{\alpha \in A} \in I^A$ 
 such that for any $J \in \cP(k)$, $J \subset I$, $x_\alpha \in J$, for all but a finite number of $\alpha$; it is an open $k$-submodule of $k^{(A,\un)}$ equipped with the subspace topology.  A basis of open $k$-submodules of $k^{(A,\un)}$ is 
 $\{I^{(A,\un)}\}_{I \in \cP(k)}$ and, for any $I \in \cP(k)$, a basis of open $k$-submodules of $I^{(A,\un)}$ is $\{J^{(A,\un)}\}_{J \in \cP(k), J \subset I}$. 
\par
The following objects of 
$\cCLMuclop_k = \cCLMpscan_k$ coincide: 
 \ben
 \item
  $I^{(A,\un)}$;
  \item $\nwhat{I^{(A)}}$;
  \item $\ol{I k^{(A,\un)}}$ endowed with the relative topology of 
 $k^{(A,\un)}$;
 \item the set of functions $A \to I$ which tend to $0$ along the filter of cofinite subsets of $A$   equipped with the topology of uniform convergence on $A$;
 \item the intersection of $I^A$ and  $k^{(A,\un)}$ taken in $k^A$ as a subspace  of $k^{(A,\un)}$.
 \een
 If $k \in \cCRouclop$, then the objects listed above are in $\cCLMpscan_k = \cCLMouclop_k$.
 \end{prop}

\begin{rmk}\label{nonnaivecompl} 
Notice that,  if $k \in \cCRuop$, $k^{(A,\un)}$ is not in general an object of $\cCLMuop_k$. However, it follows from \cite[Rmk. 8.3.3 (iv)]{GR}  
that  if $k \in \cCRoufop$  then $k^{(A,\un)}$ is endowed with the naive canonical topology, hence it is  an object of $\cCLMouop_k = \cCLMnaive_k$.
\end{rmk}  
 \begin{rmk} \label{naive-pbl}  Notice that, for $R$ in $\cRu$, the  naive canonical topology on a $R$-module $N$  runs, in general, 
 into serious difficulties. \ben
 \item Assume $R \in \cCRoufop$ and let $N \in \cCLMnaive_R$.  It  is not true in general that the subspace topology on a closed $R$-submodule of $N$ is still   the naive canonical topology. An example is given, for $R =\Z_p$, by the following inclusion of $\Z_p$-submodules  of the ring of formal power series $\Q_p[[T]]$. Namely, we set 
 $$
M := \Z_p\{T\} =  \{ \sum_{n \in \N} a_n T^n \;|\; a_n \in \Z_p\, ,\, \mbox{s.t.}\,\lim_{n \to \infty} a_n = 0\,\} \;,
 $$
 $$
 N:= \Z_p\{T/p\}  =  \{ \sum_{n \in \N} a_n T^n/p^n \;|\; a_n \in \Z_p\, ,\, \mbox{s.t.}\,\lim_{n \to \infty} a_n = 0\,\}  \;,
 $$
 both equipped with the $p$-adic topology. Then $M,N \in \cCLMnaive_{\Z_p}$ and $M$ is a closed $\Z_p$-submodule of $N$,  
but the inclusion $M \subset N$ is not a topological embedding. In fact,  for any $n \in \N$, $T^n \in M \cap p^n N$, so that $\lim_{n \to \infty} T^n = 0$ in $N$ but not  in $M$.  
\item  
 In general, for $R$ in $\cRu$,  if  the $R$-module $N$ carries the naive canonical topology, and $M$ is a sub-module of $N$, 
 then the topology of $M$ induced by $N$ (with basis of open submodules $\{IN \cap M\}_{I\in\cP(R)}$) 
 is  coarser than the naive topology of $M$ (because $IM\subseteq IN\cap M$). 
 However, if $R$ is in $\cRuop$ and  there exists $J\in\cP(R)$ such that $JN\subseteq M$,   
 then the naive canonical topology of $M$ coincides with the induced topology 
 (because $IM\supseteq IJN$, and $IJ\in\cP(R)$). This condition fails in the example of the previous point.  
In the  next point we  give a similar example in which the previous condition holds (but, of course, $R \notin \cRuop$).
\item  
Assume now $R \in \cCRuclop$. Then any open ideal $I$ of $R$, equipped with the subspace topology, is an object of $\cCLMuclop_R$, 
 but while $R$ is always endowed with its naive canonical topology, 
 the subspace topology of $I \subset R$
 is  in general strictly weaker than the naive canonical topology of $I$ 
 (unless $R \in \cCRuop$). We now give an example of a pair $(R,I)$ with these properties, based on 
the discussion of \cite[Tag 05JA]{stacks}.   We take a field $F$ and 
 $$
 S=F[x_1,x_2,x_3,\dots] \;\;,\;\; J = (x_1,x_2,x_3,\dots)
 $$
 and consider the $J$-adic completion $R := \what{S}_J$ of $S$. Then a basis of open ideals of $R$ consists of 
 $\{\cl_R(J^nR)\}_{n =0,1,2,\dots}$ so that $R \in \cCRouclop$ (since $\cl_R(\cl_R(J^mR)\cl_R(J^nR)) = \cl_R(J^{m+n}R)$). Let $I =  \cl_R(J R)$. Then, by \lc,  $R$ is not $I$-adically   complete. Since, for any $n$, $\cl_R(J^nR) = \cl_R(I^n)$,  the  topology of  $R$   is \emph{strictly weaker} than the $I$-adic topology of $R$. 
Let us show that
\emph{the subspace topology  of $I \subset R$ is strictly weaker than the naive canonical topology of the $R$-module $I$}. In fact, a basis of open $R$-submodules for the former (resp. for the latter) is  $\{\cl_R(J^nR)\}_{n =1,2,\dots}$ (resp. 
$\{\cl_R(JR)\cl_R(J^nR)\}_{n =1,2,\dots}$).  
Assume, by way of contradiction, that for any $n \in \N$, there exists $N(n) \in \N$ such that 
$$\cl_R(J^{N(n)}R) \subset \cl_R(JR)\cl_R(J^nR)\;.$$
 Then, 
$$  \cl_R(J^{N(N(n))}R) \subset \cl_R(JR) \cl_R(J^{N(n)}R) \subset  \cl_R(JR) \cl_R(JR) \cl_R(J^nR) \;,$$
and, by iteration,  for any fixed $n \in \N$,
$$ \cl_R(J^{N^h(n)}R) \subset \cl_R(JR)^h \cl_R(J^nR) \subset \cl_R(JR)^h \;\;,\;\forall \; h =1,2,\dots \;.$$
But this contradicts the fact that the topology of $R$ is strictly weaker than the $I$-adic. We also conclude  that $R$ is an example of an object of $\cCRouclop$ which is not op. 
\item As recalled in Lemma~\ref{8.3.3},  
if  $R \in \cCRoufop$  and $M$ is any $R$-module, then 
 $\nwhat{M} \in \cCLMuop_R =\cCLMnaive_R$. 
  Of course, the assumption that $R$ be an object of 
 $\cCRouclop$  is much weaker than the condition of being  an object of $\cCRoufop$. 
 See also Remark~\ref{clop-rmk}. 
 \een
 \end{rmk}  
\begin{prop} \label{can-forget-adj} Let $R \in \cCRu$
and $M \in \cCLMu_R$. 
Then \hfill \ben \item the  $R$-module $M^\for$ is separated in its naive canonical topology so that the 
natural $R$-linear map 
$M^\for \hookrightarrow (\nwhat{M})^\for$ is injective; 
\item
the completion of the 
canonical morphism \eqref{naivemap}  is  a canonical surjective $\cCLMu_R$-morphism 
\beq \label{can-forget-map}
\sigma_M: \nwhat{M} \longrightarrow M \; ;
\eeq
\item
the functor \eqref{naive} 
$$\nwhat{~}:\Mod_R \rightarrow \cCLMu_R
$$ 
(completion in the naive canonical topology)
is left adjoint to 
the forgetful functor $\cCLMu_R \rightarrow\Mod_R$. 
\item the morphism $\sigma_M$ factors through the canonical  surjective morphism 
\beq \label{coim}
\Coim(\sigma_M) \longrightarrow M
\;,
\eeq
where $\Coim(\sigma_M)$ is taken in the category $\cCLMu_R$.  
\item  Both $\nwhat{M}$ and $\Coim(\sigma_M)$ are in $\cCLMpscan_R$. 
\item 
Let $f: M_1 \longrightarrow  M$ be a bijective morphism in $\cCLMu_R$. Then
 the map $\Coim(\sigma_{M}) \map{} M$ factors through $f$. 
\item
 If $R$ is in $\cCRou$ then both $\nwhat{M}$ and $\Coim(\sigma_M)$ are in $\cCLMou_R$ and 
 the morphism \eqref{coim}  is bijective.  Moreover,  $\Coim (\sigma_M)$ coincides with $M^\for$ 
 equipped  with the finest possible structure  of an object of $\cCLMu_R$ finer than the structure of $M$ itself.
\een
\end{prop} 
\begin{proof} \hfill 
\par  $\mathit 1.$ The naive canonical topology of $M^\for$ is finer than the topology of $M$. 
Therefore, $M^\for$ equipped with the naive canonical topology is separated and 
the canonical morphism  $M^\for \hookrightarrow (\nwhat{M})^\for$ is injective. 
\par   $\mathit 2.$
The existence and continuity of $\sigma_M$ is clear 
($\sigma_M$ is obtained by completion of the identity map $M^\naive\to M$). 
Because of the canonical inclusion, any  $m \in M$ coincides with $\sigma_M(m)$.  
So, $\sigma_M$ is surjective. 
\par $\mathit 3.$
The functor  $\;\nwhat{~}\;$  is the composition of two functors ($(-)^\naive$ and completion) 
which are left adjoints of the corresponding forgetful functors $\cCLMu_R \to \cLMu_R \to \Mod_R$. 
\par $\mathit 4.$ The first part of the statement is obvious. Notice however that if we take the coimage of $\sigma_M$ in the category $\cLMu_R$ the canonical morphism 
$$\Coim^{\cLMu_R}(\sigma_M)\longrightarrow M 
$$
is bijective. The coimage  of $\sigma_M$  in the category $\cCLMu_R$, as in  \eqref{coim},  
is obtained by completion of the previous bijective morphism, and is not necessarily injective.
\par $\mathit 5.$   
A basis of open $R$-submodules of $\nwhat{M}$ consists of $\{\cl_{\nwhat{M}}(J\nwhat{M})\}_{J \in \cP(R)}$. In particular, 
   $\nwhat{M} \in \cCLMpscan_R$.  The coimage  $\Coim^{\cLMu_R}(\sigma_{M})$ of $\sigma_{M}$  in $\cLMu_R$ is  pseudocanonical because the morphism $\nwhat{M} \map{} \Coim^{\cLMu_R}(\sigma_{M})$ is open.  The coimage  $\Coim(\sigma_{M})$ of $\sigma_{M}$  in $\cCLMu_R$ is then the completion of 
$\Coim^{\cLMu_R}(\sigma_{M})$ and is therefore pseudocanonical, as well. 
 \par $\mathit 6.$   For $f: M_1 \longrightarrow  M$ as in the statement, we have a factorization of $\sigma_M$ as 
 $$
 \nwhat{M} =\nwhat{M_1} \map{\sigma_{M_1}} M_1 \map{f} M \;,
 $$
 so that $\Coim(\sigma_{M_1}) \iso \Coim(\sigma_{M})$ and the map $\Coim(\sigma_{M}) \map{} M$ factors through $f$. 
 \par $\mathit 7.$ 
  The kernel of $\sigma_M$ in $\cLMu_k$ is closed   and since 
$\nwhat{M} \in \cLMou_k$ its cokernel in  
$\cLMu_k$  is already complete. We conclude that the latter is the coimage  of $\sigma_M$ in $\cCLMou_k$
 and that the morphism \eqref{coim} is bijective. The last part of the statement follows from the previous point $\mathit 6$.
 \par
\end{proof}

 \begin{rmk}  \label{canbij}\hfill 
 \ben 
 \item
  For $M$ as in Proposition~\ref{can-forget-adj},  the completion $\nwhat{M}$ of the $k$-module $M^\naive$ is not necessarily complete in its naive topology.   So, the sequence  
 $$
 \dots  \map{\sigma_{\nwhat{\nwhat{M}}}}  \nwhat{\nwhat{M}}  \map{\sigma_{\nwhat{M}}}  \nwhat{M} \map{\sigma_M} M  \;,
 $$
 might never stop increasing (although we have no example of this situation).  
   \item  Let $k \in \cCRu$  and $M \in \cCLMu_k$. Since $\nwhat{M} \in \cCLMpscan_k$, the morphism $\sigma_M$ factors as 
 \beq
  \label{canbij0} \Coim(\sigma_{M}) \longrightarrow M^\pscan \map{(1:1)} M
\;,
\eeq
 but we do not know whether the surjective morphism 
 \beq
  \label{canbij1}\Coim(\sigma_{M}) \to M^\pscan
\eeq
 is open or bijective. (It  is a bijection if $k \in \cCRou$ by $\mathit 5$ of Proposition~\ref{can-forget-adj}.)  
 \een
 \end{rmk} 
 \begin{defn}\label{defmax} If $k \in \cCRou$  and $M \in \cCLMu_k$ we set 
$$
M^{\max} := \Coim(\sigma_{M}) \in \cCLMpscan_k \subset  \cCLMou_k \;.
$$
The bijective $\cCLMu_k$-morphism $M^{\max} \map{(1:1)} M$ exhibits $M^\for$ equipped with the finest topology of an object of  $\cCLMu_k$ finer than the topology of $M$. If $M = M^{\max}$ we say that $M$ is \emph{maximally uniform} or simply \emph{maximal}. 
\end{defn} 
The following is the most interesting result of this section. 
\begin{thm}  \label{barrpscan3} Let $k \in \cCRouclop$ and   $M \in \cCLMu_k$. 
\hfill \ben \item The surjective morphism $\nwhat{M} \map{\sigma_M} M$    in \eqref{can-forget-map}  
factors as 
\beq \label{barrpscan311} 
 \nwhat{M} \map{\Coim(\sigma_{M})} M^{\max}   \map{(1:1)}  M^\barrell  = M^\clop  =  M^\pscan  \map{(1:1)} M
\;. \eeq 
\item Assume $k \in \cCRoufop$.   Then  the morphism 
$\nwhat{M} \map{\Coim(\sigma_{M})} M^{\max}$  is an isomorphism so that $M^{\max} = M^\naive$. In particular $M$ is complete in its naive canonical topology. Formula~\ref{barrpscan311} becomes 
\beq \label{naivecompl} 
 \nwhat{M}  = M^\naive = M^{\max} \map{(1:1)}  M^\barrell = M^\clop =  M^\pscan  \map{(1:1)} M
\;.\eeq
In particular, $M^\naive$ is the unique maximal structure of an object of $\cCLMu_k$ on $M^\for$.
\een
\end{thm}
\begin{proof} Part 1 has already been proven. As for part 2, we have seen in Lemma~\ref{8.3.3} that   
 $\nwhat{M}$ is endowed with the naive $k$-canonical topology. On the other hand, the coimage of $\sigma_M: \nwhat{M} \map{} M$ in $\cCLMu_k$ may be calculated in $\cLMu_k$. So, $\Coim(\sigma_{M})$ is an open surjective map. Therefore, the topology of $M^{\max}$ is the 
 naive $k$-canonical one, as well. This means that $M^\naive$ is complete, hence $\sigma_M$ is really an isomorphism. 
\end{proof}
\begin{rmk} \label{8.3.12} Let $k \in \cCRoufop$.
\ben
\item Part $\mathit 2$ of Theorem~\ref{barrpscan3}  may be seen as a generalization of 
Lemma 8.3.12 (b) of \cite{GR} which applies  if $k$ is a $I$-adic ring for a finitely generated ideal $I$ of $k$.
If, moreover,   $M \in \cCLMou_k$, the same result also appears 
as Proposition 0.7.2.5 of \cite{FK}. 
\item By $\mathit 3$ of Proposition~\ref{can-forget-adj},  the functor $\nwhat{~} :\Mod_k \map{} \cCLMu_k$ is left adjoint to the forgetful functor $(-)^\for: \cCLMu_k \map{} \Mod_k$. We do not know how to characterize algebraically the $k$-modules $P$  such that  the unit of the adjunction $\eta (P): P\map{}(\nwhat{P})^\for$ is an isomorphism.  Such $P$ are the objects of a full subcategory of $\Mod_k$ equivalent to the category $\cCLMcan_k$ of the next section. 
\een
\end{rmk} 
 \end{subsection}
 
\begin{subsection}{Canonical modules}  \label{canmodules} 
\emph{We assume in this section} (unless otherwise specified) \emph{that $k$ is in $\cCRou$ so that $\cCLMpscan_k \subset \cCLMou_k$}. We recall (Theorem~\ref{quasiabelian-compl}) that the category $\cCLMou_k$ is quasi-abelian. As observed in Remark~\ref{baire1}, any object of $\cCLMou_k$ is a Baire space. 
\par \smallskip
Recall that, for any $R \in \cCRu$ and any (small) set $A$,  $R^{(A,\un)} \in \cCLMpscan_R$ (see part $4$ of Remark~\ref{dirsumdescr}).  
\begin{defn} \label{strong-can-def} Let $k \in \cCRou$ and 
let $M$ be an object of $\cLMu_k$. We say that $M$ is $k$-\emph{canonical}  or simply   \emph{canonical}   if 
there exist a 
set $A$  and an open surjective morphism (see Corollary~\ref{strong-can-canCOR})
$$F:k^{(A,\un)} \map{} M \;.$$ 
We denote by $\cCLMcan_k$ the full subcategory of $\cLMu_k$ whose objects 
are canonical $k$-modules. 
\end{defn}
\begin{rmk} \label{strong-can-can} \hfill 
\ben
\item  As observed in Corollary~\ref{strong-can-canCOR}, if $M$ is a canonical $k$-module, then $M \in \cCLMou_k$. So, $\cCLMcan_k$ is a full subcategory of the quasi-abelian category  $\cCLMou_k$. 
 \item 
Let us show that a canonical $k$-module is pseudocanonical. We need to show that 
if $M$ is a canonical object of $\cCLMu_k$   the family of submodules of the form 
$\ol{IM}$, for $I \in \cP(k)$, is a basis of open submodules of $M$. 
In fact, let $F:k^{(A,\un)}\to M$ be as in the definition.
Then $F(Ik^{(A,\un)})=IM$ and continuity of $F$ imply that $F(\cl_{k^{(A,\un)}}(I\,k^{(A,\un)})) \subset \ol{IM}$. 
Since $F$ is open   
and $\cl_{k^{(A,\un)}}(I\,k^{(A,\un)})$ is an open submodule of $k^{(A,\un)}$, 
$F(\cl_{k^{(A,\un)}}(I\,k^{(A,\un)}))$ is open. 
We conclude that 
 the submodules  $\ol{IM}$, for $I \in \cP(k)$, are open and  then form a basis of open submodules of $M$. 
 Notice that $F(\cl_{k^{(A,\un)}}(I\,k^{(A,\un)}))$ contains $IM$ and, being open, hence closed, it contains $\ol{IM}$.
 We conclude that, for any $I \in \cP(k)$,
 \beq \label{opensurj}
 F(\cl_{k^{(A,\un)}}(I\,k^{(A,\un)})) = \cl_M(I\,M) \;.
 \eeq 
 \item 
Let $g:M\map{} N$ be a cokernel in $\cCLMu_k$ where $M$ is an object of $\cCLMcan_k$. 
Then $g$ is an open surjective map, and $N$ is canonical. In fact, let $F: k^{(A,\un)} \map{} M$ be as in the definition. Then both $g$ and $gF:k^{(A,\un)} \map{} N$ are strict epimorphisms in $\cLMu_k$ and it follows that $N$ is canonical. 
\item   For $k$ discrete, $\cCLMcan_k$ is  the full subcategory of $\cCLMu_k$ consisting of discrete $k$-modules and therefore coincides with $\cLM^\naive_k = \cCLM^\naive_k$.  In this particular case (see Definition~\ref{topnaivedef}), 
for any $M \in \Mod_k$ the discrete $k$-module $M$ was named  $M^\naive  \in \cLMu_k$. By Proposition~\ref{naiveadj} 
the functor   $(-)^\naive: \Mod_k \map{} \cLMu_k$ is the \emph{left}  adjoint of the forgetful functor 
$(-)^\for :  \cLMu_k \map{} \Mod_k$ and $\cCLMcan_k$ is the essential image of $(-)^\naive$. 
In Corollary~\ref{cor-adj-naive} another functor with the same name $(-)^\naive$ was considered, namely 
\beq \label{rightnaive} (-)^\naive : \cLMu_k \map{} \cLM^\naive_k\;\;,\;\; M \longmapsto M^\naive :=  (M^\for)^\naive \;,
\eeq 
\emph{right} adjoint to the inclusion $\cLM^\naive_k \hookrightarrow \cLMu_k$. 
\een
\end{rmk}    
\begin{notation} \label{S(M)} For any $k \in \cCRu$, any object $N$ of $\Mod_k$ 
and any $n \in N$, we consider a copy  $ke_n$  
of the  $k$-module $k$, where $\lambda e_n$ is identified  with $\lambda \in k$. 
So  $ke_n = k \in \cCLMu_k$ and we define the object of $\cCLMu_k$
\beq \label{scan1}
S(N) := k^{(N,\un)} = \nwhat{k^{(N)}} = \bigoplus_{n \in N}\nolimits^\un\; ke_n  \;.
\eeq
If $M$ is an object of $\cCLMu_k$, we set $S(M) := S(M^\for)$. There exists a canonical surjective morphism in $\cCLMu_k$
\beq \label{scan2}
\pi_M:S(M) \longrightarrow M \quad , \quad \sum_m a_m e_m \longmapsto \sum_m a_m m \;.
\eeq
If $k \in \cCRou$ then $S(M) \in \cCLMou_k$.
 \end{notation}
\begin{prop} \label{quotient-sum}  
Let $k \in \cCRou$, $M \in \cCLMu_k$ and $S(M)$
be as in Notation~\ref{S(M)}. 
Then  $M \in \cCLMcan_k$ iff $\pi_M$  is open   (\ie coincides with its coimage). 
\end{prop} 
\begin{proof} 
The sufficiency of the condition follows from the definition. 
Conversely, assume $M \in \cCLMcan_k$ and 
let $F:k^{(A,\un)} \to M$ be a strict epimorphism  in $\cCLMou_k$, as in the definition of a canonical module. 
For any $\alpha \in A$ we denote by 
$j_\alpha : k \to k^{(A,\un)}$ the canonical $\alpha$-th injection, and by 
$\delta_\alpha$ the image $j_\alpha(1) \in k^{(A,\un)}$. 
Let $m_\alpha = F(\delta_\alpha)$, for any $\alpha \in A$. We have a natural morphism
$$
S_F: k^{(A,\un)} \longrightarrow S(M) \quad , \quad \delta_\alpha \longmapsto e_{m_\alpha}
$$
such that $F = \pi_M \circ S_F$. 
A basis of open submodules of $M$ consists of (\cf \eqref{opensurj})
$$\{\ol{IM} =  F(\cl_{k^{(A,\un)}}(I\,k^{(A,\un)})) \}_{I \in \cP(k)} \;.
$$
To prove our statement it will suffice to check that $\pi_M(\ol{IS(M)}) = \ol{IM}$. 
The inclusion $\pi_M(\ol{IS(M)}) \subset \ol{IM}$ is automatic.  
On the other hand, $S_F(\cl_{k^{(A,\un)}}(I\,k^{(A,\un)})) \subset \ol{IS(M)}$ so that, by \eqref{opensurj},
$$
\ol{IM} =   F(\cl_{k^{(A,\un)}}(I\,k^{(A,\un)})) =   \pi_M (S_F(\cl_{k^{(A,\un)}}(I\,k^{(A,\un)}))) \subset \pi_M(\ol{IS(M)}) \;.
$$
\end{proof} 

\begin{rmk} \hfill \ben 
\item  For $k \in \cCRou$, any canonical $k$-module is the cokernel in $\cLMu_k$ (and in $\cCLMou_k$) of a 
morphism $k^{(B,\un)} \to k^{(A,\un)}$, for suitable index sets $A$ and $B$.  
In fact, let $F: k^{(A,\un)} \map{} M$ be as in Definition~\ref{strong-can-def}. Then, by 5 of Remark~\ref{bim=iso}, $F$ is the cokernel of its kernel $K := \Ker (F) \map{j} k^{(A,\un)}$. Consider the  morphism $\pi_K : S(K) \map{} K$ of \eqref{scan2}. Let us show that $F$ is the cokernel of 
$$ j \circ \pi_K : S(K)=k^{(K,\un)} \map{} k^{(A,\un)}\;.$$
To prove this, let $g:   k^{(A,\un)} \map{}  C$ be a morphism in $\cLMu_k$ such that $g \circ j \circ \pi_K : S(K) \map{}  C$ vanishes. Then  $g \circ j : K \map{}  C$
 also vanishes, and therefore there is a morphism $h: 
 M \map{}  C$ such that $g = h \circ F$.
\item  
 For any $k \in \cCRu$ and any $M$ in $\Mod_k$ the natural $k$-module surjection 
$$
 k^{(M)} \longrightarrow M \quad , \quad \sum_{m\in M} a_m e_m \longmapsto \sum_{m \in M} a_m m \;.
$$
extends by continuity to a  surjective morphism 
 \beq \label{scan3}
 \phi_M: S(M) \longrightarrow \nwhat{M} \;,
\eeq
 in general   not open.
However, if $M \in \cCLMu_k$, $\pi_M$ factors as 
\beq \label{factor}
S(M) \map{\phi_M} \nwhat{M} \map{\sigma_M} M
\eeq 
which already indicates that the identity of $M^\for$ induces a   morphism 
\beq \label{factor1}
\Coim(\pi_M) \map{} \Coim(\sigma_M) \;,
\eeq
not necessarily surjective nor open. 
\item Assume $k \in \cCRou$ and $M \in \cCLMu_k$. Still 
 $\phi_M$  is not in general  open so that 
$\nwhat{M}$ is not necessarily $k$-canonical.  
However, in this case the coimages  $\Coim(\pi_M)$ of $\pi_M$  and $\Coim(\sigma_M)$ of $\sigma_M$ in $\cCLMu_k$ may be calculated   in $\cLMou_k$, since the latter coimages  are already complete.  
Therefore, 
the morphism \eqref{factor1} 
is a bijection
\beq \label{factor2}
M^\can = \Coim(\pi_M) \map{(1:1)} \Coim(\sigma_M) = M^{\max}\;.
\eeq 
\een
\end{rmk}   
 \begin{prop} \label{scan-adjoint}  We assume here that $k \in \cCRou$. \hfill \ben
  \item 
 The 
inclusion functor  $\iota: \cCLMcan_k \hookrightarrow \cCLMu_k$ admits a right adjoint functor 
\beq \label{scan-adjoint1}  (-)^\can: \cCLMu_k \longrightarrow  \cCLMcan_k \;\;,\;\; M \longmapsto M^\can  
\eeq 
where, for $\pi_M : S(M) \to M$ as in \eqref{scan2},  
$M^\can = \Coim(\pi_M)$, taken in $\cLMou_k$. Equivalently, 
$M^\can$ has underlying module $M^\for$ and is endowed with the quotient topology of the canonical morphism 
$\pi_M : S(M) \to M$. 
 \item The canonical morphism $M^{\max} \longrightarrow M^\can$ (see Definition~\ref{defmax}) is an isomorphism with inverse \eqref{factor2}. 
Therefore, 
$M^\can = M^{\max}$ 
   is the unique maximal object of $\cCLMu_k$ 
   above $M$ with the same underlying $k$-module.
\een
\end{prop} 
\begin{proof}   \hfill
\par $\mathit 1.$
The canonical map $\veps(M) = \wtilde{\pi_M} : \Coim(\pi_M) \to \Im(\pi_M) = M$  for $\pi_M$   in the category $\cCLMu_k$, coincides with the 
canonical map  for $\pi_M$   in the category $\cLMou_k$.
The canonical map $\veps(M)$ is therefore bijective and, by Proposition~\ref{quotient-sum} it is an isomorphism in $\cCLMu_k$ if and only if $M$ is canonical. 
We need to prove that for any $N \in \cCLMcan_k$ and $M \in \cCLMu_k$
there is a canonical bijection 
$$ 
\Hom_{\cCLMu_k}(\iota(N),M) 
\longrightarrow \Hom_{\cCLMcan_k}(N,\Coim(\pi_M))
\;.
$$ 
Let  $f \in \Hom_{\cCLMu_k}(\iota(N),M)$ and $\pi_N: S(N) \longrightarrow N$. Then $f \circ \pi_N: S(N) \longrightarrow M$
lifts to a morphism $F: S(N) \longrightarrow S(M)$ and we get a commutative diagram
\beq \label{square1}
\begin{tikzcd}[column sep=2.5em, row sep=2.5em] 
S(N) \arrow{r}{F} \arrow{d}{\pi_N}
&S(M)  \arrow{d}{\pi_M}   
\\ 
N \arrow{r}{f}  
& M 
\end{tikzcd}
\eeq
and eventually a morphism $N = \Coim(\pi_N) \map{g} \Coim(\pi_M)$. In the other direction, from 
$g \in \Hom_{\cCLMcan_k}(N,\Coim(\pi_M))$ we get $f = \veps(M) \circ g$. 
\par $\mathit 2.$  This follows from the properties of $M^{\max} = \Coim (\sigma_M)$ (Definition~\ref{defmax} and $\mathit 6$ of Proposition~\ref{can-forget-adj}). Namely, the identity map of $M^\for$ induces an inverse of \eqref{factor2} which is therefore an isomorphism. 
\end{proof}   
\begin{rmk} \label{limcan} \hfill \ben
\item 
Since $(-)^\can$ is a right adjoint, $\cCLMcan_k$ admits all projective limits, denoted by $\limit^\can$,  calculated by applying $(-)^\can$ to $\limit$ in $\cCLMu_k$. In particular, the kernel of a morphism $M \map{f} N$ of $\cCLMcan_k$ is $\Ker (f)^\can$, where $\Ker(f)$ is calculated in  $\cCLMu_k$ (that is in $\cLMu_k$). 
A strict monomorphism $M \map{g} N$ in $\cCLMcan_k$ is not necessarily a strict monomorphism in $\cCLMu_k$, that is a closed embedding. Rather, there exists a morphism $N \map{\varphi} P$ of $\cCLMcan_k$ such that $M = \Ker (\varphi)^\can$, where $\Ker(\varphi)$ is calculated in  $\cCLMu_k$ and  is a closed subspace of $N$. 
\item
Let  $\{M_\alpha\}_{\alpha \in A}$ be a family in $\cCLMcan_k$. Then the product of $\{M_\alpha\}_{\alpha \in A}$ in $\cCLMcan_k$ coincides with $(\PROD_{\alpha \in A} M_\alpha)^\can$ where  $\PROD_{\alpha \in A} M_\alpha$ is taken in $\cLMu_k$, and 
is denoted $\PRODcan_{\alpha \in A} M_\alpha$.
If $k \in \cCRoufop$ then  $\PRODsqu_{\alpha \in A} M_\alpha$, described in \eqref{unifsqprod}, is simply 
 $\PROD_{\alpha \in A}M_\alpha^\for$ equipped with the naive canonical topology. The latter is complete since 
 $\PROD_{\alpha \in A}M_\alpha$ is. 
Therefore it is an object of 
 $\cCLMcan_k$ and  coincides with $\PRODcan_{\alpha \in A} M_\alpha$.
 \een
\end{rmk} 
 \begin{cor}\label{canmax2}  Let  $k \in \cCRou$ and $N \in \cCLMcan_k$.  Any  bijective morphism $f:M \to N$ of $\cCLMu_k$  is an
 isomorphism. 
 \end{cor}
  \begin{proof}
Follows from the identity $M^\can = M^{\max}$ in part $\mathit 2$ of Proposition~\ref{scan-adjoint}.
 \end{proof}
  The following corollary is a version of the classical  Open Mapping Theorem. 
 \begin{cor}\label{open-map} Let  $k \in \cCRou$.
 A surjective morphism in $\cCLMcan_k$ is open. 
 \end{cor}
 \begin{proof} 
 Let $f: M_1 \to M_2$ be a surjective morphism in $\cCLMcan_k$. 
 The canonical morphism $M_1 \to \Coim (f)$ in $\cCLMcan_k$, 
  is a cokernel in $\cCLMou_k$, \ie in $\cLMou_k$, hence it is open. 
  On the other hand, the canonical morphism in $\cCLMcan_k$: 
 $$
 \Coim (f) \to \Im (f) = M_2 
 $$ 
is bijective, hence  an isomorphism. 
 \end{proof}  
 \begin{cor}\label{canmax}  Let  $k \in \cCRouclop$.  
 For any $M \in \cCLMu_k$, \eqref{barrpscan311}  becomes
\beq \label{canmax1} \begin{split} \nwhat{M} &\map{\Coim(\sigma_{M})} M^{\max} = M^\can \map{(1:1)} \\ &M^\barrell  = M^\clop  =  M^\pscan  \map{(1:1)} M
\;.
\end{split}
\eeq  
In particular, for any object $M$ of $\cCLMu_k$, 
 $M^\pscan$ (resp.  $M^\barrell$, resp. $M^\clop$)  is the unique minimal object of $\cCLMu_k$ 
   above $M$ with the same underlying $k$-module which is pseudocanonical (resp. barrelled, resp. clop). In fact 
    $$M^\pscan = M^\barrell = M^\clop =: M^{\min}\;,$$
    \ie the three objects coincide. We will say that $M^{\min}$ is a \emph{minimal clop}  or a \emph{minimal pseudocanonical} or a \emph{minimal barrelled} module
 above $M$. 
 \end{cor}
 \begin{proof}
The statement simply summarizes what has been proven before.
 \end{proof}
  \begin{rmk} \label{canpscan}  We do not know of conditions on $k$ and $M$  under which  
 the bijective morphism  
 $$M^\can  \map{(1:1)} M^\pscan$$ 
 in \eqref{canmax1} would be  an isomorphism. 
  \end{rmk}
  \begin{cor}\label{canrex} Let  $k \in \cCRou$.
 A surjective morphism $M \map{f} N$ in $\cCLMcan_k$ is a cokernel.  More precisely, 
$f$ coincides with its coimage.
 \end{cor}
 \begin{proof} 
 Let $M \map{f} N$ be a surjective morphism of  $\cCLMcan_k$. Then, by Corollary~\ref{open-map}, $f$ is an open surjective map, hence by Corollary~\ref{strong-can-canCOR}  it is a strict epimorphism of 
 $\cCLMou_k$ and therefore a cokernel.   By $5$ of Remark~\ref{bim=iso},  $f$ is the cokernel of its kernel 
 $K \map{\iota} M$  in $\cCLMou_k$. 
 We claim that $f$ is also the cokernel of $K^\can \map{\iota^\can} M$ in $\cCLMcan_k$. 
 So, let 
  $M \map{h} Q$ be a morphism in $\cCLMcan_k$ such that $h \circ \iota^\can= 0$. Then $h \circ \iota = 0$ 
and there exists a morphism $N \map{j} Q$ in $\cCLMou_k$ such that $h = j \circ f$. Since $N \map{j} Q$ is also a morphism of $\cCLMcan_k$, this proves the statement. 
 \end{proof} 
 The following is our first main theorem.
 \begin{thm} \label{quasiabelian-scan}  Let $k$ be an object of  
 $\cCRou$.  
The  category  $\cCLMcan_k$ is a   complete quasi-abelian category. 
\end{thm}
\begin{proof} We saw already in  Remark~\ref{limcan} that $\cCLMcan_k$ is complete. 
 We build upon  Theorem~\ref{quasiabelian-compl}, hence eventually  on  the proof of 
Theorem~\ref{quasiabelian}.
Again, we have to prove  that strict monomorphisms  are stable under push-out, 
and strict epimorphisms are stable under pull-back. 
 \endgraf 
Let $i: N\hookrightarrow  M$ be a strict monomorphism   
  in $\cCLMcan_k$; so, $i$ is the kernel of a morphism $M \map{p} Q$ in $\cCLMcan_k$ which we may assume, by 6 of Remark~\ref{bim=iso},  to be the cokernel of $i$, hence an open surjection. 
Let $f: N\to N'$ be any morphism of $\cCLMcan_k$. 
We consider the push-out square  
\beq \label{pushoutmonocan}
\begin{tikzcd}[column sep=2.5em, row sep=2.5em] 
N \arrow[hook]{r}{i} \arrow{d}{f}
& M  \arrow{d}{f'} 
\\ 
N' \arrow{r}{i'}  
& N' \oplus_N M =: M'
\end{tikzcd}
\eeq
in $\cCLMcan_k$.
We need to show that $i': N' \hookrightarrow  M'$ is a strict monomorphism, as well. We complete 
\eqref{pushoutmonocan}
into the push-out diagram 
\beq \label{pushoutmonotot}
\D := \;\;\;  \begin{tikzcd}[column sep=2.5em, row sep=2.5em] 
N \arrow[hook]{r}{i} \arrow{d}{f}
& M  \arrow{d}{f'} \arrow{r}{p} & Q \arrow{d}{f''}
\\ 
N' \arrow{r}{i'}  
& M'  \arrow{r}{p'}&Q'
\end{tikzcd}
\eeq
in $\cCLMcan_k$, where $p'$ is the cokernel of $i'$, hence an open surjection. The morphism $i$ is not necessarily the kernel of $p$ in $\cCLMou_k$; we set 
$j: P := \Ker^{\cCLMou_k}(p) \hookrightarrow M$  so that $N = P^\can$, and $i$ factors as $j \circ \iota_P$, where $\iota_P:P^\can \to P$ is the canonical morphism. We construct the push-out diagram in $\cCLMou_k$

\beq \label{pushoutmonotot0}
\D_0 := \;\;\;  \begin{tikzcd}[column sep=2.5em, row sep=2.5em] 
N=P^\can \arrow{r}{\iota_P}  \arrow{d}{f}&P \arrow[hook]{r}{j} \arrow{d}{g}
& M  \arrow{d}{f'} \arrow{r}{p} & Q \arrow{d}{f''}
\\ 
N'   \arrow{r}{\iota_P'} & P'\arrow{r}{j'}  
& M'  \arrow{r}{p'}&Q'
\end{tikzcd}  
\eeq 
In particular, the r.h. part of  diagram $\D_0$ is a pushout in $\cCLMou_k$ 
\beq \label{pushoutmonotot01}
\D_{0,1} := \;\;\;  \begin{tikzcd}[column sep=2.5em, row sep=2.5em] 
P \arrow[hook]{r}{j} \arrow{d}{g}
& M  \arrow{d}{f'} \arrow{r}{p} & Q \arrow{d}{f''}
\\ 
P'\arrow[hook]{r}{j'}  
& M'  \arrow{r}{p'}&Q'
\end{tikzcd} 
\eeq 
Since $\cCLMou_k$ is quasi-abelian, we conclude that   $P' := \Ker^{\cCLMou_k}(p')$ and $j' :P' \map{} M'$ is the kernel of $p'$. 
\par The l.h. part of the diagram $\D_0$ 
is   a  push-out square in $\cCLMou_k$
\beq \label{pushout}
\begin{tikzcd}[column sep=2.5em, row sep=2.5em] 
N=P^\can \arrow{r}{\iota_P} \arrow{d}{f}
& P  \arrow{d}{g}   
\\ 
N' \arrow{r}{\iota_P'}  
& P'   
\end{tikzcd}
\eeq
Then \eqref{pushout} is also a push-out in $\cLMu_k$, because the calculation of a push-out here only involves a direct sum and a cokernel, and finally  \eqref{pushout} is   a push-out in  $\Mod_k$. Therefore, since the morphism $\iota_P$ is bijective so is  $\iota_P'$. 
  Now $N'$ is canonical and 
 $\iota_P'$ is bijective, so that we necessarily have $N' = (P')^\can$, and $\iota_P'$ is the canonical morphism $\iota_{P'}: (P')^\can \to P'$. In particular, $f = g^\can$, since they coincide set-theoretically. 
\par \smallskip The conclusion of the previous argument is that we may complete diagram $\D$ into 
the push-out  diagram in  $\cCLMou_k$
\beq \label{pushoutmonotot1}
\D_1 := \;\;\;  \begin{tikzcd}[column sep=2.5em, row sep=2.5em] 
N=P^\can \arrow{r}{\iota_P}  \arrow{d}{f=g^\can}&P \arrow[hook]{r}{j} \arrow{d}{g}
& M  \arrow{d}{f'} \arrow{r}{p} & Q \arrow{d}{f''}
\\ 
N' = (P')^\can \arrow{r}{\iota_{P'}} & P'\arrow[hook]{r}{j'}  
& M'  \arrow{r}{p'}&Q'
\end{tikzcd}
\eeq 
which shows that $i' = j' \circ \iota_{P'}$ is the kernel of $p'$ in $\cCLMcan_k$.
\par \smallskip
 Let  now $p: M\rightarrow N$ be a strict epimorphism in $\cCLMcan_k$
 and let $\ell: N'\to N$ be any morphism in that category. 
 We let $i: K = \Ker^{\cCLMcan_k}(p) \to M$  be the kernel of $p$  in $\cCLMcan_k$, 
 so that $p$ identifies with $M \to N=\Coker_{\cCLMcan_k}(i) = \Coker_{\cLMu_k}(i) =  \Coker_{\cCLMou_k}(i)$. 
Let 
$$ 
\D := \;\;\; \begin{tikzcd}[column sep=2.5em, row sep=2.5em] 
K' = \Ker^{\cCLMcan_k}(p') \arrow{r}{i'} \arrow{d}{\ell''} 
&M' \arrow{r}{p'} \arrow{d}[']{\ell'}
& N'  \arrow{d}{\ell} 
\\ 
K \arrow{r}{i}  
& M \arrow{r}{p}  
& N 
\end{tikzcd}
$$ 
be the diagram obtained by pull-back by $\ell$  in  $\cCLMcan_k$. 
We will show that $M' \map{p'} N'$ is the cokernel of $i'$ in  $\cCLMcan_k$. 
The diagram $\D$ admits an adjunction morphism  to the analog diagram obtained by pull-back by $\ell$  in  $\cCLMou_k$  
$$ \D_u := \;\;\;
\begin{tikzcd}[column sep=2.5em, row sep=2.5em] 
 K'_u = \Ker^{\cCLMou_k}(p'_u) \arrow{r}{i'_u} \arrow{d}{\ell''_u} 
&M'_u \arrow{r}{p'_u} \arrow{d}[']{\ell'_u}
& N'  \arrow{d}{\ell} 
\\ 
K \arrow{r}{i}  
& M \arrow{r}{p}  
& N 
\end{tikzcd}  
$$ 
Since limits in $\cCLMcan_k$ are calculated by application of the functor $(-)^\can$ to the same limits in $\cCLMou_k$
we have  
$\D = (\D_u)^\can$. 
Since the category $\cCLMou_k$ is quasi-abelian, $p'_u$ is the cokernel of $i'_u$ in $\cCLMou_k$, hence it is surjective. By Corollary~\ref{canrex}  
$p' = (p'_u)^\can$ is a cokernel. Since  $i' = (i'_u)^\can$ is the kernel of $p'=(p'_u)^\can$ in $\cCLMcan_k$, $p'$ is the cokernel of $i'$. 
 \end{proof} 
  In  $4$ of Remark~\ref{strong-can-can} we observed that
 when $k$ is discrete, for any $M \in \cCLMu_k$, $M^\can$ coincides with $M^\naive$ \ie with $M^\for$ equipped with the discrete topology. 
So,  for $k$ discrete and $M \in \cCLMu_k$,
 \beq \label{candiscr1}
 \nwhat{M} =  M^{\max} = M^\can = \what{M^\naive} = M^\naive\;. 
 \eeq
 The next result, Corollary~\ref{naivecancor},   generalizes \eqref{candiscr1} to $k \in \cCRoufop$ and any $M \in \cCLMu_k$. 
\begin{cor} \label{naivecancor} Assume $k$ is in  $\cCRoufop$. Then 
\hfill \ben
\item
For any $M \in \cCLMu_k$ the morphism 
$\nwhat{M} \map{\Coim(\sigma_{M})} M^{\max} = M^\can$ in \eqref{canmax1} is an isomorphism. 
Therefore $\cCLMcan_k$ identifies with $\cCLMuop_k = \cCLMnaive_k$, 
\ie the full subcategory of $\cCLMu_k$ of   those 
 objects whose topology is the naive $k$-canonical one, or, equivalently, which are of the form $\nwhat{N}$, for some $N$ in $\Mod_k$.   
\item  Let $M$ be a canonical $k$-module. 
Then the  submodules of $M$ of the form $IM$, 
for $I$ an open ideal of $k$,  are open and canonical in the relative topology of $IM \subset M$.  In particular, $I$ itself, equipped with the relative topology of $I \subset k$,   
is a canonical $k$-module. 
 \item  
Let $M \in \cCLMcan_k$, 
that is an object of $\cCLMu_k$ whose topology is the naive $k$-canonical one. 
Let $N \subset M$ be a closed submodule. 
Then $N$ is separated and complete in its naive $k$-canonical topology (which is not, however, necessarily the subspace topology of $N \subset M$).
\een \end{cor} 
\begin{proof} 
$\mathit 1.$ The fact that, if $k \in \cCRoufop$, then  $\Coim(\sigma_{M})$ is an isomorphism and that 
$\nwhat{M} \iso M^\naive \iso  M^{\max}$, 
was already seen in part $\mathit 2$ of Theorem~\ref{barrpscan3}. 
\par $\mathit 2.$ The relative topology of $IM \subset M$ is its naive $k$-canonical topology and $IM$ is closed in $M$, hence complete. The conclusion follows from the previous point $\mathit 1$.
\par $\mathit 3.$ Let us equip $N$ with the relative topology of $N \subset M$. Then we have a bijective morphism 
$N^\can \to N$.    But $N^\can = \nwhat{N}$, and the latter is the $k$-module $N^\for$ equipped with its naive canonical topology. 
\end{proof}
 \begin{cor} \label{colimcan1} If   $k \in \cCRoufop$  the category $\cCLMcan_k$ admits all colimits, so that it is bicomplete.  \end{cor}  
  \begin{proof}
 For any inductive system $\{M_\alpha\}_\alpha$ in $\cCLMcan_k$, let $C$ be the colimit  of $\{M_\alpha\}_\alpha$ in the category $\cLMu_k$. Then $C$ is simply $\colimit_\alpha M_\alpha^\for$ equipped with the naive $k$-canonical topology. Its completion is  the colimit  of $\{M_\alpha\}_\alpha$ in the category $\cCLMu_k$ and is still equipped with the naive $k$-canonical topology, so that it is an object of $\cCLMcan_k$. 
\end{proof}
\begin{rmk} \label{colimcan2} Let $k \in \cCRoufop$. We may reach the   conclusions of  Corollary~\ref{colimcan1} using formula  \eqref{indu} instead. In fact, let  $\{M_\alpha\}_{\alpha \in A}$ be as in the proof of that corollary  
and let  $I \in \cP(k)$ be finitely generated. Then, by  \cite[Rmk. 8.3.3 (iv)]{GR}, $\ol{I M_\alpha} =  I M_\alpha$, $\forall \alpha \in A$.  Let $M :=  \colimit_{\alpha \in A} M_\alpha$ in $\Mod_k$ so that
$$
 M/IM = \colimit_{\alpha \in A} (M_\alpha/IM_\alpha) \;,
 $$
 in $\Mod_k$, and
 $$
 (M/IM)^\dis = \colimit^\un_{\alpha \in A} (M_\alpha/IM_\alpha)^\dis \;.
 $$
Therefore
\eqref{indu} becomes
\beq \label{indu11}
\colimit^\un_{\alpha \in A} M_\alpha = \limit_{I \in \cP(k)} (M/IM)^\dis = \nwhat{M}  \;.\eeq
Since, by $\mathit 1$ of Corollary~\ref{naivecancor},  $\nwhat{M}$ is canonical, we conclude  that 
$\cCLMcan_k$ is closed under $\colimit^\un$ (\ie under colimits   in $\cCLMu_k$). 
\end{rmk}
We have established  our second main theorem
\begin{thm}\label{2ndmainthm} Let $k \in \cCRoufop$. Then the category $\cCLMcan_k$ is the full subcategory of $\cCLMou_k$ whose objects are complete in their naive $k$-canonical topology. It is a bicomplete quasi-abelian category.
\end{thm}
  \begin{rmk}\label{cannotab} The category $\cCLMcan_k$ fails to be abelian because for a morphism $f:M \to N$ in 
 $\cCLMcan_k$ the canonical bimorphism $\wtilde{f} : \Coim (f) \longrightarrow \Im (f)$  is injective, but, even in case $k \in \cCRoufop$, not always surjective. As an example, let us consider the morphism described in  \cite[Tag 07JQ]{stacks}, on which 
 Lemma~110.10.1 of \lc is based. It is the injective morphism in $\cCLMcan_{\Z_p}$
 \beq \label{diag110}
 \begin{split}
 \phi := {\rm diag}(1,p,p^2,\dots) : \Z_p^{(\N,\un)} &\map{} \Z_p^{\N,\square, \un} \\
 (x_1,x_2,x_3,\dots) &\longmapsto (x_1,px_2,p^2 x_3,\dots) \;.
 \end{split}
 \eeq
Contrary to the natural morphism $\Z_p^{(\N,\un)} \map{} \Z_p^{\N,\square, \un}$ of Proposition~\ref{pscansumbox}, which is a closed embedding, \eqref{diag110} is not closed.  Here  $\Z_p^{(\N,\un)} \map{} \Coim (\phi)$ is the identity isomorphism, while $\Im(\phi)$ is the closure of the set-theoretic image of $\phi$, equipped with its $p$-adic topology. So, $(1,p,p^2,\dots) \in \Im(\phi)$, and the canonical morphism $\wtilde{\phi} : \Coim (\phi) \map{} \Im (\phi)$ of $\cCLMcan_{\Z_p}$ is not injective, hence is not an isomorphism. 
  \end{rmk}  
  \end{subsection}
\begin{subsection}{Projective canonical modules}
\label{projcanss}
We assume in this section that $k \in \cCRou$. 
\begin{lemma} \label{kproj} 
For any $M$ in $\cCLMcan_k$ the functor 
$$
\cCLMcan_k \longrightarrow \Mod_k \quad,\quad M \longmapsto \Hom_{\cCLMcan_k}(k, M)
$$
coincides with $M \mapsto M^\for$. For any small set $A$, 
$$M \longmapsto \Hom_{\cCLMcan_k}(k^{(A,\un)}, M) = (M^\for)^A \;.
$$ 
 In particular, any direct summand of $k^{(A,\un)}$ is projective (\cf Definition~\ref{projdef}).  
 \end{lemma}
\begin{proof} 
For any $M$ in $\cCLMcan_k$, the map
$$
M^\for \longrightarrow \Hom_{\cCLMcan_k}(k, M)\;,\;\; m \longmapsto (\lambda \mapsto \lambda m) 
$$
is an isomorphism in $\Mod_k$.  Strict epimorphisms in $\cCLMcan_k$ are surjective, 
so the functor $\Hom_{\cCLMcan_k}(k , -)$ transforms a  strict epimorphism $M \longrightarrow N$  into the surjection 
$M^\for  \longrightarrow  N^\for$. Since the abelian category  $\Mod_k$ satisfies $AB4^\ast$ (exactness of products), the $k$-linear map $(M^\for)^A \longrightarrow (N^\for)^A$ is also a surjection. 
The functor $\Hom_{\cCLMcan_k}(k^{(A,\un)}, -)$ transforms the  strict epimorphism $M \longrightarrow N$  into the $k$-linear map $(M^\for)^A \longrightarrow (N^\for)^A$, so into a surjection. 
\par Finally, let $R \oplus S = k^{(A,\un)}$ be a direct sum decomposition of $k^{(A,\un)}$ in $\cM_k$ (hence in $\cCLMcan_k$). Then 
$$\Hom_{\cCLMcan_k}(k^{(A,\un)}, M) =  \Hom_{\cCLMcan_k}(R, M) \bigoplus  \Hom_{\cCLMcan_k}(S, M)
$$
so 
$$\Hom_{\cCLMcan_k}(R, M) \longrightarrow \Hom_{\cCLMcan_k}(R, N)
$$
is surjective. Therefore $R$ is a projective object of $\cCLMcan_k$. 
\end{proof}
\begin{rmk} \label{existproj}
By definition, the category $\cCLMcan_k$ has enough projectives. It follows from  \cite[Prop. 1.4.5]{schneiders} that in 
$\cCLMcan_k$ products are exact. Moreover, the object $k$ is a strict generator of $\cCLMcan_k$.
\end{rmk}
\begin{prop}\label{thickprop} The quasi-abelian category 
$\cCLMcan_k$ is a full subcategory  of the quasi-abelian category $\cCLMou_k$ closed by  quotients and extensions.
\end{prop}
\begin{proof} $\cCLMcan_k$ being ``closed by quotients'' in $\cCLMou_k$ means that if $f:M \to N$ is a strict epimorphism in  $\cCLMou_k$, and $M \in \cCLMcan_k$, then $N \in \cCLMcan_k$, as well. This is clear. 
\par
We now consider an exact sequence (\ie any kernel-cokernel pair) in $\cCLMou_k$
$$
  A  \map{i} B \map{p} C   \;.
$$ 
where $A, C$ are objects of $\cCLMcan_k$. Let $\pi_A$ and $\pi_C$ be the open morphisms of 
\eqref{scan2} and let us consider the commutative diagram 
where $\alpha = i \circ \pi_A$ and $\beta : k^{(C,\un)} \longrightarrow B$  is any morphism  such that $p \circ \beta = \pi_C$
(such a $\beta$ exists because $k^{(C,\un)}$ is projective)
$$ 
\begin{tikzcd}[column sep=2.5em, row sep=2.5em] 
& k^{(A,\un)}  \arrow{dr}{\alpha}  \arrow{d}[']{\pi_A}  & 
& k^{(C,\un)} 
 \arrow{dl}[']{\beta} 
 \arrow{d}{\pi_C}
&  
\\ 
 & A \arrow{r}[']{i}  
& B \arrow{r}[']{p}  & C  
&  \;\;.
\end{tikzcd}
$$
We need to show that the canonical surjective morphism 
$$(\alpha,\beta) : k^{(A,\un)} \oplus k^{(C,\un)} \longrightarrow B$$
is open.  
Let $J$ be an open ideal of $k$, so that there exists $U \in \cP_k(B)$ such that 
$$
U \cap A = \ol{JA} = \pi_A(\ol{Jk^{(A,\un)}})\; \;\mbox{while}\; \; \ol{JC}  = \pi_C(\ol{Jk^{(C,\un)}})  \in \cP_k(C) \;.
$$
Let $U':=  U \cap p^{-1}(\ol{JC}) \in \cP_k(B)$ so that $U' \cap A = \ol{JA}$ and $p(U') \subset \ol{JC}$
is an open $k$-submodule of $C$. Finally, we set  $U'':= \beta(\ol{Jk^{(C,\un)}})  + U' \in \cP_k(B)$ so that
$$
U'' \cap A = \pi_A(\ol{Jk^{(A,\un)}})\; \;\mbox{while}\; \;  p(U'') = \pi_C(\ol{Jk^{(C,\un)}})\;.
$$
We conclude that $(\alpha,\beta) \l(\,\ol{J(k^{(A,\un)} \oplus k^{(C,\un)})}\,\r) = U'' \in \cP_k(B)$,  so that $\alpha \oplus \beta$ is open.
\end{proof}
\begin{prop} \label{sumproj} An object of $\cCLMcan_k$ is projective if and only if it is a direct summand of an object of the form $k^{(A,\un)}$, for a small set $A$. Any projective object of $\cCLMcan_k$ is pro-flat.  
\end{prop}
\begin{proof}   We need to show that if $M \in \cCLMcan_k$ is projective, then $M$ is a direct summand of   $k^{(M,\un)}$. In fact,   $\pi_M : k^{(M,\un)} \longrightarrow M$ is a strict epimorphism,  and, $M$ being projective, $\pi_M$  admits a section $s: M \longrightarrow k^{(M,\un)}$.    
Let $\iota: N \longrightarrow k^{(M,\un)}$ be the kernel of $\pi_M$ in the category $\cCLMcan_k$.
By   Corollary~\ref{canmax2}, the bijective morphism $(s,\iota): M \oplus N \longrightarrow k^{(M,\un)}$  
is an isomorphism.   
  \par  
 Since, for any $I \in \cP(k)$, $k^{(A,\un)}/\ol{Ik^{(A,\un)}}= k^{(A)}/\ol{Ik^{(A)}} = (k/I)^{(A)}$ is free, 
 hence flat, over $k/I$, $k^{(A,\un)}$ is pro-flat. If $M \oplus N = k^{(A,\un)}$ is a direct sum decomposition in 
 $\cCLMcan_k$, both $M$ and $N$ are then pro-flat. 
 \end{proof}
\begin{prop} \label{projcan}  Let 
$$  
Z\map{i} Y \map{p} X   
$$ 
be an   exact sequence in $\cCLMcan_k$ with $X$  projective. Then 
$Y$ is projective iff $Z$ is projective.
\end{prop}
\begin{proof} Let $s: X \longrightarrow Y$ be a section of $p$. 
 The bijective morphism $(s,i) : X \oplus Z \longrightarrow Y$ is necessarily an isomorphism in $\cCLMcan_k$. So, if  $X$ and $Z$ are projective,    $Y$ is projective. 
 Conversely, if  $Y$ is projective it is a direct summand of an object of the form $k^{(A,\un)}$. Hence so is $Z$ which is then projective.
\end{proof}

\end{subsection} 
\end{section}
 
\begin{section}{Tensor products} \label{tensors}   
  In this section, $k$ is any object of $\cRu$. 
  Further requirements will be specified 
  when needed.
\begin{subsection}{Bilinear maps}
\begin{defn} \label{hypodefsp}  Let $M_1,M_2,N$ be objects of  $\cLMu_k$.
We denote by $\Bil^\se_k(M_1 \times M_2,N)$ (resp. $\Bil^\co_k(M_1 \times M_2,N)$, resp.   $\Bil^\un_k(M_1 \times M_2,N)$) the $k$-module of
  $k$-bilinear maps 
\beq
\label{hypodef1} 
\varphi: M_1 \times M_2 \to N
\eeq
which are separately continuous in the two variables (resp.   continuous, resp.   uniformly continuous).
\end{defn}  
\begin{notation} \label{linbilin}
To a $k$-bilinear map as in \eqref{hypodef1} we associate the two $k$-linear maps 
$$\Phi: M_1 \longrightarrow \Hom_k(M_2,N)\;\;,\;\; \Psi: M_2 \longrightarrow \Hom_k(M_1,N)$$
defined by
$$\Phi(m_1) = \varphi(m_1,-): M_2 \longrightarrow N \;\;,\;\; m_2 \longmapsto   \varphi(m_1,m_2) \;,
$$  
and 
$$\Psi(m_2) = \varphi(-,m_2): M_1 \longrightarrow N \;\;,\;\; m_1 \longmapsto   \varphi(m_1,m_2) \;.
$$ 
\end{notation}
 \begin{rmk}
 \label{sepcont} \hfill 
\ben
\item $\varphi$ as in \eqref{hypodef1} is continuous if and only if it is separately continuous and it is continuous at $(0,0)$. If  $M_1$ and $M_2$ are pseudocanonical, the latter condition is automatic.  
In fact, it suffices to show that for any $Q \in \cP_k(N)$ there exist $P_i \in \cP_k(M_i)$, for $i=1,2$, such that $\varphi(P_1 \times P_2) \subset Q$. This follows from the fact that there is $J \in \cP(k)$ such that $JN \subset Q$. Then pick $P_i = \ol{JM_i}$, for $i=1,2$, to get 
$$\varphi(P_1 \times P_2)\subset \varphi(P_1 \times M_2) + \varphi(M_1 \times P_2) \subset \ol{JN} \subset Q \;.$$
 
  \item  
   $\varphi$ as in \eqref{hypodef1} is uniformly continuous $\Leftrightarrow$ for any $Q \in \cP_k(N)$, there are $P_i \in \cP_k(M_i)$, for $i=1,2$, such that  
$$
\varphi(M_1 \times P_2) + \varphi(P_1 \times M_2) \subset Q \;.
$$   
   \item 
If $M_1$ and $M_2$ are clop, 
  a $k$-bilinear map as in \eqref{hypodef1} is continuous if and only if it is 
uniformly  continuous. In fact, if $\varphi$ is continuous, for any $Q \in \cP_k(N)$ there are 
$P_i \in \cP_k(M_i)$, for $i=1,2$, such that $\varphi(P_1 \times P_2) \subset Q$. 
Now fix $Q \in \cP_k(N)$. 
There exist $J_i \in \cP(k)$, for $i=1,2$, such that
$J_i M_i \subset P_i$ for $i=1,2$. By the clop property, $\ol{J_i P_j} \in \cP_k(M_j)$  
for $\{i,j\} = \{1,2\}$. Since $\ol{J_i P_j}  \subset P_j$
$$\varphi(M_1 \times \ol{J_1P_2}) \subset \ol{\varphi(M_1 \times  J_1P_2)}  = \ol{\varphi(J_1 M_1 \times  P_2)} \subset \ol{\varphi(P_1 \times P_2)} \subset Q \;.
$$
Similarly, $\varphi(\ol{J_2P_1} \times M_2) \subset Q$. We conclude that for any $Q \in \cP_k(N)$ there exist
$Q_i   \in \cP_k(M_i)$, for $i=1,2$, such that, for any $(x_1,x_2) \in M_1 \times M_2$, 
$$
\varphi (x_1 +Q_1,x_2 +Q_2) \subset \varphi (x_1,x_2) + Q 
$$
(take for example $Q_1 = \ol{J_2P_1}$ and $Q_2 = \ol{J_1P_2}$).
Therefore
  $\varphi$ is uniformly continuous.  
This is true in particular if  $k \in \cCRouclop$ and $M_1, M_2$ are canonical. 
\item If  $M_1, M_2$ are clop (hence pseudocanonical), then 
$$\Bil^\se_k(M_1 \times M_2,N)=\Bil^\co_k(M_1 \times M_2,N)=\Bil^\un_k(M_1 \times M_2,N)\;.
$$ 
\een
\end{rmk}

\end{subsection}
  
\begin{subsection}{Internal tensor products}
 The category  $\cLMu_k$ 
as well as  its separated and complete   counterparts, admits  a 
natural structure  of symmetric monoidal  category.
We start by discussing  the corresponding  internal tensor products. 
 \begin{cor} \label{bildef10} \hfill
 \ben 
\item Let $M,N \in \LMu_k$. The functor  
$$
\cLMu_k \to \Mod_k  \;,\quad  X \mapsto \Bil^\un_k(M \times N, X) 
$$
is corepresented by an object $M \otimes^\un_k N$ of $\LMu_k$ with underlying $k$-module 
 $M^\for \otimes_k N^\for$
and an element
 $$\otimes^\un_k \in \Bil^\un_k(M \times N, M \otimes^\un_k N)$$
 with underlying $k$-linear map  $\otimes_k$.
A basis of open $k$-submodules of $M \otimes^\un_k N$ is
$$\{\Im(P \otimes N) + \Im(M \otimes Q)\,|\, P \in \cP_k(M)\,,\,Q \in \cP_k(N)\,\}\;.$$
 \een
\end{cor} 
\begin{proof}   Follows from 2 of Remark~\ref{sepcont}.
\end{proof} 
 \begin{prop}\label{cor-mon-cat-noncompl}  
The functor  $\otimes^\un_k$  gives to  $\cLMu_k$ (resp. to $\cLMpscan_k$) a structure 
of   symmetric monoidal category with unit $k$.  
  If $k \in \cRuclop$ (resp. $\in \cRuop$) the same holds true
for  $\cLMuclop_k$ (resp. for $\cLMuop_k = \cLM^\naive_k$). 
\end{prop}
\begin{proof} All we need to observe is the following 
  \begin{lemma} \label{tensprop} Let  $M, N \in \cLMpscan_k$,  then $M \otimes^\un_k N \in \cLMpscan_k$.  
  \end{lemma} 
\begin{proof}  For any $I \in \cP(k)$ we have 
$$
I (M \otimes_k N) = IM \otimes_k N  = M \otimes_k IN
$$
so that
\beq \label{naivetens}
I (M \otimes_k N) = IM \otimes_k N + M \otimes_k IN
\eeq
and, by separate continuity of $\otimes^\un_k$, 
$$
\ol{I (M \otimes_k N)} = \ol{IM} \otimes_k N + M \otimes_k \ol{IN} \;.
$$
 \end{proof} 
 The last part of the statement in the case $k \in \cRuclop$ follows from 3 of Remark~\ref{barrpscan}. 
 For $k \in \cRuop$, $\cLMuop_k = \cLM^\naive_k$ and if $M,N \in \cLM^\naive_k$, 
 $M \otimes_k N \in \cLM^\naive_k$ by \eqref{naivetens}.
 \end{proof} 
 \begin{cor} \label{cloptensor1}  Let   
$A,B,C \in \cRu$ and let 
$\chi : A \to B$ and $\psi : A \to C$ $\clop$-adic (resp. $\op$-adic) morphisms in $\cRu$. Then 
$\chi \otimes^\un_A \psi:
A \longrightarrow B \otimes^\un_A C
$
is $\clop$-adic (resp. $\op$-adic).  In particular, if $A \in \cRuclop$ (resp. $A \in \cRuop$), then $B \otimes^\un_A C \in \cRuclop$ (resp. $\in \cRuop$).
\end{cor}
\begin{proof} Let $I$ be an open ideal of $A$. We want to show that the closure of $I(B \otimes^\un_A C)$ is open. In fact it contains both $\ol{IB} \otimes^\un_A C$ and $B \otimes^\un_A \ol{IC}$. Similarly for the op case.
\end{proof}
 
 \end{subsection} 

\begin{subsection}{Complete tensor products} 
We assume here that $k$ is in $\cCRu$ and let $M,N \in \cCLMu_k$. 
The complete tensor product  $M \wt^\un_k N \in \cCLMu_k$ is the completion of  $M \otimes^\un_k N$. It is the representative of the functor  
$$\cCLMu_k \map{} \Mod_k \;\;,\;\; X \mapsto \Bil_k^\un(M,N;X) \;.
$$
Its existence and construction are recalled in the next
 \begin{prop} \label{compltenshu} Let $k \in \cCRu$ and $M,N \in \cCLMu_k$. Then
 the functor 
$$\cCLMu_k \longrightarrow \Mod_k$$
given by $N \mapsto \Bil^\un_k(M_1 \times M_2,N)$, is corepresented by the completion $M_1 \wt^\un_k M_2$ of $M_1  \otimes^\un_k M_2$ and by the natural map
$\wt^\un_k \in \Bil^\un_k(M_1 \times M_2,M_1 \wt^\un_k M_2)$ obtained from $\otimes^\un_k$.  Explicitly, for $M,N$ in $\cCLMu_k$, we have 
\beq \label{tensudef} \begin{split}
M \wt^\un_k N  = \widehat{M \otimes^\un_k N} &=   \limit_{P,Q} \, \; (M  \otimes_k N) /  (P  \otimes_k N + M  \otimes_k Q) 
\\ &= \limit_{P,Q} \, \; M/P \otimes_k N/Q  \;,
\end{split}
\eeq
for $P$ (resp. $Q$) varying in the set of open submodules of $M$ (resp. $N$), 
where all the $k$-modules appearing  in the projective systems carry the discrete topology.
A fundamental system  of open submodules of  $M \wt^\un_k N$  
consists of the closures in $M \wt^\un_k N$   of  the $k$-submodules 
$P  \otimes_k N + M  \otimes_k Q \subset M \otimes_kN$, for  $P, Q$ as before. 
\end{prop}
\begin{proof} See
\cite[{$\bf 0$}.7.7]{EGA}  or  \cite[Chap. III, \S 2, Exer. 28]{algebracomm}.\par
\end{proof}
\begin{rmk} \label{surjsystem} The calculation of $M \wt^\un_k N$  can be performed starting from any description of $M$ and $N$ as limits of cofiltered  projective systems of discrete uniform $k$-modules $\{M_\alpha\}_{\alpha \in A}$ and $\{N_\beta\}_{\beta \in B}$ such that the morphisms $M_{\alpha'} \map{} M_\alpha$ and $N_{\beta'} \map{} N_\beta$ are surjective, for any $\alpha \leq \alpha'$ in $A$ and $\beta \leq \beta'$ in $B$. Recall  that   a discrete $k$-module is uniform iff
it has  open annihilator in $k$ (see Remark~\ref{discrete-mod}). Then
$$
M \wt^\un_k N =  \limit_{\alpha,\beta} M_\alpha   \otimes _k N_\beta  \;.
$$
\end{rmk} 
\begin{lemma} \label{tensfunct} \hfill \ben
\item Let $\{M_\alpha\}_{\alpha \in A}$ (resp.  $\{N_\beta\}_{\beta \in B}$) be an inductive system in $\cCLMu_k$ and let 
 $$
 M = \colimit_{\alpha \in A}^\un M_\alpha \;\;,\;\;  N = \colimit_{\beta \in B}^\un N_\beta
 $$
 in $\cCLMu_k$. Then we have a natural morphism
 \beq  \label{tensfunct1}
 \colimit_{\alpha,\beta}^\un (M_\alpha \wt^\un_k N_\beta) \map{} M \wt^\un_k N \;.
 \eeq
 \item Let $\{M_\alpha\}_{\alpha \in A}$ (resp.  $\{N_\beta\}_{\beta \in B}$) be a  projective system in $\cCLMu_k$ and let 
 $$
 M = \limit_{\alpha \in A} M_\alpha \;\;,\;\;  N = \limit_{\beta \in B} N_\beta
 $$
 in $\cCLMu_k$. Then we have a natural morphism
 \beq  \label{tensfunct2}
 M \wt^\un_k N \map{}   \limit_{\alpha,\beta} (M_\alpha \wt^\un_k N_\beta)\;.
 \eeq
 \een
 \end{lemma}
 \begin{proof} Clear. 
 \end{proof} 
\begin{prop} \label{promodcompl}  
Let  $\{M_\alpha\}_{\alpha \in A}$ and $\{N_\beta\}_{\beta \in B}$ be 
cofiltered projective systems in $\cCLMu_k$, indexed by the filtered posets $A,B$, such that the morphisms $M_{\alpha'} \map{} M_\alpha$ and $N_{\beta'} \map{} N_\beta$ are cokernels in $\cCLMu_k$ 
for any $\alpha \leq \alpha'$ in $A$ and $\beta \leq \beta'$ in $B$. 
Then the canonical  morphism  
\beq  \label{promodcompl2} \limit_{\alpha} M_\alpha \, \wt^\un_k \,  \limit_{\beta} N_\beta \map{}
 \limit_{\alpha,\beta} \, M_\alpha \wt^\un_k N_\beta 
\eeq
is an isomorphism in $\cCLMu_k$.  
\end{prop}
\begin{proof}  For any $\alpha \in A$, $M_\alpha$ is the limit  of a cofiltered  projective system  of discrete uniform $k$-modules 
$\{M_\alpha/P_\alpha\}_{P_\alpha \in \cP(M_\alpha)}$.  We define a new filtered poset 
$\Gamma$ consisting of the pairs $(\alpha, P_\alpha)$ such that $\alpha \in A$ and $P_\alpha \in \cP(M_\alpha)$. Then 
$(\alpha, P_\alpha) < (\alpha', P_{\alpha'})$ iff the morphism $M_{\alpha'} \map{} M_\alpha$ sends $P_{\alpha'}$ into $P_\alpha$.
Similarly, we define a new filtered poset $\Delta$ consisting of the pairs $(\beta,Q_\beta)$ such that $\beta \in B$ and $Q_\beta \in \cP(N_\beta)$. 
Then, for any $\gamma = (\alpha, P_\alpha) \in \Gamma$ and  $\delta = (\beta, Q_\beta) \in \Delta$, we set 
$$
M_\gamma := M_\alpha/P_\alpha\;\;\mbox{and}\;\; N_\delta := N_\beta/Q_\beta\;.
$$
For $\gamma = (\alpha, P_\alpha) \leq \gamma' =(\alpha', P_{\alpha'})$ the morphism $M_{\gamma'} \map{} M_\gamma$ is a  surjection of discrete $k$-modules because  $M_{\alpha'} \map{} M_\alpha$ is a cokernel in $\cCLMu_k$, so that 
$M_{\alpha'}/P_{\alpha'} \map{} M_\alpha/P_\alpha$ is a cokernel in $\Mod_k$. 
 The projective system $\{M_\gamma\}_{\gamma \in \Gamma}$ then satisfies the assumptions in Remark~\ref{surjsystem}. Similarly for $\{N_\delta\}_{\delta \in \Delta}$. 
We have 
$$\limit_{\alpha} M_\alpha = \limit_{(\alpha, P_\alpha) \in \Gamma} M_\alpha/P_\alpha = \limit_{\gamma \in \Gamma} M_\gamma
$$ 
and, similarly, 
$$\limit_{\beta} N_\beta = \limit_{(\beta, Q_\beta) \in \Delta} N_\beta/Q_\beta = \limit_{\delta \in \Delta} N_\delta \;.
$$  
Then
$$
\limit_{\alpha} M_\alpha  \, \wt^\un_k \, \limit_\beta N_\beta = \limit_{\gamma} M_\gamma \, \wt^\un_k \, \limit_\delta N_\delta = 
\limit_{\gamma,\delta} \, M_\gamma   \otimes_k   N_\delta = 
$$
$$
\limit_{(\alpha, P_\alpha),(\beta, Q_\beta)} M_\alpha/P_\alpha  \otimes_k N_\beta/Q_\beta = \limit_{\alpha,\beta} M_\alpha \, \wt^\un_k \,  N_\beta \;, 
$$
where the second equality follows from Remark~\ref{surjsystem}. 
\end{proof}  
\begin{rmk} \label{tensinvlim} It should be noticed that, for any  $R \in \Rings$,  in the algebraic category $\Mod_R$  the tensor product $-\otimes_R-$ does not  in general commute with limits taken in $\Mod_R$. A standard counterexample is given for $R=\Z_p$ by the $R$-module $\Z_p= \limit_n \Z/p^n\Z \in \Mod_R$  by 
\beq  \label{tensinvlim1}
\Q_p = \Z_p \otimes_{\Z_p} \Q_p = (\limit_n \Z/p^n\Z) \otimes_{\Z_p} \Q_p  \neq \limit_n (\Z/p^n\Z \otimes_{\Z_p} \Q_p) =  \limit_n (0) = (0) \;.
\eeq
This does not contradict \eqref{promodcompl2}. In fact,  if   $R=\Z_p$ is viewed as a discrete ring, $\Mod_R$ is a full subcategory of $\cCLMu_R$, but limits  of projective systems in $\Mod_R$ do not coincide with their limits in $\cCLMu_R$. In the present case, 
$$ 
\cCLMu_R\hbox{-}\limit_n \; \Z/p^n\Z = (\Z_p,p\hbox{-}\mbox{adic}) \;,
$$ 
while $\Q_p \in \Mod_R \subset \cCLMu_R$ carries the discrete topology. So, 
\beq \label{tensinvlim2}
(\cCLMu_R\hbox{-}\limit_n \; \Z/p^n\Z)\; \wt_R\; \Q_p = (\Z_p,p\hbox{-}\mbox{adic}) \; \wt_R \; \Q_p^\discr = (0)
\eeq
because  $(\Z_p,p\hbox{-}\mbox{adic}) \; \wt_R\;  \Q_p^\discr$ is the completion of $\Z_p \otimes_R \Q_p = \Q_p$ for the topology with basis of open $R$-submodules $\{p^n \Z_p \otimes_R \Q_p  + (0) \otimes_R \Q_p \}_n = \{\Q_p\}$. Therefore \eqref{tensinvlim2} coincides with the r.h.s. of \eqref{tensinvlim1} which confirms \eqref{promodcompl2}. 
\end{rmk}
 \begin{cor}\label{cor-mon-cat-compl} 
For any $k \in \cCRu$ (resp. $k \in \cCRou$)  the category $\cCLMu_k$ (resp. $\cCLMou_k$), equipped with the tensor product $ \wt^\un_k$, is a (resp. quasi-abelian) symmetric monoidal category with unit $k$.    
 \end{cor}   
\begin{rmk} \label{flatclos0} Let $k$ be any ring, $I$ be an ideal of $k$ and $M$ be any $k$-module. 
Then we have a canonical isomorphism $M/IM \to M\otimes_k k/I$ 
(found by applying the functor $M\otimes_k-$ to the exact sequence $0\to I\to k\to k/I\to 0$). 
The following is a generalization to linearly topologized modules. 
\end{rmk}

\begin{prop}  \label{flatclos1}
Let $M$ be an object of $\cCLMou_k$
and let $I\in\cP(k)$. 
Then we have a canonical isomorphism  
\beq \label{latticeeq} 
M/\ol{IM} \iso M \wt^\un_k (k/I) 
 \;.
 \eeq 
 in  $\cCLMou_{k/I}$. 
 If $k \in \cCRou$ and $M$ is an object of $\cCLMpscan_k$, the previous map is an isomorphism 
 of discrete $k/I$-modules. 
\end{prop}
\begin{proof} For any 
$P \in \cP(M)$ 
we deduce from Remark~\ref{flatclos0} that 
$$
M/P\otimes_k k/I \cong  M/(IM + P) \;.
$$ 
We then reconsider the exact sequence  \eqref{ML1} with $G=M$ and $K=IM$ and obtain from the exact sequence \eqref{ML2}, where $\limit^1 =0$,
  the isomorphism
$$ 
M/\ol{IM} \iso \limit_{P \in \cP(M)} M/(IM + P) 
$$
in $\cCLMou_k$. We conclude that 
$$ M/\ol{IM} \iso \limit_{P \in \cP(M)} M/P\otimes_k k/I \cong M \wt^\un_k (k/I) \;.
$$ 
The last assertion is clear. 
\end{proof}  
 \begin{cor}\label{pscancoco} $\cCLMpscan_k$ is cocomplete so in fact bicomplete. 
\end{cor}
\begin{proof} By Remark~\ref{adjointpscanrmk},  $\cCLMpscan_k$ is complete. 
Let $\{M_\alpha\}_{\alpha \in A}$ be an inductive system in $\cCLMpscan_k$. By  \eqref{indu} and \eqref{latticeeq} we have
$$
\colimit^\un_{\alpha \in A} M_\alpha = \limit_{I \in \cP(k)}  \colimit^\un_{\alpha \in A} (M_\alpha / \ol{I M_\alpha})^\un  
$$
where, since $M_\alpha \in \cCLMpscan_k$, $(M_\alpha / \ol{I M_\alpha})^\un$ is a discrete $k/I$-module and therefore so is 
$\colimit^\un_{\alpha \in A} (M_\alpha / \ol{I M_\alpha})^\un$. Then it follows from $\mathit 3$ of Proposition~\ref{prodiscrete}
that $\colimit^\un_{\alpha \in A} M_\alpha \in  \cCLMpscan_k$. 
\end{proof}
 \begin{cor}\label{cor-mon-cat-compl2} 
If $k \in \cCRouclop$,   the bifunctor $\wt^\un_k$     gives to the bicomplete additive category 
$$\cCLMpscan_k= \cCLMouclop_k = \cCLMoubarrell_k$$ 
a structure of symmetric monoidal category  with unit $k$.   
\end{cor} 
\begin{proof}   This follows from Theorem~\ref{barrpscan3} and 
Proposition~\ref{cor-mon-cat-noncompl} together with the fact that the completion functor sends $\cLMpscan_k$ to $\cCLMpscan_k$.  
\end{proof}
We have  the following  completed version of Corollary~\ref{cloptensor1}
\begin{cor} \label{cloptensor}  Let   
$A,B,C \in \cCRu$, and let 
$\chi : A \to B$ and $\psi : A \to C$ $\clop$-adic morphisms in $\cCRu$. Then 
$
\chi \wt^\un_A \psi: A \longrightarrow B \wt^\un_A C
$
is $\clop$-adic. If moreover $A \in \cCRuclop$ then $B \wt^\un_A C \in \cCRuclop$.
\end{cor}
\begin{proof} Let $I$ be an open ideal of $A$. We want to show that the closure of $I(B \wt^\un_A C)$ is open. In fact it contains both $\ol{IB} \wt^\un_A C$ and $B \wt^\un_A \ol{IC}$.
\end{proof} 
 
The following proposition is due to Gabber and Ramero \cite[Lemma 15.1.27]{GR}.
\begin{prop} \label{exactprofl} 
Let $k \in \cCRou$ and  $M \in \cCLMou_k$. Then:
\ben
\item
The functor 
\beq \label{exactprofleq} 
\cCLMou_k \longrightarrow \cCLMou_k \;,\quad X \longmapsto X \wt^\un_k M
\eeq
 is strongly right exact.  
\item 
If  
$M \in \cCLMpscan_k$ and  is pro-flat, then 
the previous functor is  exact.  
\een
\end{prop} 
\begin{rmk}  \label{exactproflrmk}   Under the  assumptions of point $\mathit 2$ of the proposition, the functor \eqref{exactprofleq} is  exact  and strongly right exact, but not 
necessarily \emph{strongly left exact} (see Definition~\ref{leftrightex}). The consequence is that  \eqref{exactprofleq} preserves images (that is kernels of strict epimorphisms) but does not preserve kernels of general morphisms. 
\end{rmk} 
 
\begin{cor}\label{tenscan} Let $k$ be in $\cCRou$. 
If $M$, $N$ are canonical $k$-modules then 
 $M \wt_k^\un N$
 is canonical. \end{cor}
\begin{proof}  Let $\phi:k^{(A,\un)} \to M$ and $\psi:k^{(B,\un)} \to N$ be strict epimorphisms. Then 
$k^{(A,\un)} \wt^\un_k k^{(B,\un)} = k^{(A \times B,\un)}$, and the morphism $\phi  \wt^\un_k \psi : k^{(A \times B,\un)} \to 
M \wt_k^\un N$ is a strict epimorphism, as well. In fact, $\phi  \wt^\un_k \psi$ decomposes into the product
$$ k^{(A,\un)} \wt^\un_k k^{(B,\un)} \map{\id_{k^{(A,\un)}} \wt^\un_k \psi} k^{(A,\un)} \wt^\un_k N 
\map{\phi  \wt^\un_k  \id_N}
M \wt_k^\un N
$$
The functor $\cCLMou_k \longrightarrow \cCLMou_k$,  $X \longmapsto k^{(A,\un)} \wt^\un_k X$ (resp. $Y \longmapsto Y \wt^\un_k  N$)  is  right exact, so that both  $\id_{k^{(A,\un)}} \wt^\un_k \psi$ and $\phi  \wt^\un_k  \id_N$ are strict epimorphisms, and therefore so is their composition $\phi  \wt^\un_k \psi$ \cite[\S 1.1.3]{schneiders}.   \end{proof} 
 \begin{cor}\label{cor-mon-cat-compl3} 
If $k \in \cCRou$ the quasi-abelian category  $\cCLMcan_k$ equipped with 
 the bifunctor   $\wt^\un_k$    is 
a symmetric monoidal category with unit $k$. 
\end{cor}
\begin{proof}
Follows from Corollaries~\ref{cor-mon-cat-compl2} and \ref{tenscan}.
\end{proof} 
 \begin{rmk} \label{cantens} The functor $(-)^\can : \cCLMu_k \longrightarrow \cCLMcan_k$ does not commute in general  with $\wt^\un_k$. For example, for $k$   discrete $\cCLMcan_k = \Mod_k$  and $(-)^\can = (-)^\for$. For the power-series topology we have in $\cCLMu_k$
  $$
  k[[x]] \wt^\un_k k[[y]] \iso k[[x,y]]
  $$
while 
    $$
  k[[x]]^\discr \otimes_k k[[y]]^\discr \subsetneq k[[x,y]]^\discr \;.
  $$
  Similarly, if $\Z_p$ has the $p$-adic topology, and  $\Z_p[[x]] = \Z_p[[y]]$,  $\Z_p[[x,y]]$ are equipped with 
  their maximal-adic topologies,  we have $\Z_p[[x]] \wt^\un_{\Z_p} \Z_p[[y]] =  \Z_p[[x,y]]$
  in $\cCLMu_{\Z_p}$, while in $\cCLMcan_{\Z_p}$ (\ie for the $p$-adic topologies)  $\Z_p[[x]] \wt^\un_{\Z_p} \Z_p[[y]] \subsetneq  \Z_p[[x,y]]$. 
  \end{rmk}
\begin{cor} \label{topflat}  
Let $k \in \cCRou$. 
Then any pro-flat object of $\cCLMpscan_k$  (resp. of  $\cCLMcan_k$) is  $\wt^\un_k$-flat. 
\end{cor}
\begin{proof}
Follows from Proposition~\ref{exactprofl} and Corollary~\ref{cor-mon-cat-compl2} (resp. \ref{cor-mon-cat-compl3}).
\end{proof}
\begin{prop} \label{basechange}  
Let $k \in \cCRou$ and let $R$ be a ring object of $\cCLMou_k$.
  Then, for any $M$ in $\cCLMpscan_k$  (resp. in $\cCLMcan_k$), $M \wt^\un_k R$ is an object of $\cCLMpscan_R$  (resp. of $\cCLMcan_R$).  \end{prop}
\begin{proof}   The fact that, if $M \in \cCLMpscan_k$, $M \wt^\un_k R \in \cCLMpscan_R$ is clear. Let $M \in \cCLMcan_k$ and 
let $\varphi: k^{(A,\un)} \longrightarrow M$  be 
a strict epimorphism. Then by application of the functor $(-)\wt^\un_k R$ we get a 
strict epimorphism $\varphi\wt^\un_k R : k^{(A,\un)} \wt^\un_k R \longrightarrow M\wt^\un_k R$. 
On the other hand, by Corollary~\ref{cor-mon-cat-compl},
$$
 k^{(A,\un)} \wt^\un_k R = R^{(A,\un)} \;,
$$
so that $M\wt^\un_k R$ is in fact $R$-canonical. 
\end{proof}  
\begin{cor} \label{clopbasechange}  
Let $k \in \cCRouclop$ and let $k \longrightarrow R$ be a $\clop$-adic morphism of 
$\cCRu$ (so that, in particular, $R \in \cCRouclop$).  
Then, for any $M \in \cCLMpscan_k =\cCLMouclop_k = \cCLMoubarrell_k$, $M \wt^\un_k R \in  \cCLMpscan_R =\cCLMouclop_R = \cCLMoubarrell_R$.
\end{cor}
\end{subsection} 
\end{section}  
\begin{section}{Internal Homs} \label{inthoms}
 As in the previous section, $k$ is here any object of $\cRu$. 
 More requirements on $k$ will be specified as needed.   
\begin{subsection}{Uniform convergence}

 \begin{defn} \label{boundedhomgen}  
For any $M,N$ in $\cLMu_k$, we denote by   $\Lin^\ba_k$  (resp. $\Lin^\se_k$)  the $k$-module $\Hom_k(M^\for,N^\for)$
of $k$-linear maps $M^\for \to N^\for$, equipped with the $k$-linear topology for which a fundamental 
system of open $k$-submodules
is the family  
$$
W(B,Q) = \{f \in \Hom_k(M^\for,N^\for) \,|\, f(B) \subset Q\,\}   \; ,
$$
for  $Q\in \cP_k(N)$ and $B = M$ (resp. $B=$ a finite subset of $M$). 
This topology   on the $k$-module $\Hom_k(M^\for,N^\for)$ will be called the topology of  
\emph{uniform } (resp. \emph{simple})  \emph{convergence on $M$} or  the
\emph{strong} (resp. \emph{weak})  \emph{topology}.
\end{defn}
For $\ast \in \{\ba , \se\}$, the topological $k$-module $\Lin^\ast_k(M,N)$   is an object of $\cLMu_k$. 
\begin{rmk} \label{boundedhomgen1} 
If $N$ is complete then  
$\Lin^\ast_k(M,N)$  is  complete, as well.  
This is because if a net $\alpha \mapsto f_\alpha$, $\alpha \in A$, of elements of $\Hom_k(M,N)$ 
is Cauchy for the topology of  simple convergence, then, 
 for any $x \in M$, the net $\alpha \mapsto f_\alpha(x)$ converges in $N$ to a well-defined  element $f(x)$. 
 Now, for  fixed 
 $x,y \in M$ and $\lambda,\mu \in k$, the
 nets 
 $$
 \alpha \mapsto \lambda  f_\alpha(x)\;\;,\;\; \alpha \mapsto \mu f_\alpha(y)\;\;,\;\; \alpha \mapsto f_\alpha(\lambda x+ \mu y)
 $$ 
 all converge and the identity 
 $$
 f_\alpha (\lambda x+ \mu y) = \lambda  f_\alpha(x) + \mu  f_\alpha(y)\;\;,\; \forall \,\alpha \in A 
 $$ 
 implies that $f \in \Hom_k(M,N)$. 
\end{rmk} 
There is a natural continuous bijection 
\beq \label{simplevsbddgen}
\Lin^\ba_k(M,N) \longrightarrow \Lin^\se_k(M,N) \;.
\eeq

\begin{defn} \label{boundedhom2}  
For   $M,N$ in $\cLMu_k$ and $\ast \in \{\ba , \se\}$,  we denote by ${\cL}^\ast_k(M,N)$   the
  topological $k$-module
$\Hom_{\cLMu_k}(M,N)$, equipped with the subspace topology of  $\Lin^\ast_k(M,N)$. 
If $N$ is in $\cCLMu_k$, 
we define the object $\what{\cL}^\ast_k(M,N)$ of $\cCLMu_k$ as the 
closure of ${\cL}^\ast_k(M,N)$ in $\Lin^\ast_k(M,N)$, equipped with the subspace topology.  
\end{defn}
\begin{rmk} \label{contmap} 
By general results on complete uniform spaces \cite[Chap. II, \S3, N.9, Cor. 1]{topgen}, 
the object $\what{\cL}^\ast_k(M,N)$  of $\cCLMu_k$ is the completion of 
${\cL}^\ast_k(M,N)$.  
The bijective morphism \eqref{simplevsbddgen} induces a continuous injection 
\beq \label{simplevsbddcatgen}
\what{\cL}^\ba_k(M,N) \longrightarrow \what{\cL}^\se_k(M,N)  \;.
\eeq
\end{rmk}
In general, 
the $k$-linear embedding 
\beq \label{embed} {\cL}^\ast_k(M,N) \longrightarrow \what{\cL}^\ast_k(M,N)
\eeq
is not surjective, \ie 
${\cL}^\ast_k(M,N)$ is not complete. 
However, we have (\cf    \cite[Prop. 7.16]{schneider}):
\begin{thm} \label{borncomplete}  
Let  $M,N \in \cLMu_k$ with $N$ complete. 
Then $ {\cL}^\ba_k(M,N)$  is complete 
\ie \eqref{embed} induces an isomorphism 
\beq \label{borncomplete1}  
{\cL}^\ba_k(M,N) \iso  \what{\cL}^\ba_k(M,N) \in \cCLMu_k\;. 
\eeq
If, in particular,  $N \in \cCLMou_k$, ${\cL}^\ba_k(M,N)  \in \cCLMou_k$.
\end{thm}
\begin{proof}  Let  $\alpha \mapsto \phi_\alpha$, for $\alpha$ in the filtered set $(A,\leq)$, be a net in ${\cL}^\ba_k(M,N)$
converging to $\phi \in \Lin^\ba_k(M,N)$.
We want to show that $\phi$ is in fact continuous. It suffices to show that, for any $P \in \cP_k(N)$, the $k$-submodule $\phi^{-1}(P)$ is open in $M$. There is an index $\alpha_P \in A$ such that, if $\alpha \geq \alpha_P$, 
$\phi$ and $\phi_\alpha$ induce the same $k$-linear map $M \longrightarrow N/P$. So, $\phi^{-1}(P) = \phi_\alpha^{-1}(P)$ is open because $\phi_\alpha$ is continuous. 
\end{proof} 
\begin{rmk}\label{noncomplete} \hfill \ben
\item We conclude from Theorem~\ref{borncomplete} and Definition~\ref{boundedhom2} that,  for  any   $M,N \in \cCLMu_k$, there is a  natural  isomorphism
\beq \label{unifcomplform2} \Hom_{\cCLMu_k}(M,N) \iso \cL^\ba_k(M,N)^\for  
\eeq
in $\Mod_k$. 
The topology induced on $\Hom_{\cCLMu_k}(M,N)$ by the previous isomorphism  is usually called the 
\emph{strong topology} of $\Hom_{\cCLMu_k}(M,N)$. 
So, $\Hom_{\cCLMu_k}(M,N)$ is separated and complete in its strong topology.  The strong topology of $\Hom_{\cCLMu_k}(M,N)$ coincides with the  topology of uniform convergence on $M$. 
\item In contrast to Theorem~\ref{borncomplete}, if $N$ is complete, it is not necessarily the case that  $\cL^\se_k(M,N)$ is complete. See however Theorem~\ref{barrelquasicomplete}  below. 
\een
\end{rmk}
\begin{prop}\label{inter} Let $k \in \cCRu$ and let $M,N \in \cCLMu_k$. Then we have the formula
\beq \label{stronghom} 
 \cL^\ba_k(M,N) =  \limit_{Q \in \cP(N)}  \colimit^\un_{P \in \cP(M)} \Hom_k(M/P,N/Q)^\dis    
\eeq 
in $\cCLMu_k$. 
\end{prop}
\begin{proof}
For $P \in \cP(M)$ and $Q \in \cP(N)$, $\cL^\ba_k(M/P,N/Q)$ is simply $\Hom_k(M/P,N/Q)$, equipped with the discrete topology. 
Moreover,
$$
\cL^\ba_k(M,N/Q) =  \Hom_k(M,N/Q)^\dis =\colimit^\un_{P \in \cP(M)}\Hom_k(M/P,N/Q)^\dis  \;.
$$
Finally, 
$$\cL^\ba_k(M,N) = \limit_{Q \in \cP(N)} \cL^\ba_k(M,N/Q)\;.$$
\end{proof}
\begin{rmk} \label{interset} It follows from \eqref{stronghom}, taking into account that 
$$\colimit^\un_{P \in \cP(M)}\Hom_k(M/P,N/Q)^\dis = (\colimit_{P \in \cP(M)}\Hom_k(M/P,N/Q))^\dis \;,
$$
that
\beq \label{sethom} 
 \cL^\ba_k(M,N)^\for =  \limit_{Q \in \cP(N)}  \colimit_{P \in \cP(M)} \Hom_k(M/P,N/Q)^\dis    
\eeq 
\end{rmk}  
\begin{defn} \label{cancomplete2}  Let $k \in \cCRou$ and let $M,N \in \cCLMcan_k$. We set
 $\cL^\can_k(M,N) := (\cL^\ba_k(M,N))^\can$. 
\end{defn}
\begin{cor}\label{intercan} Let $k \in \cCRou$ and let $M,N \in \cCLMcan_k$. Then we have the formula
\beq \label{canhom} 
 \cL^\can_k(M,N) =  \limit^\can_{Q \in \cP(N)}  \colimit^\un_{P \in \cP(M)} \Hom_k(M/P,N/Q)^\dis    
\eeq 
in $\cCLMu_k$. Again, 
$$ \cL^\can_k(M,N)^\for =  \cL^\ba_k(M,N)^\for$$
has the expression \eqref{stronghom}.
\end{cor}
\begin{rmk} \label{leftexhom} Let $k \in \cCRou$. For any $M \in \cCLMou_k$, 
the functor 
\beq \label{exactHom} 
\cCLMou_k \longrightarrow \cCLMou_k \;,\quad X \longmapsto \cL^\ba_k(M,X)
\eeq
commutes with  countable limits; in particular, it is strongly left  exact. 
Similarly,    for  any $M \in \cCLMcan_k$, 
 the functor 
\beq \label{exactHomcan} 
\cCLMcan_k \longrightarrow \cCLMcan_k \;,\quad X \longmapsto \cL^\can_k(M,X)
\eeq
 is strongly left  exact and  commutes with all (small) limits.  
 \end{rmk}
 \end{subsection}
\begin{subsection}{Equicontinuity}  
\begin{defn} \label{equicont} 
For any $M,N$ in $\LMu_k$, we say that a subset $H \subset \Hom_{\LMu_k}(M,N)$ is \emph{equicontinuous} if, 
for any $Q \in \cP_k(N)$, there exists 
$P \in \cP_k(M)$ such that $f(P) \subset Q$, for any $f \in H$.  
\end{defn}

\begin{lemma} \label{boundedint}  
Let $M,N$ be objects of $\LMu_k$ and let $H \subset \Hom_{\LMu_k}(M,N)$ be a  $k$-submodule.
Then, 
for   any $Q \in \cP_k(N)$, there exists $J \in \cP(k)$ such that 
\beq \label{bddequ}
J M \subset \bigcap_{u \in H} u^{-1}(Q) \; .
\eeq 
\end{lemma}
\begin{proof} Any $J \in \cP(k)$ such that $JN \subset Q$ will do. 
\end{proof}

\begin{prop} \label{equivequicont} 
Let $M,N \in \LMu_k$.
Consider the following assertions  for a $k$-submodule $H \subset \Hom_{\LMu_k}(M,N)$. \hfill
\ben
\item $H$ is equicontinuous;
\item For any $Q \in \cP_k(N)$, $\bigcap_{u \in H}u^{-1}(Q) \in \cP_k(M)$. 
\een
Then $1 \Leftrightarrow 2$. 
\endgraf  
 If $M$ is pseudocanonical (in particular, if $M$ is barrelled) the previous assertions   hold for $H = \Hom_{\LMu_k}(M,N)$. 
\end{prop}
\begin{proof}  
$1 \Leftrightarrow 2$ is clear.   
To prove the last assertion, let $H = \Hom_{\LMu_k}(M,N)$ and let $Q \in \cP_k(N)$. By Lemma~\ref{boundedint}
there is $J \in \cP(k)$ such that \eqref{bddequ}   
holds.   
By the condition on $M$, we conclude that 
the closed  $k$-submodule
$\bigcap_{u \in H} u^{-1}(Q)$ is  open, that is $\mathit 2$ in the statement.  
\end{proof} 

 \begin{lemma}\label{nfa_6.10} 
 Let   $M,N$ be objects of $\cLMu_k$ with $N$ separated. 
 Then,  for any equicontinuous subset 
$H \subset \Hom_{\cLMu_k}(M,N)$, the closure $\ol{H}^\se$ (resp. $\ol{H}^\ba$) of $H$ in $\Lin^\se_k(M,N)$ (resp.  in $\Lin^\ba_k(M,N)$) is equicontinuous. In particular, $\ol{H}^\se$ (resp. $\ol{H}^\ba$) is 	
contained in   $\Hom_{\cLMu_k}(M,N)$  hence $\ol{H}^\se$ (resp. $\ol{H}^\ba$)  is the closure of $H$ in $\cL_k^\se(M,N)$ (resp. in $\cL_k^\ba(M,N)$).
 \end{lemma}
 \begin{proof}  
 The assertion for the strong topology follows from the one for the weak topology. We then  prove that the closure $\ol{H}^\se$ of $H$ in the weak topology is equicontinuous. Let $Q \in \cP_k(N)$ and let $P \in \cP_k(M)$ be such that $f(P) \subset Q$ for any $f \in H$. We claim that this also holds for $f \in \ol{H}^\se$. 
 In fact, let  
  $\alpha \mapsto \phi_\alpha$, for $\alpha$ in the filtered set $(A,\leq)$, be a net in $H$ converging in $\Lin^\se_k(M,N)$ to $\phi \in \ol{H}^\se$.  
We want to show that $\phi(P) \subset Q$. In fact, for any $v \in P$, there is an index $\alpha_Q \in A$ such that, if $\alpha \geq \alpha_Q$, 
$\phi(v) - \phi_\alpha(v) \in Q$. Then $\phi(v)  \in Q$, hence $\phi(P) \subset Q$.  \end{proof}

 \begin{lemma}\label{pf_of_7.13} Let $M, N$ be objects of $\cLMu_k$. Then any subset $H \subset \Hom_{\cLMu_k} (M,N)$  which is complete for the weak topology on $\Hom_{\cLMu_k} (M,N)$, is complete for the strong topology, as well. 
\end{lemma}
\begin{proof} 
Let $(f_i)_{i \in I}$ be a  Cauchy net in $H$ with respect to the strong topology. By assumption $(f_i)_{i \in I}$ converges for the weak topology to some $f \in H$. We show that $(f_i)_{i \in I}$ converges to $f$ in the strong topology, as well. For any $P \in \cP_k(N)$ there is an $i_P \in I$ such that $f_j - f_h \in \Hom_{\cLMu_k} (M,P)$ for any any $j,h \geq i_P$. So, let us fix $P$ and $h \geq i_P$; then $f_j  \in  f_h + \Hom_{\cLMu_k} (M,P)$ for any any $j \geq i_P$  so that, taking limits for $j \in I_{\geq i_P}$ in the weak topology, we deduce that 
$$
f \in f_h + \Hom_{\cLMu_k} (M,P)\;\;,\;\;\forall \; h \geq i_P\;.
$$
The conclusion follows. 
\end{proof}
The proof of \cite[Prop. 7.13]{schneider} easily generalizes to show the following
 \begin{prop} \label{stabquasicomplete} 
Let $M,N$ be objects of $\cLMu_k$ with  $N$ separated and complete.  For $\ast \in \{\se,\ba\}$ any equicontinuous closed subset 
 $H \subset \cL_k^\ast(M,N)$ is complete. 
 \end{prop}
\begin{proof} Assume first $\ast = \se$.
Lemma~\ref{nfa_6.10} shows that the closure $\ol{H}^\se$ of $H$   in $\Lin^\se_k (M,N)$ is equicontinuous, so that $\ol{H}^\se$ is the closure of $H$   in $\cL^\se_k (M,N)$, hence coincides with $H$. 
So, $H$ is closed in $\Lin^\se_k (M,N)$.  
It then suffices to observe that, by Remark~\ref{boundedhomgen1}, 
$\Lin^\se_k (M,N)$ is  complete. 
\par Let now  $\ast = \ba$, so that $H$ is closed in $\cL_k^\ba(M,N)$. Let $\ol{H}^\se$ be the closure of $H$ 
in $\Lin^\se_k (M,N)$. Lemma~\ref{nfa_6.10} shows that  $\ol{H}^\se \subset \Hom_{\cLMu_k} (M,N)$ is equicontinuous and it is closed in $\cL_k^\se (M,N)$.  By the previous case, $\ol{H}^\se$ is a complete subset of $\cL_k^\se (M,N)$. 
By Lemma~\ref{pf_of_7.13} any subset of $\Hom_{\cLMu_k} (M,N)$  which is complete for the weak topology, is complete for the strong topology, as well. So, $H$, which is closed in $\ol{H}^\se$ for the strong topology,  is complete for the strong topology.
\end{proof}
We also have (\cf    \cite[Cor. 7.14]{schneider}):
\begin{prop} \label{barrelquasicomplete}  
Let $k \in \cCRu$ and let $M,N$ be objects of $\cLMu_k$ with $N$  complete. If $M$  pseudocanonical (in particular, if $M$ is barrelled), then   ${\cL}^\se_k(M,N)$  is  
complete.
\end{prop}
\begin{proof} 
By Proposition~\ref{equivequicont}, ${\cL}^\se_k(M,N)$ is equicontinuous. By 
Proposition~\ref{stabquasicomplete}, ${\cL}^\se_k(M,N)$ is complete.
\end{proof}
\begin{rmk} By Theorem~\ref{borncomplete},  ${\cL}^\ba_k(M,N)$  is  complete as soon as $N$ is. 
\end{rmk}
\end{subsection}
\begin{subsection}{Adjunctions and closedness} \label{adjunct}
\begin{lemma} \label{genadj0} Let  $M_1,M_2,N \in \cLMu_k$. There are
 canonical $k$-linear isomorphisms
\beq  
\label{adjgen} \begin{split}
 &\Bil^\un_k(M_1 \times M_2,N)  \iso  \\
&\left\{ \Phi:M_1 \to \cL^\ba_k(M_2,N) : ~
 \parbox[c][3em][c]{0.38\textwidth}{$\Phi$   is $k$-linear, continuous,  and \\ $\Phi(M_1)$  is equicontinuous}  ~ \right\}
 \iso  \\
&\left\{ \Psi:M_2 \to \cL^\ba_k(M_1,N) : ~
 \parbox[c][3em][c]{0.38\textwidth}{$\Psi$   is $k$-linear, continuous,  and \\ $\Psi(M_2)$  is equicontinuous}  ~ \right\}\;.
\end{split}
\eeq   
\end{lemma}
\begin{proof}
Let $\varphi:  M_1 \times M_2 \longrightarrow N$ be an element of   
$\Bil^\un_k(M_1 \times M_2,N)$. Let $\Phi$ be associated to $\varphi$ as in Notation~\ref{linbilin}.
For any $Q \in \cP_k(N)$  
there exist $P_i \in \cP_k(M_i)$, for $i=1,2$, such that 
$$
\varphi (M_1 \times P_2) + \varphi (P_1 \times M_2) \subset Q  \;.
$$ 
For any $m_1 \in M_1$,  $\varphi (\{m_1\} \times P_2) \subset Q$ implies that   $\Phi(m_1) \in \Hom_{\cLMu_k}(M_2,N)$. Since for any fixed $Q \in \cP_k(N)$ 
there is $P_1 \in \cP_k(M_1)$ such that  $\varphi (P_1 \times M_2) \subset Q$ 
that is 
$\Phi(P_1) \subset W(M_2,Q)$, 
 we see that 
$\Phi: M_1 \longrightarrow \cL^\ba_k (M_2,N)$   
is continuous. Finally,  
for any   $Q \in \cP_k(N)$, the existence of $P_2 \in \cP_k(M_2)$ such that $\varphi(M_1 \times P_2) \subset Q$ shows that $\Phi (M_1) \subset \Hom_{\cLMu_k}(M_2,N)$ is equicontinuous. 
The argument can obviously be reversed, to show that  the $k$-bilinear map $\varphi$ associated   to 
$\Phi \in \Hom_{\cLMu_k}(M_1,\cL^\ba(M_2,N))$, under the \rhs assumptions on $\Phi$,    by the rule 
$\varphi(m_1,m_2) = \Phi(m_1)(m_2)$, for $m_i \in M_i$, $i=1,2$,  is in $\Bil^\un_k(M_1 \times M_2,N)$. 
\end{proof}
\begin{cor} \label{adjgen1} Let $k \in \cCRou$, 
$M_1,M_2 \in \cCLMpscan_k$, $N \in \cCLMu_k$.
Then  \eqref{adjgen} becomes 
\beq  
\label{adjgen2}  \begin{split}
  \Bil^\un_k(M_1 \times M_2,N)  \iso \Hom_{\cLMu_k} (M_1 \wt^\un_k M_2,N)  \iso
  \\
   \Hom_{\cLMu_k}(M_1, \cL^\ba_k(M_2,N))  
  \iso  \Hom_{\cLMu_k}(M_2, \cL^\ba_k(M_1,N)) \;.
 \end{split} \eeq  
\end{cor}
\begin{proof} Since $M_1,M_2 \in \cCLMpscan_k$,  Proposition~\ref{equivequicont} applies to prove that 
$\Hom_{\cLMu_k}(M_i,N)$ are equicontinuous for $i=1,2$. 
\end{proof} 
\begin{cor} \label{adjcan1} Let $k \in \cCRou$ and $M_1,M_2,N \in \cCLMcan_k$. Then 
$$
\Hom_{\cCLMcan_k} (M_1 \wt^\un_k M_2,N)  \iso   \Hom_{\cCLMcan_k} (M_1, \cL^\can_k(M_2,N)) 
   \iso  \Hom_{\cCLMcan_k}(M_2, \cL^\can_k(M_1,N)) \;.
 $$
\end{cor}
\begin{proof} Follows from \eqref{adjgen2} by right-adjointness of $(-)^\can$.
\end{proof} 
\begin{cor} \label{adjcan2} Let $k \in \cCRou$ and $M,N \in \cCLMcan_k$.   Then
\beq  \label{adjcan31}
\Hom_{\cCLMcan_k}  (\cL^\can_k (M,N) \wt^\un_k M, N)  \iso \Hom_{\cCLMcan_k} (\cL^\can_k (M,N), \cL^\can_k (M,N)) \;.
\eeq
Therefore the identity morphism $\cL^\can_k (M,N) \iso \cL^\can_k (M,N)$ determines a morphism
\beq  \label{adjcan32}
{\rm ev}_{M,N}: \cL^\can_k (M,N)  \wt^\un_k M \map{} N
\eeq
called \emph{evaluation}. In particular, for any $M \in  \cCLMcan_k$, we obtain a  pairing 
\beq  \label{adjcan33}
{\rm ev}_M:   \cL^\can_k (M,k) \; \wt^\un_k \;  M \map{} k \;.
\eeq
\end{cor}
\begin{proof} 
We take $M_1 =  \cL^\can_k (M,N)$ and $M_2 = M$ in Corollary~\ref{adjcan1} to get 
\eqref{adjcan31}. The rest follows. 
\end{proof} 
\begin{thm} \label{cor-closed-scan} Let $k$ be in $\cCRou$. \ben
\item The category $\cCLMcan_k = (\cCLMcan_k,  \wt^\un_k, k)$  is a quasi-abelian complete  symmetric monoidal subcategory of $\cCLMou_k$. It is moreover closed with internal Hom $\cL^\can_k$ and has enough projectives. 
\item Any projective object of $\cCLMcan_k$ is $\wt^\un_k$-flat.  
\item If $k \in \cCRoufop$ then $\cCLMcan_k$ is bicomplete and coincides with the full subcategory of $\cCLMou_k$
of the objects that are complete in their naive canonical topology.
\een
\end{thm}
\begin{proof} $\mathit 1.$ This part follows from  Theorem~\ref{quasiabelian-scan},   Corollary~\ref{cor-mon-cat-compl3}, Corollary~\ref{adjcan1}  and Remark~\ref{existproj}. 
\par $\mathit 2.$ The statement is clear for a projective object of $\cCLMcan_k$  of the form $k^{(A,\un)}$ because then for any $M \in \cCLMcan_k$ one has 
$$
k^{(A,\un)} \wt^\un_k M = M^{(A,\un)} \;.
$$
Since projectives are direct summands of modules of the previous form, the statement follows in full 
generality. 
\par $\mathit 3.$ This follows from Corollary~\ref{colimcan1}
and part $\mathit 1$ of Corollary~\ref{naivecancor}.
\end{proof}
\begin{rmk} \label{notclosed}\hfill  \ben
\item
The category $\cCLMcan_k$ has exact products \cite[Prop. 1.4.5]{schneiders}.
\item From the fact that $\cCLMcan_k$ is closed, it follows formally that $\wt^\un_k$ commutes with colimits in $\cCLMcan_k$. 
\een
\end{rmk}
\begin{defn}  Let $k \in \cCRou$. A (possibly non-commutative) ring-object of the closed symmetric monoidal  category  $\cCLMcan_k$ will be called a 
\emph{canonical $k$-algebra}. We will  denote by  $\cR^\can_k$ the full subcategory of $\cRu$ consisting of canonical $k$-algebras. 
\end{defn}
\begin{prop} For $k \in \cCRouclop$ (resp. $\in \cCRoufop$) a commutative canonical $k$-algebra is an object of $\cCRouclop$ (resp. of $\cCRoufop$) and its structural morphism is $\clop$-adic (resp. $\op$-adic).
\end{prop}
\begin{proof} For $k \in \cCRouclop$ (resp. $\in \cCRoufop$) any object of $\cCLMcan_k$ is in $\cCRouclop$ (resp. $\cCRoufop$).
\end{proof}
\end{subsection} 

\end{section}


\end{document}